\documentclass[11pt]{article}     
\usepackage{graphicx}
\usepackage{amssymb}
\usepackage{amsthm}
\usepackage[T1]{fontenc}
\usepackage[utf8]{inputenc}
\usepackage{newtxtext,newtxmath}
\usepackage{hyperref}
    \hypersetup{
        colorlinks,
        linkcolor=blue,
        citecolor=blue,
        urlcolor=blue,
        final
    }
\usepackage{textcomp} 
\usepackage{doi}
\usepackage{amsmath,amssymb,mathtools}
\usepackage{float}
\usepackage{dsfont}
\usepackage{siunitx,booktabs}
\usepackage{subcaption}
\usepackage[square,sort,comma,numbers]{natbib}

\newtheorem{theorem}{Theorem}[section]

\newtheorem{proposition}[theorem]{Proposition}

\newtheorem{definition}{Definition}[section]

\newtheorem{remark}{Remark}

\usepackage[top=1in, bottom=1in, left=1in, right=1in]{geometry}
\title{Efficient Numerical Algorithms based on Difference Potentials for Chemotaxis Systems in 3D}

\author{Yekaterina Epshteyn \thanks{Department of Mathematics, The University of Utah, 155 S 1400 E Rm. 233, Salt Lake City, UT, 84112, USA, epshteyn@math.utah.edu} \and Qing Xia \thanks{Department of Mathematics, The University of Utah, 155 S 1400 E Rm. 233, Salt Lake City, UT, 84112, USA, xia@math.utah.edu}}

\providecommand{\keywords}[1]{\noindent\textbf{Keywords} #1}

\providecommand{\subclass}[1]{\noindent\textbf{AMS Subject Classification} #1}

\begin{document}

\maketitle

\begin{abstract}
In this work, we propose efficient and accurate numerical algorithms based on Difference Potentials Method for numerical solution of chemotaxis systems and related models in 3D. The developed algorithms handle 3D irregular geometry with the use of only Cartesian meshes and employ Fast Poisson Solvers.  In addition, to further enhance computational efficiency of the methods, we design a Difference-Potentials-based domain decomposition approach which allows mesh adaptivity and easy parallelization of the algorithm in space. Extensive numerical experiments are presented to illustrate the accuracy, efficiency and robustness of the developed numerical algorithms.
\end{abstract}


\subclass{65M06, 65M70, 65M99, 92C17, 35K57}


\section{Introduction} \label{sec:intro}
In this work we develop efficient and accurate numerical algorithms based on Difference Potentials Method (DPM) for the Patlak-Keller-Segel (PKS) chemotaxis model and related problems in 3D. The proposed methods handle irregular geometry with the use of only Cartesian meshes, employ Fast Poisson Solvers, allow easy parallelization and adaptivity in space.

Chemotaxis refers to mechanisms by which cellular motion occurs in response to an external stimulus, for instance a chemical one. Chemotaxis is an essential process in many medical and biological applications, including cell aggregation and pattern formation mechanisms and tumor growth. Modeling of chemotaxis dates back to the pioneering work by Patlak, Keller and Segel \cite{keller_1970,keller_1971,Patlak_1953}. The PKS chemotaxis model consists of coupled convection-diffusion and reaction-diffusion partial differential equations:
\begin{equation}\label{eqn:minimal_chemotaxis}
\left\{\begin{aligned}
&\rho_t+\nabla\!\cdot\!\left(\chi\rho\nabla c\right)=\Delta\rho,\\
&\alpha c_t=\Delta c-\gamma_cc+\gamma_\rho\rho,\end{aligned}\right.
\quad (x,y,z)\in\Omega\subset\mathbb{R}^3,~~t>0,
\end{equation}
subject to zero Neumann boundary conditions and the initial conditions on the cell density $\rho(x,y, z, t)$ and on the chemoattractant concentration $c(x,y,z,t)$. In the system (\ref{eqn:minimal_chemotaxis}), the coefficient $\chi$ is a chemotactic sensitivity constant, $\gamma_{\rho}$ and $\gamma_c$ are the reaction coefficients. The parameter $\alpha$ is equal to either 1 or 0, which corresponds to the ``parabolic-parabolic'' or reduced ``parabolic-elliptic'' coupling, respectively. In this work, we will focus on the PKS chemotaxis model \eqref{eqn:minimal_chemotaxis} with $\alpha=\gamma_\rho=\gamma_c=1$, but the developed methods are not restricted by this assumption and can be easily extended to more general chemotaxis systems and related models.

In the past several years, chemotaxis models have been extensively studied (see, e.g., \cite{HV,Hillen_2008,Hor03,Hor04,Per} and references below). A known property of many chemotaxis systems is their ability to represent a concentration phenomenon that is mathematically described by fast growth of solutions in small neighborhoods of concentration points/curves. The solutions may blow up or may exhibit a very singular behavior. This blow-up represents a mathematical description of a cell concentration phenomenon that occurs in real biological systems, see, e.g., \cite{Adl,Bon,BB91,BB95,CP,CR,Nan,PHK}. 

Capturing blowing up or spiky solutions is a challenging task numerically, but at the same time design of robust numerical algorithms is crucial for the modeling and analysis of chemotaxis mechanisms. Let us briefly review some of the recent numerical methods that have been proposed for the ``parabolic-parabolic'' coupling of chemotaxis models in the literature.  High-order discontinuous Galerkin methods have been proposed in \cite{MR2511733,Epshteyn_2008_discontinuous} and \cite{Li2017} for chemotaxis models in 2D rectangular domains. A flux corrected finite element method is designed in 2D in \cite{Strehl_2010}, and is extended to chemotaxis models on stationary surface domains and cylindrical domains in \cite{Sokolov_2013,Strehl_2013}. A finite-volume based fractional step numerical method is proposed in \cite{Tyson_2000} for models in 2D domains. However, the operator splitting approach may not be applicable when the convective part of the chemotaxis system is not hyperbolic. A simpler and more efficient second order positivity-preserving finite-volume central-upwind scheme is developed for 2D rectangular domains in \cite{Chertock_2008}, and is extended to fourth-order accurate numerical method in \cite{Chertock_2017}. A novel numerical method based on symmetric reformulation of the 2D PKS chemotaxis system has been developed and analyzed in \cite{MR3766384}. The proposed method is both positivity-preserving and asymptotic preserving. A random particle blob method for PKS chemotaxis model has been proposed and analyzed in a series of papers \cite{MR3584546,MR3717909}. A new hybrid-variational approach which is based on the generalization of the implicit Wasserstein scheme \cite{ISI:000071938700001} has been proposed and analyzed for 2D PKS chemotaxis system in \cite{MR3423264}. Finally, an upwind Difference Potentials Method is proposed in \cite{Epsh1} to approximate chemotaxis models in 2D irregular domains, using uniform Cartesian meshes and Fast Poisson Solvers. Note that among the methods that have been proposed, only \cite{Strehl_2013} is designed to handle chemotaxis models in 3D irregular domains by the use of unstructured meshes. For a more detailed review on  recent developments of numerical methods for chemotaxis problems, the reader can consult \cite{Alina_2018} and \cite{MR2958929}.

In general, the design of numerical methods on unstructured meshes is more computationally intensive, in comparison to design of methods on structured meshes. Thus, in this work, we extend the numerical algorithm designed in \cite{Epsh1} to chemotaxis systems in 3D irregular domain. The developed numerical methods based on Difference Potentials Method can handle 3D irregular geometry with the use of only Cartesian meshes and employ Fast Poisson Solvers. In addition, to further enhance computational efficiency of the numerical algorithms, we design a Difference-Potentials-based domain decomposition approach which allows mesh adaptivity and easy parallelization of the algorithm in space. The proposed numerical algorithms are shown to be positivity-preserving, robust, accurate and computationally efficient.

The paper is organized as follows. In Section~\ref{sec:dpm_sd}, we formulate a positivity-preserving upwind Difference Potentials Method for accurate and efficient approximations of solutions to chemotaxis models in spherical domains. Next, in Section~\ref{sec:domain_decomposition}, a domain decomposition approach based on Difference Potentials Method is introduced, which improves the efficiency of the algorithm at no loss of accuracy. In Section~\ref{sec:numerical_results}, extensive numerical experiments (convergence studies in space and in time, long-time simulations, etc.) are presented to illustrate the accuracy, efficiency and robustness of the developed numerical algorithms. 


\section{An Algorithm Based on DPM}\label{sec:dpm_sd}

The current work is an extension of the work \cite{Epsh1} in 2D, to 3D chemotaxis systems and to adaptive algorithms in space that allows easy parallelization. For the time being, we will consider the PKS chemotaxis model in a spherical domain, but the proposed methods can be extended to more general domains in 3D (and the main ideas of the algorithms stay the same). We employ a finite-volume-finite-difference scheme as the underlying discretization of the model \eqref{eqn:minimal_chemotaxis} in space, combined with the idea of Difference Potentials Method (\cite{Ryab} and some very recent work \cite{MR3659255,Epsh1,Epsh,MR3817808}, etc.), that provides flexibility to handle irregular domains accurately and efficiently by the use of simple Cartesian meshes.


\paragraph{Introduction of the Auxiliary Domain:}
As a first step of the proposed method, we embed the domain $\Omega$ of model \eqref{eqn:minimal_chemotaxis} into a computationally simple auxiliary domain $\Omega^0 \subset \mathbb{R}^3$ that we will select to be a cube.  Next we discretize the auxiliary domain $\Omega^0$ using a Cartesian mesh, with uniform cells $D_{j,k,l}=[x_{j-\frac{1}{2}},x_{j+\frac{1}{2}}]\times[y_{k-\frac{1}{2}},y_{k+\frac{1}{2}}]\times[z_{l-\frac{1}{2}},z_{l+\frac{1}{2}}]$ of volume $\Delta x\Delta y\Delta z$  centered at the point $(x_j,y_k,z_l)$, ($j,k,l=1,\dots,N$), and we assume here for simplicity, $h:=\Delta x=\Delta y=\Delta z$. Note that, we select the same auxiliary domain and the same mesh for the approximation of $\rho$ and $c$, whereas in general, the auxiliary domains and meshes for $\rho$ and $c$ need not be the same. After that, we define the standard 7-point stencil with center placed at $(x_j,y_k,z_l)$  that we will consider as a part of the discretization of model \eqref{eqn:minimal_chemotaxis}:
\begin{align}\label{stencil:7-point}
    \mathcal{N}^7_{j,k,l}:=\{(x_j,y_k,z_l),(x_{j\pm1},y_k,z_l),(x_j,y_{k\pm1},z_l),(x_j,y_k,z_{l\pm1})\}
\end{align}

Now we are ready to define point sets that will be used in the proposed hybrid finite-volume-finite-difference approximation combined with DPM. 
\begin{definition}\label{def:point_sets} 
Introduce following point sets:\\
\begin{itemize}
    \item $M^0= \left\{(x_j,y_k,z_l)\mid(x_j,y_k,z_l)\in\Omega^0\right\}$ denotes the set of all the cell centers $(x_j,y_k,z_l)$ that belong to the interior of the auxiliary domain $\Omega^0$;
    \item $M^+=M^0\cap\Omega=\left\{(x_j,y_k,z_l)\mid(x_j,y_k,z_l)\in\Omega\right\}$ denotes the set of all the cell centers $(x_j,y_k,z_l)$ that belong to the interior of the original domain $\Omega$ (see Fig.~\ref{fig:SD_2D_Mp_Np});
    \item
    $M^-=M^0\backslash M^+=\{(x_j,y_k,z_l)\mid(x_j,y_k,z_l)\in\Omega^0\backslash\Omega\}$ is the set of all the cell centers $(x_j,y_k,z_l)$ that are inside of the auxiliary domain $\Omega^0$,  but belong to the exterior of the original domain $\Omega$;
    \item $N^+=\left\{\bigcup_{j,k,l}\mathcal{N}_{j,k,l}^{7}\mid(x_j,y_k,z_l)\in M^+\right\}$;
    \item $N^-=\left\{\bigcup_{j,k,l}\mathcal{N}_{j,k,l}^{7}\mid(x_j,y_k,z_l)\in M^-\right\}$;
    \item $N^0=\left\{\bigcup_{j,k,l}\mathcal{N}_{j,k,l}^{7}\mid(x_j,y_k,z_l)\in M^0\right\}$;\\ The point sets $N^\pm$ and $N^0$ are the sets of cell centers covered by the stencil $\mathcal{N}^7_{j,k,l}$ for every cell center $(x_j,y_k,z_l)$ in $M^\pm$ and $M^0$ respectively (see Fig.~\ref{fig:SD_2D_Mp_Np});
    \item $\gamma=N^+\cap N^-$ defines a thin layer of cell centers that straddles the continuous boundary $\Gamma$ and is called the discrete grid boundary (see Fig.~\ref{fig:SD_2D_gamma});
    \item $\gamma_{in}=M^+\cap\gamma$ and $\gamma_{ex}=M^-\cap\gamma$ are subsets of the discrete grid boundary that lie inside and outside of the spherical domain $\Omega$ respectively (see Fig.~\ref{fig:SD_2D_gamma}).
\end{itemize}
\end{definition}

\begin{figure}
    \centering
    \begin{subfigure}{0.4\textwidth}
        \centering
        \includegraphics[width=\textwidth]{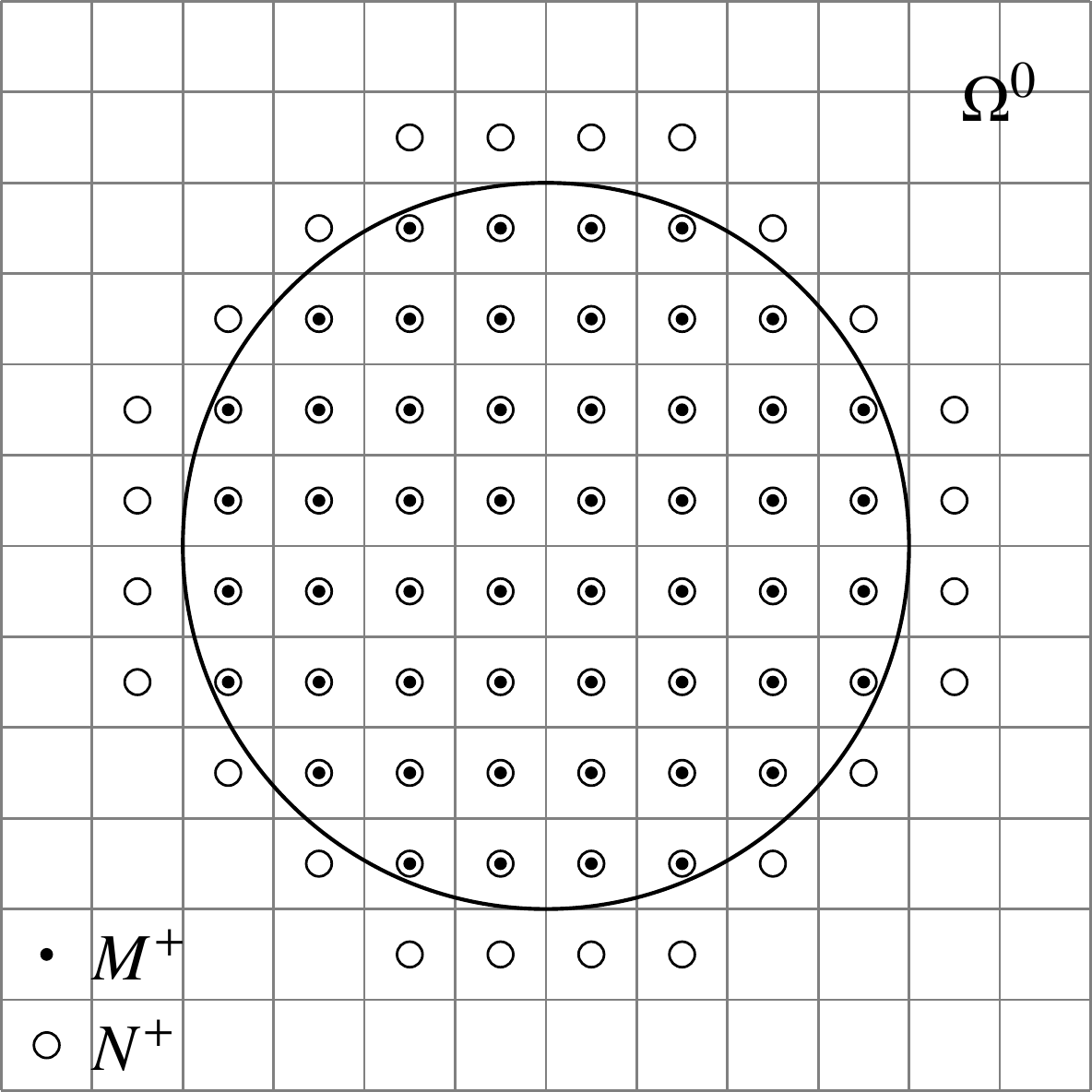}
        \caption{}
        \label{fig:SD_2D_Mp_Np}
    \end{subfigure}
    ~
    \begin{subfigure}{0.4\textwidth}
        \centering
        \includegraphics[width=\textwidth]{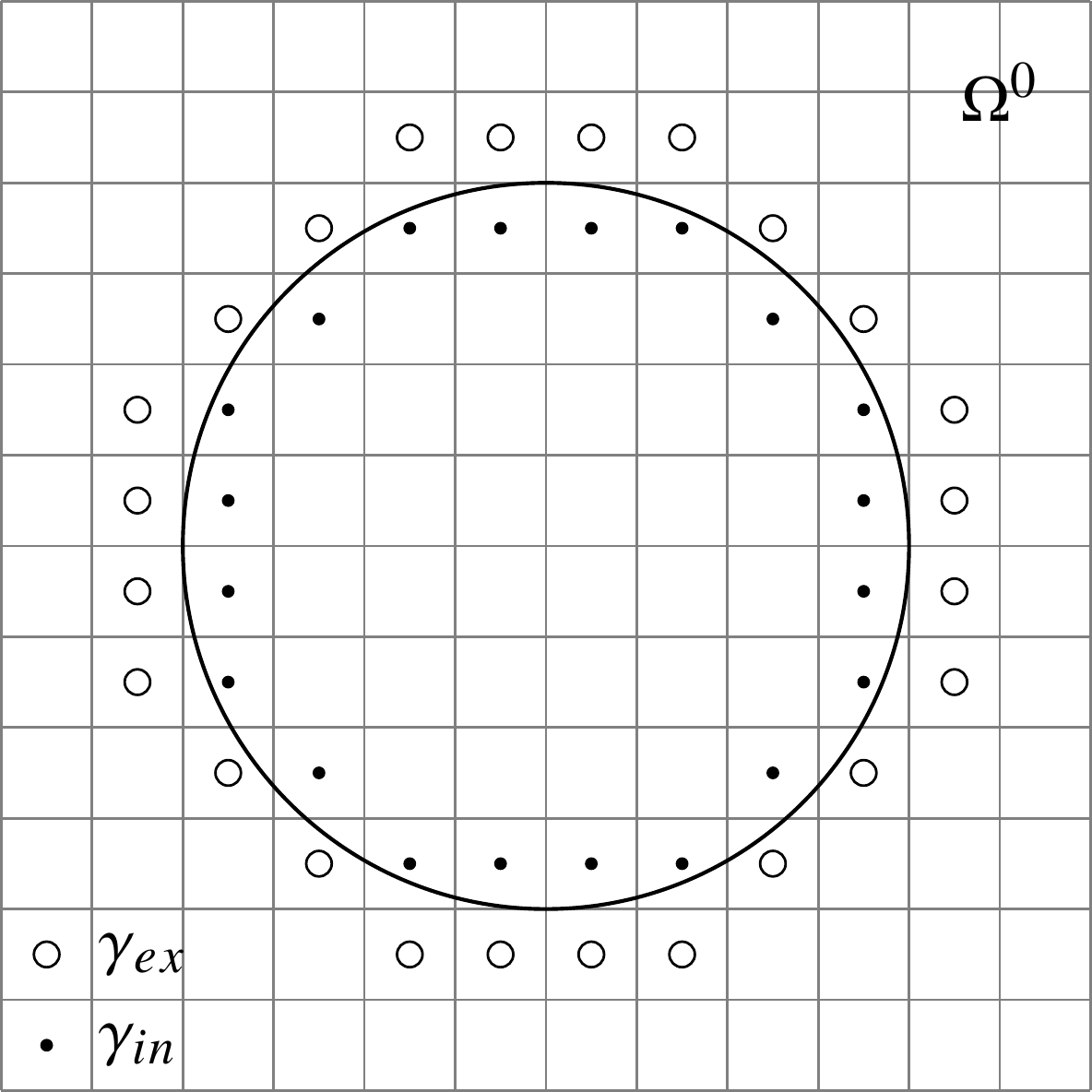}
        \caption{}
        \label{fig:SD_2D_gamma}
    \end{subfigure}
    \caption{Example in the cross-sectional view: (a) $M^+$ (solid dots) as a subset of $N^+$ (open circles), where solid dots in open circles show the overlap between $M^+$ and $N^+$; and (b) the discrete grid boundary $\gamma$ as the union of $\gamma_{ex}$ (open circles) and $\gamma_{in}$ (solid dots). The square auxiliary domain is denoted as $\Omega^0$ in both figures.}\label{fig:SD_2D_example}
\end{figure}


\paragraph{Construction of the System of Discrete Equations for Model~\eqref{eqn:minimal_chemotaxis}.} In each cell $D_{j,k,l}$ of volume $|D_{j,k,l}|$, we denote the approximation of the cell average of the density, the approximation of point values of the density and the point values of the chemoattractant concentration at time $t$ as:
\begin{align}
    \bar{\rho}_{j,k,l} (t) &\approx \frac{1}{|D_{j,k,l}|}\int_{D_{j,k,l}}\rho(x,y,z,t)dxdydz, \mbox{ and } \rho _{j,k,l} (t) \approx \rho(x_j,y_k,z_l,t),\\
    c _{j,k,l} (t) &\approx c(x_j,y_k,z_l,t).
\end{align}
We will also denote by $\bar{\rho}^{i}_{j,k,l}$, ${\rho}^i_{j,k,l}$ and $c^i _{j,k,l}$ the computed cell average of the density, point value of the density and the point value of the chemoattractant concentration at cell center $(x_j,y_k,z_l)$ at the discrete time level $t^i$, respectively.

Next, we use a hybrid finite-volume-finite-difference scheme in space, similar to 2D algorithms in \cite{Chertock_2017,Epsh1} and first order implicit-explicit (IMEX) scheme in time as the underlying discretization for the Patlak-Keller-Segel model \eqref{eqn:minimal_chemotaxis} at time $t^{i+1}$ for every $(x_j,y_k,z_l)$ in $M^+$:
\begin{equation}\label{eqn:discrete_chemotaxis}
    \left\{
    \begin{aligned}
    (I-\Delta t\Delta_h)\bar{\rho}^{i+1}_{j,k,l}&= \bar{\rho}^{i}_{j,k,l}-\Delta tg^i_{j,k,l}\\
    (I-\Delta t\Delta_h)c^{i+1}_{j,k,l} &= (1-\Delta t)c^{i}_{j,k,l}+\Delta t\bar{\rho}^i_{j,k,l}
    \end{aligned} \right.
    \quad (x_j,y_k,z_l)\in M^+,
\end{equation}
where $\Delta_h$ is the discrete Laplace operator obtained using standard second-order centered finite difference approximation, $I$ denotes the identity matrix of the same size with $\Delta_h$ and $\Delta t:=t^{i+1}-t^{i}$. Here, in the right hand side of the density equation, $g^i_{j,k,l}$ denotes the discretization of the convection term $\nabla\!\cdot\!(\chi\rho\nabla c)$ in \eqref{eqn:minimal_chemotaxis} on $M^+$ and is evaluated at previous time level $t^{i}$ as:
\begin{align}\label{eqn:discrete_convective_term}
g^{i}_{j,k,l} =&\quad \frac{\chi\rho^i_{j+\frac{1}{2},k,l}\nabla^h_x c^i_{j+\frac{1}{2},k,l}-\chi\rho^i_{j-\frac{1}{2},k,l}\nabla^h_x c^i_{j-\frac{1}{2},k,l}}{h}\nonumber\\
&+ \frac{\chi\rho^i_{j,k+\frac{1}{2},l}\nabla^h_y c^i_{j,k+\frac{1}{2},l}-\chi\rho^i_{j,k-\frac{1}{2},l}\nabla^h_y c^i_{j,k-\frac{1}{2},l}}{h}\quad\\
&+\frac{\chi\rho^i_{j,k,l+\frac{1}{2}}\nabla^h_z c^i_{j,k,l+\frac{1}{2}}-\chi\rho^i_{j,k,l-\frac{1}{2}}\nabla^h_z c^i_{j,k,l-\frac{1}{2}}}{h}\nonumber
\end{align}
where $\nabla^h_x c^i_{j\pm1/2,k,l}$, $\nabla^h_y c^i_{j,k\pm1/2,l}$, $\nabla^h_z c^i_{j,k,l\pm1/2}$ are components of the discrete gradient of concentration $c$ at the center of the six faces in cell $D_{j,k,l}$. The component $\nabla^h_x c^i_{j+\frac{1}{2},k,l}$ is computed using the central finite difference:
\begin{align}\label{eqn:c-central-difference}
    \nabla^h_x c^i_{j+\frac{1}{2},k,l} = \frac{c^i_{j+1,k,l}-c^i_{j,k,l}}{h}, \quad (x_j,y_k,z_l)\in M^+.
\end{align} 
The other components $\nabla^h_x c^i_{j-\frac{1}{2},k,l}$, $\nabla^h_y c^i_{j,k\pm1/2,l}$, $\nabla^h_z c^i_{j,k,l\pm1/2}$ in the discrete gradient of $c$ are computed similarly as in \eqref{eqn:c-central-difference}. 

In addition, $\rho^i_{j\pm1/2,k,l}$, $\rho^i_{j,k\pm1/2,l}$, $\rho^i_{j,k,l\pm1/2}$ are the approximation of density values at the center of the six faces of the same cell $D_{j,k,l}$, which are evaluated in an upwind manner. For example, the component in $x$-direction $\rho^i_{j+1/2,k,l}$ is computed using the following piecewise linear construction:
\begin{equation}\label{eqn:upwind_approximation}
\begin{aligned}
    \rho^i_{j+\frac{1}{2},k,l}=\left\{
    \begin{array}{ll}
    \tilde{\rho}^i(x_{j+\frac{1}{2}}-0,y_k,z_l), & \quad\nabla^h_x c^i_{j+\frac{1}{2},k,l}>0,\\
    \tilde{\rho}^i(x_{j+\frac{1}{2}}+0,y_k,z_l), & \quad\mbox{otherwise},\\
    \end{array}
    \right.
\end{aligned}
\end{equation}
where
\begin{align}\label{eqn:truncated_taylor_expansion}
    \tilde{\rho}^i(x,y,z) = \bar{\rho}^i_{j,k,l}+\nabla^h \bar{\rho}^i_{j,k,l}\cdot\langle x-x_j,y-y_k,z-z_l\rangle,\quad (x,y,z)\in D_{j,k,l},
\end{align}
and, thus in (\ref{eqn:upwind_approximation}), 
\begin{align*}
\tilde{\rho}^i(x_{j+\frac{1}{2}}-0,y_k,z_l)=\bar{\rho}^i_{j,k,l}+\frac{h}{2}\nabla^h_x \bar{\rho}^i_{j,k,l}, \mbox{ and }  \tilde{\rho}^i(x_{j+\frac{1}{2}}+0,y_k,z_l)=\bar{\rho}^i_{j+1,k,l}-\frac{h}{2}\nabla^h_x \bar{\rho}^i_{j+1,k,l}.
\end{align*}
In the second order piecewise linear construction~\eqref{eqn:truncated_taylor_expansion}, $\langle x-x_j,y-y_k,z-z_l\rangle$ denotes the vector defined by two points: $(x,y,z)\in D_{j,k,l}$ and $(x_j,y_k,z_l)\in D_{j,k,l}$, $\nabla^h\bar{\rho}^i_{j,k,l}$ is the discrete gradient of $\bar{\rho}^i$ at cell center $(x_j,y_k,z_l)$ and time $t^i$, and $\nabla^h_x \bar{\rho}^i_{j,k,l}$ is its $x$-component. Each component of the gradient $\nabla^h\bar{\rho}^i_{j,k,l}$ is calculated using the minmod slope limiter. For example, the $x$-component $\nabla^h_x \bar{\rho}^i_{j,k,l}$ is computed as
\begin{align*}
    \nabla^h_x \bar{\rho}^i_{j,k,l} & = \mbox{minmod}\left(2\frac{\bar{\rho}^i_{j+1,k,l}-\bar{\rho}^i_{j,k,l}}{h},\frac{\bar{\rho}^i_{j+1,k,l}-\bar{\rho}^i_{j-1,k,l}}{2h},2\frac{\bar{\rho}^i_{j,k,l}-\bar{\rho}^i_{j-1,k,l}}{h}\right)
\end{align*}
where the minmod function is defined by
\begin{align}\label{eqn:slope_limiter}
    \mbox{minmod}(x_1,x_2,\dots,x_{\mathcal{M}})=\left\{
    \begin{array}{cl}
    \min_j\{x_j\}, & \quad \mbox{ if }x_j>0, \forall j = 1,2,\dots,\mathcal{M},\\
    \max_j\{x_j\}, & \quad \mbox{ if }x_j<0, \forall j = 1,2,\dots,\mathcal{M},\\
    0, & \quad\mbox{ otherwise}.
    \end{array}
    \right.
\end{align}
The density values $\rho^i_{j-\frac{1}{2},k,l}$, $\rho^i_{j,k\pm1/2,l},\rho^i_{j,k,l\pm1/2}$ at the center of the other faces in cell $D_{j,k,l}$ are computed similarly as in \eqref{eqn:upwind_approximation}.
\begin{remark}
The non-negativity of the reconstructed point values of $\rho$ is ensured by the positivity-preserving
generalized minmod limiter, \cite{MR1486274,MR1976211,MR1047564,MR760628} with minmod function as defined in (\ref{eqn:slope_limiter}), under the assumption that the cell averages of the density are nonnegative.
\end{remark}


\paragraph{The Discrete Auxiliary Problem (AP).} One of the important steps of DPM-based methods is the introduction of the auxiliary problem (AP). For brevity, we will denote $u^{i+1} :=\bar{\rho}^{i+1}$ or $u^{i+1} :=c^{i+1}$. Thus,  the two difference equations in \eqref{eqn:discrete_chemotaxis}, which are decoupled after discretization of model~\eqref{eqn:minimal_chemotaxis} with IMEX, can be cast into a compact form, i.e.
\begin{align}\label{eqn:shorter_discrete_chemotaxis}
    L_{h, \Delta t} u_{j, k, l}^{i+1}=f_{j, k, l}^{i},\quad(x_j,y_k,z_l)\in M^+,
\end{align}
where $L_{h, \Delta t}:=(I-\Delta t\Delta_h)$ and $f_{j, k, l}^i$ is the right hand side function which is either $f_{j, k, l}^i=\bar{\rho}^{i}_{j,k,l}-\Delta tg^i_{j,k,l}$ (for the density equation) or $f_{j, k, l}^i=(1-\Delta t)c^{i}_{j,k,l}+\Delta t\bar{\rho}^i_{j,k,l}$ (for the chemoattractant concentration equation), as in \eqref{eqn:discrete_chemotaxis}.

Next, we define the discrete Auxiliary Problem, which will play a key role in construction of the \textit{Particular Solution} and the
\textit{Difference Potentials} as a part of DPM-based algorithm proposed in this work.
\begin{definition}
At time $t^{i+1}$, given the grid function $q^{i}$ on $M^0$, the following difference equations~\eqref{eqn:discrete_ap}--\eqref{eqn:discrete_ap_boundary} are defined as the discrete Auxiliary Problem (AP):
\begin{align}
L_{h,\Delta t}v_{j,k,l}^{i+1}&=q^{i}_{j,k,l},\quad(x_j,y_k,z_l)\in M^0,\label{eqn:discrete_ap}\\
v^{i+1}_{j,k,l}&=0,\quad(x_j,y_k,z_l)\in N^0\backslash M^0.\label{eqn:discrete_ap_boundary}
\end{align}
\end{definition}
Here $L_{h,\Delta t}$ is the linear operator similar to the one  in \eqref{eqn:shorter_discrete_chemotaxis}, but is defined now on a larger point set $M^0$.

\begin{remark}
The homogeneous Dirichlet boundary condition~\eqref{eqn:discrete_ap_boundary} in the AP is chosen merely for efficiency of our algorithm, i.e. we employ Fast Poisson Solvers to solve the AP. In general, other boundary conditions can be selected for the AP as long as the defined AP is well-posed and can be solved computationally efficiently.
\end{remark}


\paragraph{Construction of the Particular Solution.} Let us denote by
$G_{h,\Delta t}f_{j,k,l}^{i},\;(x_j,y_k,z_l)\in N^+$ the Particular Solution which is defined on $N^+$ of the fully discrete problem~\eqref{eqn:shorter_discrete_chemotaxis} at time level $t^{i+1}$. The Particular Solution is obtained by solving the AP \eqref{eqn:discrete_ap}--\eqref{eqn:discrete_ap_boundary} with the following right hand side:
\begin{align}\label{rhs:particular_solution}
q^i_{j,k,l}=\left\{
\begin{array}{ll}
f_{j,k,l}^{i},&\quad(x_j,y_k,z_l)\in M^+,\\
0,& \quad(x_j,y_k,z_l)\in M^-,
\end{array}
\right.
\end{align}
and by restricting the computed solution from $N^0$ to $N^+$.
\begin{remark}
Note that, if the center $(x_j,y_k,z_l)$ of cell $D_{j,k,l}$ belongs to $\gamma_{in}$, some points $(x_{j\pm1},y_k,z_l)$, $(x_j,y_{k\pm1},z_l)$ or $(x_j,y_k,z_{l\pm1})$ from stencil $\mathcal{N}_{j,k,l}^7$ with the center point $(x_j,y_k,z_l)$ will lie in $\gamma_{ex}$, and hence, outside of the domain $\Omega$. To construct the Particular Solution for the density approximation in \eqref{eqn:discrete_chemotaxis}-\eqref{eqn:discrete_convective_term} and \eqref{eqn:discrete_ap}-\eqref{rhs:particular_solution}, one needs to compute the discretized convective term $g^{i}_{j,k,l}$ for every point $(x_j,y_k,z_l)\in M^+$. Hence, we need to approximate values of $\rho$ and $c$ at the points that belong to set $\gamma_{ex}$. We will take the following strategy to define values of $\rho$ and $c$ on $\gamma_{ex}$ and $\gamma_{in}$:
\begin{itemize}
    \item Initially, we approximate values $\rho^0_{j,k,l}$ and $c^0_{j,k,l}$ for $(x_j,y_k,z_l)\in \gamma$ using 2-term extension operator~\eqref{eqn:extension_operator} and zero Neumann boundary conditions, i.e.
    \begin{align}\label{eqn:boundary_approx}
    \rho^0_{j,k,l}\approx \rho_0(x'_j,y'_k,z'_l)\mbox{ and }c^0_{j,k,l}\approx c_0(x'_j,y'_k,z'_l).
    \end{align}
    Here, $(x'_j,y'_k,z'_l)$ is the orthogonal projection on the continuous boundary $\Gamma$ corresponding to a cell center $(x_j,y_k,z_l)\in\gamma$, and $\rho_0$ and $c_0$ are the initial conditions.
    \item At later time level $t^{i+1}$, $g^i_{j,k,l}$ is computed using the solution $\bar{\rho}^{i}_{j,k,l}$ and $c^{i}_{j,k,l}$ obtained from the discrete generalized Green's formula~\eqref{eqn:generalized_greens_formula} at previous time level $t^i$.
\end{itemize}
\end{remark}


\paragraph{Construction of the Difference Potentials.} To construct the Difference Potentials, let us first define a linear space $V_{\gamma}$ of all grid functions $v^{i+1}_{\gamma} (x_j,y_k, z_l)$ at $t^{i+1}$ on $\gamma$. The functions are extended by zero to other points in $N^0$ set. These grid functions $v^{i+1}_{\gamma}$ are  called densities on the discrete grid boundary $\gamma$ at the time level $t^{i+1}$.
\begin{definition}\label{DP}
The Difference Potential associated with a given density $v^{i+1}_{\gamma}\in V_{\gamma}$ is the grid function $P_{N^+\gamma}v_{\gamma}^{i+1}$ defined on $N^+$ at the time level $t^{i+1}$, and is obtained by solving the AP~\eqref{eqn:discrete_ap}--\eqref{eqn:discrete_ap_boundary} with the following right hand side:
\begin{align}
q^i_{j,k,l}&=\left\{
\begin{array}{ll}
0,&\quad(x_j,y_k,z_l)\in M^+,\\
L_{h,\Delta t}[v^{i+1}_{\gamma}],&\quad(x_j,y_k,z_l)\in M^-,
\end{array}
\right.\label{rhs:difference_potentials}
\end{align}
\end{definition}
and by restricting the solution from $N^0$ to $N^+$. In Definition~\ref{DP}, $P_{N^+ \gamma}$ denotes the operator which constructs Difference Potential $P_{N^+\gamma}v_{\gamma}^{i+1}$ from the density $v^{i+1}_{\gamma}$ at time $t^{i+1}$. Note that Difference Potential is a linear operator of the density function, $P_{N^+\gamma}v_{\gamma}^{i+1}|_m=\sum_{\mathfrak{n}\in \gamma} A_{\mathfrak{n}m} v^{i+1}_\mathfrak{n}$, where $m\equiv (j, k, l)$ is the index of the grid point in the set $N^+$, $\mathfrak{n}$ is the index of the grid point in the set $\gamma$, $P_{N^+\gamma}v_{\gamma}^{i+1}|_m$ is the value of the Difference Potential at the grid point with index $m$ and $A_{\mathfrak{n}m}$ are the coefficients of the Difference Potential.

Next we will introduce the trace operator.  Given a grid function $v^{i+1}$ defined on the point set $N^+$, we denote by $Tr_{\gamma}v^{i+1}$ the trace or restriction of $v^{i+1}$ from $N^+$ to the discrete grid boundary $\gamma$. Similarly, we define $Tr_{\gamma_{in}}v^{i+1}$ as the trace or restriction of $v^{i+1}$ from $N^+$ to  $\gamma_{in}\subset \gamma$.  We are ready to define an operator $P_{\gamma}:V_{\gamma}\rightarrow V_{\gamma}$ such that $P_{\gamma}v^{i+1}_{\gamma}:=Tr_{\gamma}P_{N^+\gamma}v^{i+1}_{\gamma}$. The operator $P_{\gamma}$ is a projection operator. Now we will state the key theorem in Difference Potentials Method, which allows us to reformulate the difference equation \eqref{eqn:shorter_discrete_chemotaxis} defined in $M^+$ into an equivalent {\it Boundary Equation with Projections} (BEP) defined on the discrete grid boundary $\gamma$ only (see \cite{Ryab} for more details).
\begin{theorem}[Boundary Equations with Projections (BEP)]\label{thm:full_BEP}
At time $t^{i+1}$, the discrete density $u^{i+1}_{\gamma}$ is the trace of some solution $u^{i+1}$ on $N^+$ to the difference equation~\eqref{eqn:shorter_discrete_chemotaxis}, i.e. $u^{i+1}_{\gamma}:=Tr_{\gamma}u^{i+1}$, if and only if the following BEP holds:
\begin{align}\label{eqn:full_BEP}
    u^{i+1}_{\gamma}-P_{\gamma}u^{i+1}_{\gamma}=G_{h,\Delta t}f^{i}_{\gamma}, \quad  (x_j, y_k, z_l)\in \gamma,
\end{align}
where $G_{h,\Delta t}f^{i}_{\gamma}:=Tr_{\gamma}G_{h,\Delta t}f^{i}_{j, k, l}$ is the trace of the Particular Solution restricted to the discrete grid boundary $\gamma$.
\end{theorem}

\begin{proof}
See Appendix~\ref{apd:proof-thm-full-BEP}.
\end{proof}

\begin{remark}
Note, using that Difference Potential is a linear operator, we can recast (\ref{eqn:full_BEP}) as
\begin{align}\label{eqn:full_BEP1}
    u^{i+1}_{m}-\sum_{\mathfrak{n}\in \gamma}A_{\mathfrak{n}m} u^{i+1}_{\mathfrak{n}}=G_{h,\Delta t}f^{i}_m, \quad m\in \gamma,
\end{align}
where $m$ is the index of the grid point in the set $\gamma$ and $G_{h,\Delta t}f^{i}_m$ is the value of the Particular Solution at the grid point with index $m$ in the set $\gamma$.
\end{remark}

\begin{proposition}\label{prop:rank}
The rank of linear equations in BEP~\eqref{eqn:full_BEP} is $|\gamma_{in}|$, which is the cardinality of the point set $\gamma_{in}$.
\end{proposition}

\begin{proof}
See Appendix~\ref{apd:proof-prop-rank}.
\end{proof}

Next we introduce the reduced BEP \eqref{eqn:reduced_BEP} defined only on $\gamma_{in}$ and show that it is equivalent to the BEP~\eqref{eqn:full_BEP} defined on $\gamma$.
\begin{theorem}\label{thm:reduced_BEP}
The BEP~\eqref{eqn:full_BEP} defined on $\gamma$ in Theorem~\ref{thm:full_BEP} is equivalent to the following BEP~\eqref{eqn:reduced_BEP} defined on a smaller subset $\gamma_{in}\subset \gamma$:
\begin{align}\label{eqn:reduced_BEP}
    u^{i+1}_{\gamma_{in}}-Tr_{\gamma_{in}}P_{\gamma}u_{\gamma}^{i+1}=Tr_{\gamma_{in}}G_{h,\Delta t}f^{i}_{\gamma}, \quad (x_j, y_k, z_l)\in \gamma_{in}
\end{align}
Moreover, the reduced BEP~\eqref{eqn:reduced_BEP} contains only linearly independent equations.
\end{theorem}

\begin{proof}
See Appendix~\ref{apd:proof-thm-reduced-BEP}.
\end{proof}

Similarly to (\ref{eqn:full_BEP})-(\ref{eqn:full_BEP1}), the reduced BEP~\eqref{eqn:reduced_BEP} can be recast as 
\begin{align}\label{eqn:reduced_BEP1}
    u^{i+1}_{m}-\sum_{\mathfrak{n}\in \gamma}A_{\mathfrak{n}m} u^{i+1}_{\mathfrak{n}}=G_{h,\Delta t}f^{i}_m, \quad m\in \gamma_{in}.
\end{align}

\begin{remark}
The BEP~\eqref{eqn:full_BEP} or \eqref{eqn:reduced_BEP} reduces degrees of freedom from $\mathcal{O}(h^{-3})$ in the difference equation \eqref{eqn:shorter_discrete_chemotaxis} to $\mathcal{O}(h^{-2})$. In addition, the reduced BEP~\eqref{eqn:reduced_BEP} defined on $\gamma_{in}$ reduces the number of equations in BEP \eqref{eqn:full_BEP} by approximately one half, since $|\gamma_{in}|\approx|\gamma|/2$. Thus, using the reduced BEP~\eqref{eqn:reduced_BEP} will further reduce the computational cost in our numerical algorithm and we will use the reduced BEP as a part of the proposed numerical algorithm.
\end{remark}

Additionally, let us note the BEP~\eqref{eqn:full_BEP} or BEP~\eqref{eqn:reduced_BEP} will admit multiple solutions $u^{i+1}_{\gamma}$ since the system of equations \eqref{eqn:full_BEP} is equivalent to the system of difference equations \eqref{eqn:discrete_chemotaxis} without imposing boundary conditions yet. Therefore, to construct a unique solution to BEP~\eqref{eqn:reduced_BEP}, we need to supply the BEP \eqref{eqn:reduced_BEP} with zero Neumann boundary conditions for the density and concentration. To impose boundary conditions efficiently into BEP, we will introduce the extension operator \eqref{eqn:extension_operator} below, similarly to \cite{MR3659255,Epsh1,Epsh,MR3817808}, etc.
\begin{definition}
The extension operator $\pi_{\gamma \Gamma}[u^{i+1}]$ of the function $u(x,y,z,t^{i+1})$ is defined as:
\begin{align}\label{eqn:extension_operator}
    \pi_{\gamma \Gamma}[u^{i+1}]|_{(x_j, y_k,z_l)}:=u^{i+1}|_{\Gamma}+\left.d\frac{\partial u^{i+1}}{\partial n}\right|_{\Gamma}+\beta\left.\frac{d^2}{2}\frac{\partial^2u^{i+1}}{\partial n^2}\right|_{\Gamma},
\end{align}
where $n$ is the unit outward normal vector on $\Gamma$, $d$ is the signed distance between a point $(x_j,y_k,z_l)\in \gamma$ and the
point of its orthogonal projection on the continuous boundary $\Gamma$ in the direction of $n$. The parameter $\beta$ controls the number of terms that will be used in the extension operator. If $\beta=0$, we call \eqref{eqn:extension_operator} the 2-term extension operator and if $\beta=1$, we call \eqref{eqn:extension_operator} the 3-term extension operator. We select $\beta$ based on the regularity of the solution and to achieve the overall second-order accuracy of the numerical approximation in space.
\end{definition}
Basically, the extension operator~\eqref{eqn:extension_operator} defines values of the density $u^{i+1}_{\gamma}$ at the point of the discrete grid boundary $(x_j, y_k, z_l)\in \gamma$ through the values of the continuous solution and its gradients at time $t^{i+1}$ at the continuous boundary $\Gamma$ of the domain with the desired accuracy. 


\paragraph{Spectral Approach.} With the extension operator \eqref{eqn:extension_operator} defined, the homogeneous Neumann boundary condition is readily incorporated, i.e. the second term in the right hand side of \eqref{eqn:extension_operator} will vanish. Hence, the extension operator~\eqref{eqn:extension_operator} for our problem reduces to:
\begin{align}\label{eqn:extension_operator1}
    \pi_{\gamma \Gamma}[u^{i+1}]|_{(x_j, y_k, z_l)}=u^{i+1}|_{\Gamma}+\beta\left.\frac{d^2}{2}\frac{\partial^2u^{i+1}}{\partial n^2}\right|_{\Gamma}.
\end{align} 
Therefore, we incorporate the extension operator~\eqref{eqn:extension_operator1} into BEP~\eqref{eqn:reduced_BEP} to solve for the unique density $u^{i+1}_{\gamma} (x_j, y_k, z_l)\approx \pi_{\gamma \Gamma}[u^{i+1}]|_{(x_j, y_k, z_l)}, (x_j, y_k, z_l)\in \gamma$  at time $t^{i+1}$ as described below. However, to determine the unique density $u^{i+1}_{\gamma}$ at time $t^{i+1}$ using the BEP~\eqref{eqn:reduced_BEP} together with \eqref{eqn:extension_operator1}, we need to solve for the unknown solution $u^{i+1}|_{\Gamma}$ and its second order normal derivative $\left.\frac{\partial^2u^{i+1}}{\partial n^2}\right|_{\Gamma}$ at the continuous boundary $\Gamma$ at $t^{i+1}$. To do this efficiently and accurately, we will employ spectral approximations for the two unknown terms on the boundary of the domain:
\begin{align}\label{eqn:spectral_approximation_sd}
    u^{i+1}|_{\Gamma}\approx\sum_{\nu=0}^{M}C_{\nu}^{0, i+1}\phi_{\nu}(\theta,\varphi),\mbox{ and }\left.\frac{\partial^2u^{i+1}}{\partial n^2}\right|_{\Gamma}\approx\sum_{\nu=0}^{M}C_{\nu}^{2, i+1}\phi_{\nu}(\theta,\varphi),
\end{align}
where we take the basis functions $\phi_{\nu}(\theta,\varphi)$ to be spherical harmonics and ($\theta,\varphi$) are the polar and azimuthal angles respectively on the continuous boundary $\Gamma$. The spectral coefficients $\{C_{\nu}^{0, i+1},C_{\nu}^{2, i+1}\}$ ($\nu=0,1,\dots,M$) are the unknown coefficients that will be computed using BEP~\eqref{eqn:reduced_BEP} at every time level $t^{i+1}$.

After we incorporate the spectral approximation~\eqref{eqn:spectral_approximation_sd} into the extension operator \eqref{eqn:extension_operator1}, we will have:
\begin{align}\label{eqn:extension_operator_recast}
    u^{i+1}_{\gamma}\approx\sum_{\nu=0}^{M}C_{\nu}^{0, i+1}\phi_{\nu}(\theta,\varphi)+\beta\frac{d^2}{2}\sum_{\nu=0}^{M}C_{\nu}^{2, i+1}\phi_{\nu}(\theta,\varphi).
\end{align}
Here, $(\theta,\varphi)$ are the polar and azimuthal angles for every point on the continuous boundary $\Gamma$ which is the orthogonal projection of each point in the discrete grid boundary $\gamma$. Therefore, at every time level $t^{i+1}$, the BEP~\eqref{eqn:reduced_BEP} becomes an over-determined linear system of dimension $|\gamma_{in}|\times[(\beta+1)(M+1)]$ for the unknown coefficients $\{C_{\nu}^{0, i+1},C_{\nu}^{2, i+1}\}$ ($\nu=0,1,\dots,M$):
\begin{align}\label{eqn:BEP_recast}
\sum_{\nu=0}^{M}C_{\nu}^{0,i+1}\left[(I-P_\gamma)\phi_{\nu}(\theta,\varphi)\right]+\beta\sum_{\nu=0}^{M}C_{\nu}^{2,i+1}\left[(I-P_\gamma)\frac{d^2}{2}\phi_{\nu}(\theta,\varphi)\right]=G_{h,\Delta t}f^{i}_\gamma \mbox{ on } \gamma_{in},
\end{align}
where Least Squares Method can be used to solve for the coefficients $\{C_{\nu}^{0, i+1},C_{\nu}^{2, i+1}\}$ ($\nu=0,1,\dots,M$). 

\begin{remark}
(i) In numerical experiments, we consider initial conditions for $\rho$ and $c$ in the form of $f(-a(x^2+y^2+b(z-c)^2))$. Hence, we will only need the zonal modes in the spherical harmonics, i.e. $\phi_{\nu}(\theta,\varphi)=P_{\nu}^0(\cos\theta)$ where $P_{\nu}^0$ is the associated Legendre polynomial of degree $\nu$ and order 0. If the highest degree of the zonal spherical harmonics is $L$, the total number of harmonics for each term in \eqref{eqn:spectral_approximation_sd} is only $L+1$, which is significantly reduced comparing to $(L+1)^2$ when full spectrum of spherical harmonics up to degree $L$ is used.

(ii) Moreover, the spectral approach reduces the degrees of freedom from $\mathcal{O}(h^{-2})$ in the BEP \eqref{eqn:reduced_BEP} to $\mathcal{O}(M)$ in \eqref{eqn:BEP_recast}. In practice, we require $(\beta+1)(M+1)\ll|\gamma_{in}|$ in the over-determined linear system~\eqref{eqn:BEP_recast}. This implies, that $\mathcal{O}(M)\ll\mathcal{O}(h^{-2})$ since $|\gamma_{in}|\sim\mathcal{O}(h^{-2})$. Thus, we can solve the BEP~\eqref{eqn:reduced_BEP} efficiently. See also Section~\ref{sec:numerical_results} for the details on the number of harmonics $M$ and the mesh size $h$ used in the numerical examples.

(iii) In addition, we used the same number of harmonics $M+1$ for the terms $u^{i+1}|_{\Gamma}$ and $\left.\frac{\partial^2u^{i+1}}{\partial n^2}\right|_{\Gamma}$ for simplicity in our implementation. In general, the number of harmonics can be chosen independently based on the regularity of each term. We should note that one can select a different basis functions $\phi_{\nu}(\theta,\varphi)$ too, for example, spherical radial basis functions (see \cite{Hubbert_2015}). Additionally, when the bounded domain is of a more general shape but smooth, instead of spherical harmonics as considered in this paper, one can investigate use of  more general local radial basis functions to represent Cauchy data on the boundary of the domain as a part of the developed algorithms.

\end{remark}


\paragraph{Discrete Generalized Green's Formula.}  The final step of DPM is to use the computed density $u^{i+1}_{\gamma}$ to construct the approximation to continuous solutions of the chemotaxis model (\ref{eqn:minimal_chemotaxis}) in domain $\Omega$ subject to zero Neumann boundary conditions on $\rho$ and $c$.

\begin{proposition}[Discrete Generalized Green's formula.]\label{prop:discrete_gene_green_formula}
The discrete solution $u^{i+1}:={\bar \rho}^{i+1}$ or $u^{i+1}:={c}^{i+1}$ on $N^{+}$ constructed using {\it Discrete Generalized Green's formula}:
\begin{align}\label{eqn:generalized_greens_formula}
    u^{i+1}_{j,k,l}=P_{N^+\gamma}u^{i+1}_{\gamma}+G_{h,\Delta t}f^i_{j,k,l},\quad(x_j,y_k,z_l)\in N^+,
\end{align}
is the approximation to the exact solution $u:=\rho$ or $u:=c$, respectively, at $(x_j, y_k, z_l) \in \Omega$ at time  $t^{i+1}$ of the continuous chemotaxis model (\ref{eqn:minimal_chemotaxis}) subject to homogeneous Neumann boundary condition.  We also conjecture that we have the following accuracy of the proposed numerical scheme:
\begin{equation}
    \left|\left|u^{i+1}_{j,k,l}-u(x_j,y_k,z_l,t^{i+1})\right|\right|_{\infty}=\mathcal{O}(h^2+\Delta t).
\end{equation}
\end{proposition}

\begin{remark}
Indeed, in our numerical results (Section~\ref{sec:numerical_results}), we observe second order convergence in space and first order order convergence in time in the approximation of the solutions $\rho$ and $c$. Moreover, we use $\Delta t=0.5h^2$ in the numerical tests before blow-up to achieve second order accuracy of the proposed algorithms. See also \cite{MR3659255,MR3817808} for a more detailed discussion.
\end{remark}


\paragraph{Positivity in a Spherical Domain.}
Now, we establish a positivity-preserving property of the proposed hybrid finite-volume-finite-difference DPM scheme for the approximation of solutions to (\ref{eqn:minimal_chemotaxis}).

\begin{theorem}[Positivity in a Spherical Domain.]\label{thm:spherical_positivity}
The discrete solution $u^{i+1}=\bar{\rho}^{i+1}$ or $u^{i+1}=c^{i+1}$ obtained using the Discrete Generalized Green's formula~\eqref{eqn:generalized_greens_formula} is non-negative on $N^+$, provided that the following conditions are satisfied:
\begin{enumerate}
\item The discrete solution $u^i$ is non-negative on $N^+$ at the previous time level $t^i$;\label{cond:previous_positivity}
\item The CFL-type condition is satisfied in $M^+$:
\begin{equation}\label{eqn:cfl_condition}
\Delta t\leq\min\left\{\frac{h}{6\chi\max\limits_{j,k,l}|\nabla^h_xc^i_{j\pm\frac{1}{2},k,l}|},\frac{h}{6\chi\max\limits_{j,k,l}|\nabla^h_yc^i_{j,k\pm{\frac{1}{2}},l}|},\frac{h}{6\chi\max\limits_{j,k,l}|\nabla^h_zc^i_{j,k,l\pm{\frac{1}{2}}}|}\right\}
\end{equation}
where $(x_j,y_k,z_l)\in M^+$ and $\chi$ is the chemotactic sensitivity constant;\label{cond:cfl_condition}
\item The density $u^{i+1}_{\gamma}$ is non-negative on the discrete grid boundary $\gamma$ at time level $t^{i+1}$.\label{cond:positive_gamma}
\end{enumerate}
\end{theorem}

\begin{proof}
First, the discrete solution $u^{i+1}=\bar{\rho}^{i+1}$ or $u^{i+1}=c^{i+1}$ is constructed by using the Discrete Generalized Green's formula~\eqref{eqn:generalized_greens_formula}. Also, we assume that the associated density $u^{i+1}_{\gamma}$ that was obtained by solving the BEP as discussed above is non-negative. From the construction of our algorithm above, it can be seen that $u^{i+1}$ satisfies the discrete system on $N^0$:
\begin{align}
L_{h,\Delta t}u^{i+1}&=\left\{
\begin{array}{cc}
f^i,&\quad (x_j,y_k,z_l)\in M^+,\\
L_{h,\Delta t}[u^{i+1}_{\gamma}],&\quad (x_j,y_k,z_l)\in M^-,
\end{array} 
\right.\\
u^{i+1}&=0,\quad (x_j,y_k,z_l)\in N^0\backslash M^0,
\end{align}
where $L_{h,\Delta t}\equiv I-\Delta t\Delta_h$ on $M^0$. Hence, $u^{i+1}$ satisfies the following discrete system defined on $M^+$:
\begin{align}\label{eqn:difference_equation_Mplus}
(I-\Delta t\Delta_h)u^{i+1}&=f^i,\quad (x_j,y_k,z_l)\in M^+,
\end{align}
where the right hand side $f^i$ is defined as the right hand side in \eqref{eqn:discrete_chemotaxis} ($f^i$ is either the right hand side for $\rho$ or the right hand side for $c$).
We will use discrete system~\eqref{eqn:difference_equation_Mplus} with the known density $u^{i+1}_{\gamma}$ to prove the non-negativity of the discrete solution $u^{i+1}$.

To this end, we will first establish the CFL-type condition~\eqref{eqn:cfl_condition} to guarantee the non-negativity of right hand side $f^{i}\geq0$ in \eqref{eqn:difference_equation_Mplus}. When $u^i=c^i$, the right hand side $f^{i}$ in \eqref{eqn:difference_equation_Mplus} is automatically non-negative provided Condition~\ref{cond:previous_positivity} in Theorem 3 is satisfied, since we assume that $\Delta t$ is sufficiently small, namely, $\Delta t<1$. When $u^i=\bar{\rho}^i$, the right hand side $f^{i}$ in \eqref{eqn:difference_equation_Mplus} is non-negative provided the CFL-type condition \eqref{eqn:cfl_condition} is satisfied. To show this, we will proceed similarly to \cite{Chertock_2017} and \cite{Epsh1}. Let us first define the values at the center of six faces in a cell $D_{j,k,l},(x_j,y_k,z_l)\in M^+,$ and introduce the convenient abbreviations:
\begin{equation}\label{eqn:six_face_values}
\begin{aligned}
\mbox{West: }\rho^{W}_{j,k,l} &:=\tilde{\rho}^i(x_{j-\frac{1}{2}}+0,y_k,z_l),\\
\mbox{East: }\rho^{E}_{j,k,l} &:=\tilde{\rho}^i(x_{j+\frac{1}{2}}-0,y_k,z_l),\\
\mbox{South: }\rho^{S}_{j,k,l} &:=\tilde{\rho}^i(x_j,y_{k-\frac{1}{2}}+0,z_l),\\
\mbox{North: }\rho^{N}_{j,k,l} &:=\tilde{\rho}^i(x_j,y_{k+\frac{1}{2}}-0,z_l),\\
\mbox{Up: }\rho^{U}_{j,k,l} &:=\tilde{\rho}^i(x_j,y_k,z_{l+\frac{1}{2}}-0),\\
\mbox{Down: }\rho^{D}_{j,k,l} &:=\tilde{\rho}^i(x_j,y_k,z_{l-\frac{1}{2}}+0),
\end{aligned}
\end{equation}
where the function $\tilde{\rho}^i$ is the piecewise linear reconstruction defined in \eqref{eqn:truncated_taylor_expansion}. Note that $\rho^{W}_{j,k,l}, \rho^{E}_{j,k,l}$, $\rho^{S}_{j,k,l},\rho^{N}_{j,k,l},\rho^{U}_{j,k,l},\rho^{D}_{j,k,l}$ are non-negative due to Condition~\ref{cond:previous_positivity} of Theorem~\ref{thm:spherical_positivity} and the positivity preserving property of the piecewise linear reconstruction \eqref{eqn:truncated_taylor_expansion}. Further, the value $\bar{\rho}^i_{j,k,l}$ at the center of the cell $D_{j,k,l}$ can be expressed as:
\begin{align}
\bar{\rho}^i_{j,k,l}=\frac{1}{6}(\rho^{W}_{j,k,l}+\rho^{E}_{j,k,l}+\rho^{S}_{j,k,l}+\rho^{N}_{j,k,l}+\rho^{U}_{j,k,l}+\rho^{D}_{j,k,l}),
\end{align}
due to conservation property of the piecewise reconstruction.

Next, we rewrite right hand side $f^{i}_{j,k,l}$ in \eqref{eqn:difference_equation_Mplus} at point $(x_j,y_k,z_l)\in M^+$ as:
\begin{align}\label{eqn:rho_rhs_flux}
f^i_{j,k,l}=\bar{\rho}^{i}_{j,k,l}-\Delta t g^{i}_{j,k,l}=F^E_{j,k,l}+F^W_{j,k,l}+F^N_{j,k,l}+F^S_{j,k,l}+F^U_{j,k,l}+F^D_{j,k,l},
\end{align}
where
\begin{align}
F^E_{j,k,l}=\frac{1}{6}\rho^E_{j,k,l}-\dfrac{\Delta t \chi\nabla^h_x c^i_{j+\frac{1}{2},k,l}}{h}\rho^i_{j+\frac{1}{2},k,l},\label{eqn:rho_flux_E}\\
F^W_{j,k,l}=\frac{1}{6}\rho^W_{j,k,l}+\dfrac{\Delta t \chi\nabla^h_x c^i_{j-\frac{1}{2},k,l}}{h}\rho^i_{j-\frac{1}{2},k,l},\label{eqn:rho_flux_W}
\end{align}
and $F^{S}_{j,k,l},F^{N}_{j,k,l},F^{U}_{j,k,l},F^{D}_{j,k,l}$ in the other directions are defined similarly.

We will take the first term $F^E_{j,k,l}$ as an example to establish the CFL-type condition~\eqref{eqn:cfl_condition}:
\begin{itemize}
    \item When $\nabla^h_x c^i_{j+\frac{1}{2},k,l}<0$, the upwind scheme gives $\rho^i_{j+\frac{1}{2},k,l}=\rho^W_{j+1,k,l}$. Then the coefficients for $\rho^W_{j+1,k,l}$ and $\rho^E_{j,k,l}$ are automatically non-negative, which ensures $F^E_{j,k,l}\geq0$. 
    \item When $\nabla^h_x c^i_{j+\frac{1}{2},k,l}\geq0$, the upwind scheme gives $\rho^i_{j+\frac{1}{2},k,l}=\rho^E_{j,k,l}$. In this case, we require the coefficient of $\rho^E_{j,k,l}$ to be non-negative, and it leads us to the constraint on $\Delta t$:
    \[
    \Delta t\leq \frac{h}{6\chi|\nabla^h_xc^i_{j+\frac{1}{2},k,l}|},
    \]
    which will also ensure $F^E_{j,k,l}\geq0$. 
\end{itemize}
The details for non-negativity of $F^{W}_{j,k,l},F^{S}_{j,k,l},F^{N}_{j,k,l},F^{U}_{j,k,l},F^{D}_{j,k,l}$ are similar to $F^E_{j,k,l}$.
Thus, we obtain the CFL-type condition~\eqref{eqn:cfl_condition} to ensure non-negative right hand sides $f^{i}$ for the cell density $\rho$. 

Thus far, we have shown that the right hand side $f^i_{j,k,l}\geq0$ in \eqref{eqn:difference_equation_Mplus} is non-negative on $M^+$ under Conditions~\ref{cond:previous_positivity} and \ref{cond:cfl_condition}. What remains to be shown is that the solution $u^{i+1}$ to the discrete system~\eqref{eqn:difference_equation_Mplus} is non-negative on $M^+$, provided that the discrete boundary condition $u_{\gamma_{ex}}^{i+1}$ is non-negative on the point set $\gamma_{ex}$.

Without loss of generality, let us assume that the center $(x_{j},y_{k},z_{l})$ of a cell $D_{j,k,l}$ belongs to $M^+$, and only the point $(x_{j-1},y_{k},z_{l})$ in the 7-point stencil of $(x_{j},y_{k},z_{l})$ is in $\gamma_{ex}$. Since the boundary condition $u^{i+1}_{\gamma_{ex}}$ to the discrete system \eqref{eqn:difference_equation_Mplus} is given and non-negative, $u^{i+1}_{j-1,k,l}$ is non-negative for $(x_{j-1},y_{k},z_{l})\in \gamma_{ex}$. Note, that from the equation \eqref{eqn:difference_equation_Mplus} for cell $D_{j,k,l}$, we have
\begin{equation}
u^{i+1}_{j,k,l}-\frac{\Delta t}{h^2}(-6u^{i+1}_{j,k,l}+u^{i+1}_{j\pm1,k,l}+u^{i+1}_{j,k\pm1,l}+u^{i+1}_{j,k,l\pm1})=f^i_{j,k,l}.
\end{equation}
Re-arranging the non-negative known value $u^{i+1}_{j-1,k,l}$ to the right hand side would give us:
\begin{equation}\label{eqn:nonnegativity_in_rhs}
u^{i+1}_{j,k,l}-\frac{\Delta t}{h^2}(-6u^{i+1}_{j,k,l}+u^{i+1}_{j+1,k,l}+u^{i+1}_{j,k\pm1,l}+u^{i+1}_{j,k,l\pm1})=f^i_{j,k,l}+\frac{\Delta t}{h^2}u^{i+1}_{j-1,k,l}.
\end{equation}
Now consider every point $(x_j,y_k,z_l)$ in $M^+$ and we would have a modified version of \eqref{eqn:difference_equation_Mplus}:
\begin{equation}\label{eqn:modified_equation}
(I-\Delta t\tilde{\Delta}_h)u^{i+1}=\tilde{f}^i,\quad(x_j,y_k,z_l)\in M^+,
\end{equation}
where $\tilde{\Delta}_h$ is the modified Laplace operator as we discussed in \eqref{eqn:nonnegativity_in_rhs}. Note that, the non-negative boundary condition $u^{i+1}_{\gamma_{ex}}$ is already incorporated into the modified right hand side $\tilde{f}^i$. Also, the discrete equation \eqref{eqn:modified_equation} admits a unique solution, since \eqref{eqn:difference_equation_Mplus} supplemented with condition $u^{i+1}\equiv u^{i+1}_{\gamma_{ex}}$ on $\gamma_{ex}$ has a unique solution.

Note that, if the stencil $\mathcal{N}^7_{j,k,l}$ at point $(x_j,y_k,z_l)\in M^+$ has more than one point in $\gamma_{ex}$, we simply move more boundary terms to the right hand side, like in~\eqref{eqn:nonnegativity_in_rhs}. In other words, we always add non-negative boundary terms to the right hand side, so the modified right hand side $\tilde{f}^{i}$ will always be non-negative if Conditions~\ref{cond:previous_positivity}--\ref{cond:positive_gamma} of Theorem~\ref{thm:spherical_positivity} are satisfied. Also, note that we only remove off-diagonal entries in $\Delta_h$ to obtain the modified Laplace operator $\tilde{\Delta}_h$, hence $I-\Delta t\tilde{\Delta}_h$ is an M-matrix with a non-negative inverse, similar to $I-\Delta t\Delta_h$ as discussed in \cite{Chertock_2008,Epsh1,MR2002152}. Thus we conclude that the discrete system \eqref{eqn:difference_equation_Mplus} admits a unique, non-negative solution.

Finally, since $u^{i+1}$ (constructed using the Discrete Generalized Green's formula~\eqref{eqn:generalized_greens_formula}) is a solution to the discrete system \eqref{eqn:difference_equation_Mplus}, we conclude that $\bar{\rho}^{i+1}$ and $c^{i+1}$ are non-negative on $N^+$, if Conditions~\ref{cond:previous_positivity}--\ref{cond:positive_gamma} of Theorem~\ref{thm:spherical_positivity} are satisfied.
\end{proof}

\begin{remark} Theorem~\ref{thm:spherical_positivity} is not restricted to a spherical domain and can be applied to a general bounded domain $\Omega\subset\mathbb{R}^3$. In our algorithm, indeed when we require the density $u^{i+1}_{\gamma}$ to be non-negative, we observe no negative values in the discrete solutions, as can be seen in Fig. \ref{fig:isosurface-blowup-center} and Fig. \ref{fig:isosurface-blowup-boundary}, for example.

\end{remark}


\paragraph{An Outline of Main Steps of Algorithm based on DPM.}\label{par:main_algorithm_dpm}
Let us summarize the main steps of the proposed algorithm:
\begin{itemize}
    \item \textit{Step 1}: Outside of the time loop: embed the spherical domain $\Omega$ inside a larger cubic Auxiliary Domain $\Omega^0$ and formulate the Auxiliary Problem (AP) \eqref{eqn:discrete_ap}--\eqref{eqn:discrete_ap_boundary} on uniform meshes. Note that, the APs in our algorithm are solved using Fast Poisson Solvers.

    \item \textit{Step 2}: Outside of the time loop: construct the matrices $(I-P_{\gamma})\phi_\nu(\theta,\varphi)$ and $(I-P_{\gamma})(d^2/2)\phi_\nu(\theta,\varphi)$ ($\nu=0,1,\dots,M$) as a part of the BEP~\eqref{eqn:BEP_recast} via several solutions of the APs, using $\Delta t=0.5h^2$. Then precompute the QR decomposition of the BEP matrix in the left hand side of BEP~\eqref{eqn:BEP_recast}.

    \item \textit{Step 3}: Inside the time loop at each time level $t^{i+1}$: update the time step $\Delta t$ using the minimum between CFL-type condition~\eqref{eqn:cfl_condition} and $0.5h^2$.

    \item \textit{Step 4a}: At time level $t^{i+1}$: construct the Particular Solution $G_{h,\Delta t}f^{i}$ using the updated $\Delta t$ on $N^+$ for the density $\rho^{i+1}$ and concentration $c^{i+1}$ respectively.

    \item \textit{Step 4b}: At time level $t^{i+1}$: if the time step $\Delta t$ becomes smaller than the time step at previous time level $t^i$, recompute the BEP matrix using the same time step $\Delta t$ as in \textit{Step 4a}. Then, recompute the QR decomposition of the BEP matrix as in \textit{Step 2}. Otherwise, skip this step.

    \item \textit{Step 5}: At time level $t^{i+1}$: use the precomputed QR decomposition of the BEP matrix and solve BEP \eqref{eqn:BEP_recast} for the spectral coefficients $\{C_{\nu}^{0,i+1},C_{\nu}^{2,i+1}\}$ ($\nu=0,1,\dots,M$). Then reconstruct the density $u^{i+1}_{\gamma}$ using extension operator~\eqref{eqn:extension_operator_recast} with the obtained spectral coefficients.

    \item \textit{Step 6}: At time level $t^{i+1}$: construct the Difference Potentials $P_{N^+\gamma}u^{i+1}_{\gamma}$ using the same time step $\Delta t$ as in \textit{Step 4a} by solving the AP with the right hand side~\eqref{rhs:difference_potentials} and density $u^{i+1}_{\gamma}$ obtained in \textit{Step 5}.

    \item \textit{Step 7}: At time level $t^{i+1}$: construct the discrete solution $u^{i+1}$ ($\bar{\rho}^{i+1}$ or $c^{i+1}$) via the Discrete Generalized Green's formula~\eqref{eqn:generalized_greens_formula}.

    \item \textit{Step 8}: Repeat \textit{Steps 3--7} till final time or the $\infty$-norm of the discrete solution of $\bar{\rho}^{i+1}$ is above certain threshold.
\end{itemize}

\begin{remark}
When the solution $u^{i+1}$ ($\bar{\rho}^{i+1}$ or $c^{i+1}$) is sufficiently smooth, the time step stays $\Delta t=0.5h^2$ and
\textit{Step 4b} is skipped. \textit{Step 4b} is only required near blow-up, when the time step decreases due to the CFL-type condition~\eqref{eqn:cfl_condition}. Hence, the proposed algorithm based on Difference Potentials approach is very efficient for such problems.
\end{remark}


\section{A Domain Decomposition approach based on DPM}\label{sec:domain_decomposition}

Note that, the most computationally expensive step in our algorithm proposed in Section~\ref{par:main_algorithm_dpm} for a single spherical domain is \textit{Step 4b}, where we need to recompute the BEP \eqref{eqn:BEP_recast} by solving the APs as many times as the total number of basis functions. In particular, we need to recompute BEP~\eqref{eqn:BEP_recast} at every time level near blow-up time in \textit{Step 4b}, when the time step $\Delta t$ decreases due to the CFL-type condition~\eqref{eqn:cfl_condition}. In addition, as usual with any numerical algorithm, the computational cost of the entire algorithm increases significantly with global mesh refinement in 3D. However, the considered solutions ($\rho$ and $c$) to chemotaxis model \eqref{eqn:minimal_chemotaxis} have a compact support, which means global mesh refinement introduces unnecessary computation. Hence, to develop a more efficient scheme for chemotaxis models in 3D, we introduce the adaptivity in space that is compatible with Fast Poisson Solvers for the AP \eqref{eqn:discrete_ap} and \eqref{eqn:discrete_ap_boundary}.

To this end, we consider a domain decomposition approach based on DPM. Our proposed domain decomposition algorithm follows and extends the numerical algorithms for interface problems in 2D \cite{Albright_2017_elliptic,MR3659255,MR3817808}. In the domain decomposition approach, we decompose the spherical domain $\Omega$ into two non-intersecting sub-domains $\Omega_1$ and $\Omega_2$. Next, similarly as in Section~\ref{sec:dpm_sd}, we introduce computationally simple auxiliary domains $\Omega^0_\ell$ (cubic) and embed each sub-domain $\Omega_\ell$ into its corresponding auxiliary domain ($\ell=1,2$). We assume the large values or the non-smooth parts of the solutions to chemotaxis system~\eqref{eqn:minimal_chemotaxis} near blow-up are located in sub-domain $\Omega_1$ at any given time. Next, we discretize each auxiliary domain $\Omega^0_\ell$ using uniform Cartesian meshes of dimension $N_\ell\times N_\ell\times N_\ell$ and grid size $h_\ell$ ($\ell=1,2$).

\begin{remark}
By our assumption, the solutions in sub-domain $\Omega_2$ are smooth and changing slowly. Hence, we can introduce mesh adaptivity in space and use a much coarser mesh for sub-domain $\Omega_2$ without loss of global accuracy. Essentially, in contrast to the single domain approach, the domain decomposition approach reduces the degrees of freedom significantly while maintaining similar accuracy, as we will illustrate using numerical examples in Section~\ref{sec:numerical_results}.
\end{remark}

\begin{figure}
    \centering
    \begin{subfigure}{0.4\textwidth}
        \centering
        \includegraphics[width=\textwidth]{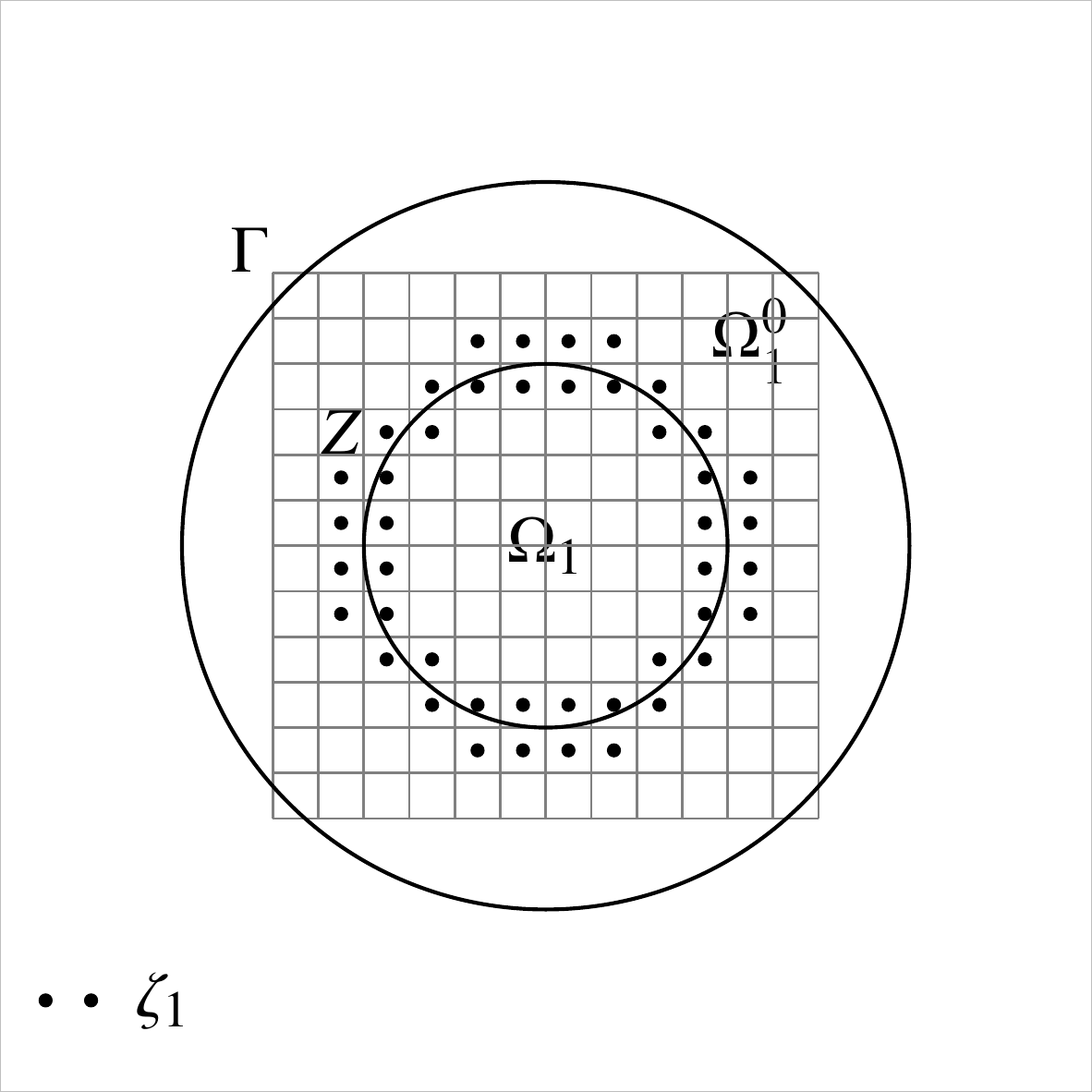}
        \caption{}
        \label{fig:DD1_Domain1}
    \end{subfigure}%
    ~
    \begin{subfigure}{0.4\textwidth}
        \centering
        \includegraphics[width=\textwidth]{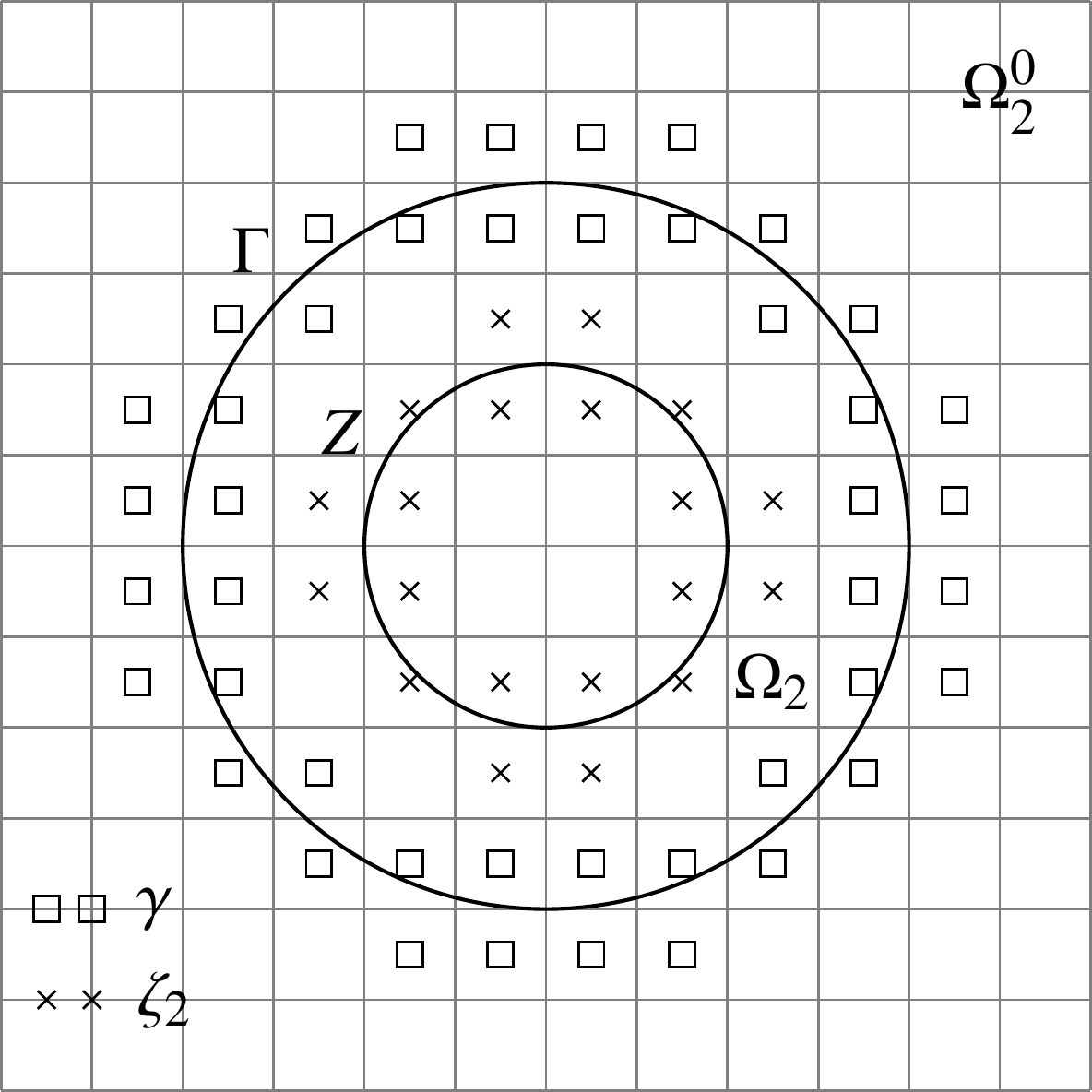}
        \caption{}
        \label{fig:DD1_Domain2}
    \end{subfigure}
    \caption{Example of discrete grid boundary/interface for Case 1 ($Z\cap\Gamma=\varnothing$) in the cross-sectional view: (a) the discrete grid interface $\zeta_1$ (solid dots along continuous interface $Z$) for sub-domain $\Omega_1$; and (b) the discrete grid boundary $\gamma$ (open squares along continuous boundary $\Gamma$ of the spherical domain $\Omega$) and the discrete grid interface $\zeta_2$ (cross dots along continuous interface $Z$), for sub-domain $\Omega_2$.}\label{fig:DD1_2D_example}
\end{figure}

We denote the artificial interface between the two sub-domains as $Z$. In this work, we consider the following two types of interfaces: $Z\cap\Gamma=\varnothing$ (Fig.~\ref{fig:DD1_2D_example}) and $Z\cap\Gamma\neq\varnothing$ (Fig.~\ref{fig:DD2_2D_example}), where $\Gamma:=\partial\Omega$ is the continuous boundary of the original spherical domain $\Omega$. Therefore, across the artificial interface $Z$, we impose the continuous interface conditions:
\begin{align}\label{eqn:interface_condition}
\left.\frac{\partial^ku^{i+1}_1}{\partial n^k}\right|_{Z}=\left.\frac{\partial^ku^{i+1}_2}{\partial n^k}\right|_{Z},\quad k=0,1,2,\dots
\end{align}
where $\frac{\partial^ku^{i+1}_\ell}{\partial n^k}$ denotes the $k$-th order normal derivative of $u^{i+1}_{\ell}$ at time $t^{i+1}$, $n$ denotes the unit normal vector on the interface $Z$, and $u^{i+1}_{\ell} :=\bar{\rho}^{i+1}_{\ell}$ or $u^{i+1}_{\ell} :=c^{i+1}_{\ell}$ in sub-domain $\Omega_\ell$ ($\ell=1,2$).

After we decompose the spherical domain $\Omega$ into two sub-domains, we proceed in each sub-domain as in the single domain approach (see Section \ref{sec:dpm_sd}). Again, we define the point sets $N^0_\ell,M^0_\ell,N^\pm_\ell,M^\pm_\ell$ $(\ell=1,2)$ as in Definition~\ref{def:point_sets} for each $\Omega_\ell$ and $\Omega^0_\ell$ ($\ell=1,2$). Next, we define the \textit{discrete grid boundary} along the continuous boundary $\Gamma$ and the \textit{discrete grid interface} along the continuous interface $Z$ in two cases respectively:

\emph{Case 1 of domain decomposition (suitable for the blow-up at the center of the domain)}: $Z\cap\Gamma=\varnothing$. The \textit{discrete grid boundary} $\gamma$ for sub-domain $\Omega_2$ (Fig.~\ref{fig:DD1_Domain2}) is defined as the intersection of $N^+_2$ and $N^-_2$ within a small neighborhood of the boundary $\Gamma$. The \textit{discrete grid interfaces} $\zeta_\ell$ are defined as the intersection of $N^+_\ell$ and $N^-_{\ell}$ ($\ell=1,2$) within a small neighborhood of the continuous interface $Z$ for sub-domain $\Omega_\ell$ ($\ell=1,2$). See Fig.~\ref{fig:DD1_2D_example} for example of $\gamma,\zeta_1,\zeta_2$ in cross-sectional view.

\emph{Case 2 of domain decomposition (suitable for the blow-up at the boundary of the domain)}: $Z\cap\Gamma\neq\varnothing$. In this case, we take the \textit{discrete grid boundary} $\gamma_1$ for sub-domain $\Omega_1$ to illustrate how to define the discrete grid boundary and the discrete grid interface. First,  define the ``discrete grid boundary'' $\gamma$ for the whole boundary $\Gamma$ (similarly to Case 1). Next, include a point  $p\in\gamma$ into the \textit{discrete grid boundary} $\gamma_1$ for sub-domain $\Omega_1$, if the polar angle of the orthogonal projection of $p$ on the continuous boundary $\Gamma$ belongs to $[0, \theta^*+\epsilon]$. Here, $\theta^*$ is the polar angle at the intersection points of $\Gamma$ and $Z$ and the tolerance $\epsilon$ is introduced to include extra layer of points near the ``wedge''. The discrete grid boundary $\gamma_2$ and the discrete grid interfaces $\zeta_1$ and $\zeta_2$ are defined similarly. See Fig.~\ref{fig:DD2_2D_example} for example of $\gamma_1,\gamma_2,\zeta_1,\zeta_2$ for geometry with wedges in cross-sectional view.

\begin{figure}
    \centering
    \begin{subfigure}{0.4\textwidth}
        \centering
        \includegraphics[width=\textwidth]{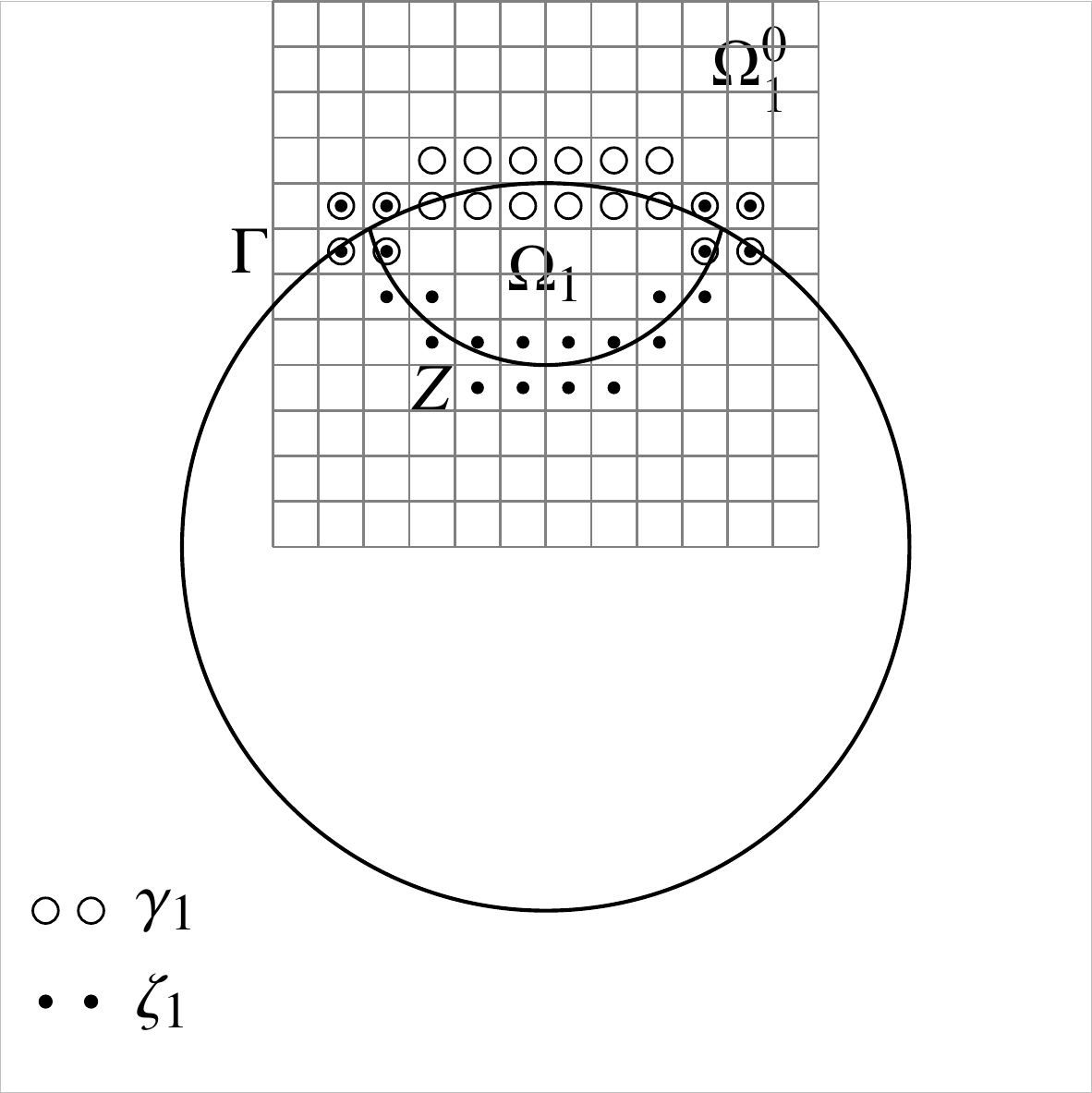}
        \caption{}
    \end{subfigure}
    ~
    \begin{subfigure}{0.4\textwidth}
        \centering
        \includegraphics[width=\textwidth]{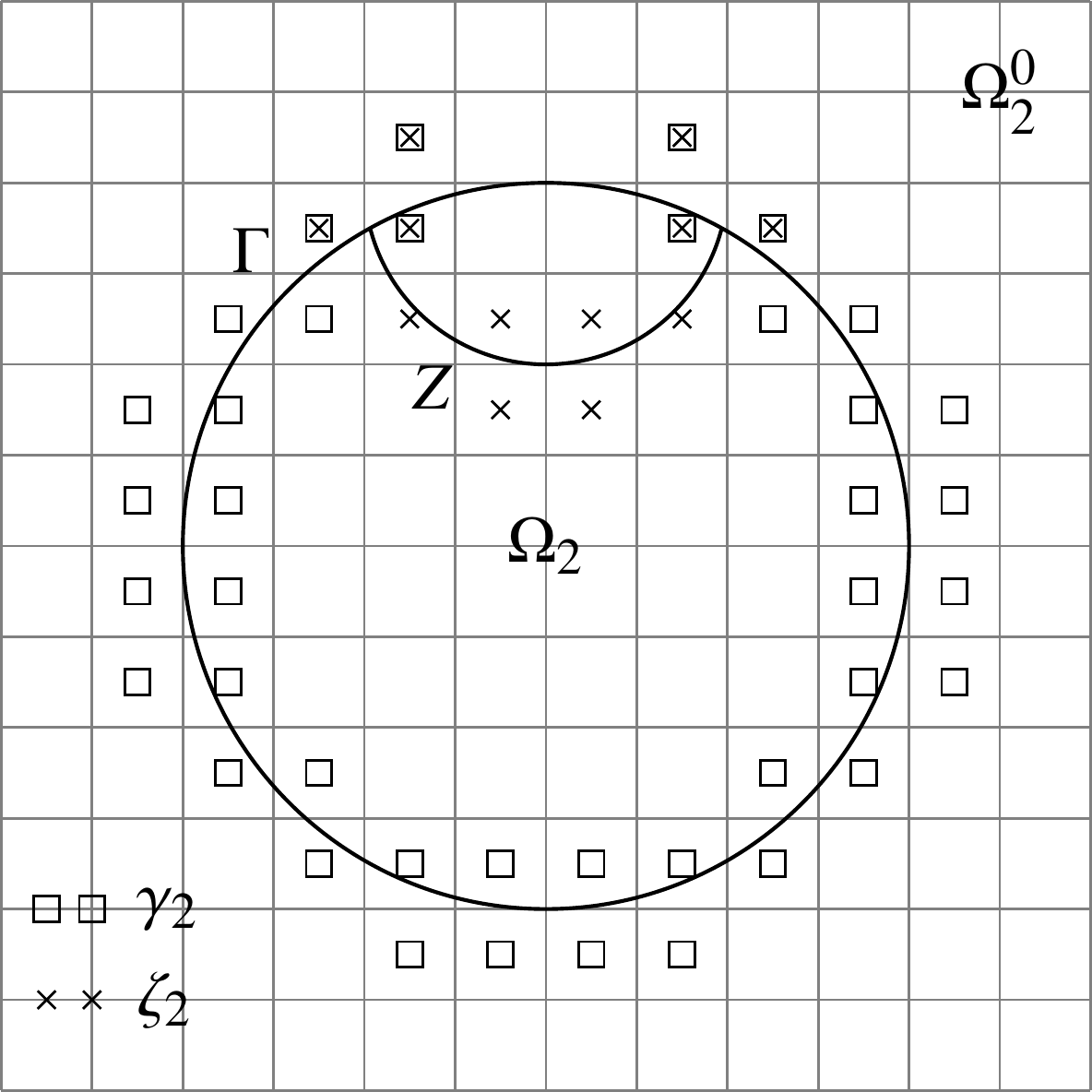}
        \caption{}
    \end{subfigure}
    \caption{Example of discrete grid boundary/interface for Case 2 ($Z\cap\Gamma\neq\varnothing$) in the cross-sectional view: (a) discrete grid boundary $\gamma_1$ (open circles along continuous boundary $\Gamma$ of the spherical domain) and discrete grid interface $\zeta_1$ (solid dots along continuous interface $Z$). Solid dots in open circles show the intersection of $\gamma_1$ and $\zeta_1$ sets; and (b) discrete grid boundary $\gamma_2$ (open squares along continuous boundary $\Gamma$ of the spherical domain) and discrete grid interface $\zeta_2$ (cross dots along continuous interface $Z$). Crosses in squares show the intersection of $\gamma_2$ and $\zeta_2$ sets.}\label{fig:DD2_2D_example}
\end{figure}


\paragraph{BEPs for Case 1: $Z\cap\Gamma=\varnothing$.} Note that, $\zeta_1$ is the ``discrete grid boundary'' for sub-domain $\Omega_1$ and the union $\gamma\cup\zeta_2$ constitutes the ``discrete grid boundary'' for sub-domain $\Omega_2$ (see Fig.~\ref{fig:DD1_2D_example}). Next, we treat each sub-domain as a single domain and recall Theorem~\ref{thm:full_BEP} to construct BEP \eqref{eqn:DD1_BEP1} for sub-domain $\Omega_1$ and BEP~\eqref{eqn:DD1_BEP2} for sub-domain $\Omega_2$:
\begin{align}
u^{i+1}_{1,\zeta_1}-P_{\zeta_1}u^{i+1}_{1,\zeta_1}&=G_{h_1,\Delta t}f^i_{1,\zeta_1},\label{eqn:DD1_BEP1}\\
\left(
\begin{array}{cc}
u^{i+1}_{2,\gamma}\\
u^{i+1}_{2,\zeta_2}
\end{array}
\right)
-P_{\gamma\cup\zeta_2}\left(
\begin{array}{cc}
u^{i+1}_{2,\gamma}\\
u^{i+1}_{2,\zeta_2}
\end{array}
\right)
&=\left(
\begin{array}{cc}
G_{h_2,\Delta t}f^i_{2,\gamma}\\
G_{h_2,\Delta t}f^i_{2,\zeta_2}
\end{array}
\right),\label{eqn:DD1_BEP2}
\end{align}
where the projection operators are defined as $P_{\zeta_1}:=Tr_{\zeta_1}P_{N^+_1\zeta_1}$ and $P_{\gamma\cup\zeta_2}:=Tr_{\gamma\cup\zeta_2}P_{N^+_2(\gamma\cup\zeta_2)}$. As discussed in Section~\ref{sec:dpm_sd} and for efficiency of our algorithms, we reduce the above coupled BEPs \eqref{eqn:DD1_BEP1} and \eqref{eqn:DD1_BEP2} to the interior of discrete grid boundary $\gamma_{in}$ and the discrete grid interface $\zeta_{1,in}$ and $\zeta_{2,in}$:
\begin{align}
u^{i+1}_{1,\zeta_{1,in}}-Tr_{\zeta_{1,in}}P_{\zeta_1}u^{i+1}_{1,\zeta_1}&=Tr_{\zeta_{1,in}}G_{h_1,\Delta t}f^i_{1,\zeta_1} \mbox{ on }\zeta_{1,in},\label{eqn:reduced_DD1_BEP1}\\
\left(
\begin{array}{cc}
u^{i+1}_{2,\gamma_{in}}\\
u^{i+1}_{2,\zeta_{2,in}}
\end{array}
\right)
-Tr_{\gamma_{in}\cup\zeta_{2,in}}P_{\gamma\cup\zeta_2}\left(
\begin{array}{cc}
u^{i+1}_{2,\gamma}\\
u^{i+1}_{2,\zeta_2}
\end{array}
\right)
&=\left(
\begin{array}{cc}
Tr_{\gamma_{in}}G_{h_2,\Delta t}f^i_{2,\gamma}\\
Tr_{\zeta_{2,in}}G_{h_2,\Delta t}f^i_{2,\zeta_2}
\end{array}
\right)\mbox{ on }\gamma_{in}\cup\zeta_{2,in}.\label{eqn:reduced_DD1_BEP2}
\end{align}
Next, we follow the single domain approach in Section~\ref{sec:dpm_sd} and supplement the BEPs \eqref{eqn:reduced_DD1_BEP1} and \eqref{eqn:reduced_DD1_BEP2} with extension operators and spectral approximations, to incorporate the zero Neumann boundary condition and the continuous interface condition~\eqref{eqn:interface_condition}. Then we solve the coupled BEPs \eqref{eqn:reduced_DD1_BEP1} and \eqref{eqn:reduced_DD1_BEP2} and obtain the densities $u^{i+1}_{\zeta_1}$, $u^{i+1}_{\gamma}$ and $u^{i+1}_{\zeta_2}$. Finally, the discrete solution $u^{i+1}_\ell$ in each sub-domain $\Omega_\ell$ $(\ell=1,2)$ at time level $t^{i+1}$ is constructed using the Discrete Generalized Green's formula \eqref{eqn:generalized_greens_formula} for each sub-domain.


\paragraph{BEPs for Case 2: $Z\cap\Gamma\neq\varnothing$.} Similarly to Case 1, note that $\gamma_1$ and $\zeta_1$ constitute the ``discrete grid boundary'' for sub-domain $\Omega_1$, and $\gamma_2$ and $\zeta_2$ form the ``discrete grid boundary'' for sub-domain $\Omega_2$.
Next, we treat each sub-domain as a single domain and recall Theorem~\ref{thm:full_BEP} to formulate the BEP for sub-domain $\Omega_\ell$:
\begin{align}
u^{i+1}_{\gamma_\ell\vee\zeta_\ell}
-P_{\gamma_\ell\vee\zeta_\ell}u^{i+1}_{\gamma_\ell\vee\zeta_\ell}
&=\left(
\begin{array}{cc}
G_{h_\ell,\Delta t}f^i_{\ell,\gamma_\ell}\\
G_{h_\ell,\Delta t}f^i_{\ell,\zeta_\ell}
\end{array}
\right),\quad (\ell=1,2)\label{eqn:DD2_BEP}
\end{align}
where $u^{i+1}_{\gamma_\ell\vee\zeta_\ell}$ denotes the densities defined on both sets $\gamma_\ell$ and $\zeta_\ell$, and $P_{\gamma_\ell\vee\zeta_\ell}$ is defined as:
\begin{align}
P_{\gamma_\ell\vee\zeta_\ell}u^{i+1}_{\gamma_\ell\vee\zeta_\ell}:= Tr_{\gamma_\ell\vee\zeta_\ell}P_{N^+_\ell(\gamma_\ell\vee\zeta_\ell)}u^{i+1}_{\gamma_\ell\vee\zeta_\ell},\quad (\ell=1,2)
\end{align}

\begin{remark}
Note that, any point $p\in\gamma_\ell\cap\zeta_\ell$ contributes two values in the density $u^{i+1}_{\gamma_\ell\vee\zeta_\ell}$: one from density $u^{i+1}_{\gamma_\ell}$ and the other from density $u^{i+1}_{\zeta_\ell}$ ($\ell=1,2$). Next, we follow~\cite{Magura_2017} and take the average of the two values corresponding to point $p$ and assign the average to be the ``effective'' density at point $p$. The ``effective'' density will be used in computing the Difference Potentials $P_{N^+_\ell(\gamma_\ell\vee\zeta_\ell)}u_{\gamma_\ell\vee\zeta_\ell}$. Also, note that we impose two equations at point $p\in \gamma_\ell\cap\zeta_\ell$ in the BEPs \eqref{eqn:DD2_BEP}: one corresponds to $u^{i+1}_{\gamma_\ell}$ and the other for $u^{i+1}_{\zeta_\ell}$ ($\ell=1,2$).
\end{remark}

Again for efficiency of our algorithms, we reduce the coupled BEPs \eqref{eqn:DD2_BEP} to the interior of the discrete grid boundary $\gamma_{\ell,in}$ and the interior of the discrete grid interface $\zeta_{\ell,in}$:
\begin{align}
u^{i+1}_{\gamma_{\ell,in}\vee\zeta_{\ell,in}}
-Tr_{\gamma_{\ell,in}\vee\zeta_{\ell,in}}P_{\gamma_\ell\vee\zeta_\ell}u^{i+1}_{\gamma_\ell\vee\zeta_\ell}
&=\left(
\begin{array}{cc}
Tr_{\gamma_{\ell,in}}G_{h_\ell,\Delta t}f^i_{\ell,\gamma_\ell}\\
Tr_{\zeta_{\ell,in}}G_{h_\ell,\Delta t}f^i_{\ell,\zeta_\ell}
\end{array}
\right),\quad(\ell=1,2)\label{eqn:reduced_DD2_BEP}
\end{align}
on $\gamma_{\ell,in}$ and $\zeta_{\ell,in}$.

Similarly, we follow the single domain approach in Section~\ref{sec:dpm_sd} and supplement the BEPs \eqref{eqn:reduced_DD2_BEP} with the extension operators and spectral approximations, to incorporate the zero Neumann boundary condition and the continuous interface condition~\eqref{eqn:interface_condition}. After we solve the coupled BEPs \eqref{eqn:reduced_DD2_BEP} and obtain densities $u^{i+1}_{\gamma_1}$, $u^{i+1}_{\zeta_1}$, $u^{i+1}_{\gamma_2}$ and $u^{i+1}_{\zeta_2}$, the discrete solution $u^{i+1}_\ell$ in each sub-domain $\Omega_\ell$ $(\ell=1,2)$ at time level $t^{i+1}$ is constructed using the Discrete Generalized Green's formula \eqref{eqn:generalized_greens_formula} for each sub-domain.

\begin{remark}
(i) In both Case 1 and Case 2, the non-negativity of discrete solutions $u^{i+1}_1$ in sub-domain $\Omega_1$ and $u^{i+1}_2$ in sub-domain $\Omega_2$ follows from Theorem~\ref{thm:spherical_positivity}: $u^{i+1}_{1}$ and $u^{i+1}_{2}$ are non-negative, provided that Conditions~\ref{cond:previous_positivity}--\ref{cond:positive_gamma} in Theorem \ref{thm:spherical_positivity} are satisfied in each sub-domain $\Omega_\ell$ ($\ell=1,2$). Moreover, we expect to recover second order convergence in space and first order convergence in time in each sub-domain as in Proposition~\ref{prop:discrete_gene_green_formula}, which indeed we observe in the numerical tests (see Section~\ref{sec:numerical_results}).\\
(ii) In addition, note that, in general, the major steps of the proposed domain decomposition algorithms will not change with a choice of subdomains $\Omega_1$ and $\Omega_2$. In the considered numerical examples in Section~\ref{sec:numerical_results}, we selected $\Omega_1$ and $\Omega_2$ for the domain decomposition approach to ensure similar accuracy with the single domain approach at a less computational cost. Furthermore, the developed domain decomposition approach handles non-trivial geometries of the subdomains by the use of simple Cartesian meshes that do not match at the interface.
\end{remark}


\section{Numerical Results}\label{sec:numerical_results}

Throughout this section, we assume the single spherical domain $\Omega$ is centered at the origin and is of radius $r=0.5$. Also, to illustrate the idea of domain decomposition approach, we assume the sub-domains $\Omega_1$ and $\Omega_2$ are non-intersecting parts of the spherical domain, and $\overline{\Omega}=\overline{\Omega}_1\cup\overline{\Omega}_2$ with $r_1=0.25$ for sub-domain $\Omega_1$ and $r_2=0.5$ for sub-domain $\Omega_2$ for both Cases 1 and 2 (see Section~\ref{sec:domain_decomposition}). We also assume $\chi=1$ in \eqref{eqn:minimal_chemotaxis} in our numerical tests.


\paragraph{Initial Conditions.} We will study the following two examples of test problems to demonstrate the accuracy, efficiency and robustness of our proposed algorithms:
\begin{itemize}
  \item Initial conditions of $\rho$ and $c$ that will give blow-up at the center of the spherical domain:
  \begin{equation}\tag{A}\label{test:blowup_at_center}
  \begin{aligned}
  \rho_0&=1000e^{-100(x^2+y^2+z^2)},\\
  c_0&=500e^{-50(x^2+y^2+z^2)}.
  \end{aligned}
  \end{equation}
  \item Initial condition of $\rho$ and $c$ that will give blow-up at the boundary of the spherical domain:
  \begin{equation}\tag{B}\label{test:blowup_at_boundary}
  \begin{aligned}
  \rho_0&=2000e^{-100(x^2+y^2+(z-0.25)^2)},\\
  c_0&=0.
  \end{aligned}
  \end{equation}
\end{itemize}

\begin{remark}
Test~\eqref{test:blowup_at_center} categorizes the typical radial symmetric initial condition in a radially symmetric domain, which leads to blow-up at the center. Test~\eqref{test:blowup_at_boundary} is more challenging in the following sense: (i) the evolution of the density takes longer to develop into almost singular solutions, which requires an efficient time stepping strategy to handle the longer-time simulations; (ii) the time step $\Delta t$ decreases due to CFL-type condition~\eqref{eqn:cfl_condition} near blow-up time; (iii) the blow-up occurs at the boundary of the irregular domain, which is more difficult to resolve, compared to Test~\eqref{test:blowup_at_center}.
\end{remark}


\paragraph{Choice of Meshes.}  We illustrate the choices of uniform Cartesian meshes using the single domain approach, and the domain decomposition approach follows similarly. In the single domain approach, we take the auxiliary domain to be a cubic domain: $[-r-2h,r+2h]\times[-r-2h,r+2h]\times[-r-2h,r+2h]$. Here $h:=2r/(N-4)$, where $r$ is the radius of the spherical domain and $N$ denotes that the cubic auxiliary domain is discretized by mesh of dimension $N\times N\times N$, and $N$ is the number of cells in $x$ (or $y$, $z$) direction. Note that, the choice of auxiliary domain here is only a few layers larger than the spherical domain $\Omega$ to have efficient computations.

Next we introduce the notations for meshes in the single domain (SD) and in the domain decomposition (DD) approaches:
\begin{itemize}
    \item ``$N$'' denotes that the mesh for $\Omega$ in the SD approach is dimension $N\times N\times N$ cells, with grid size $h=2r/(N-4)$;
    \item ``$N_1/N_2$'' denotes that the meshes for sub-domains $\Omega_1$ and $\Omega_2$ in the DD approach are dimension $N_1\times N_1\times N_1$ cells and $N_2\times N_2\times N_2$ cells respectively, with $h_1=2r_1/(N_1-4)$ for the sub-domain $\Omega_1$ and $h_2=2r_2/(N_2-4)$ for the sub-domain $\Omega_2$.
\end{itemize}
These choices of uniform Cartesian meshes are key to using Fast Poisson Solvers in 3D. We employed Fast Poisson Solver based on FFTW3 library~\cite{Frigo_1999} with openMP which enabled for better parallel efficiency. For other choices and comparison of high performance Fast Poisson Solvers, see \cite{Gholami_2016}.


\paragraph{Basis functions, extension operators and BEPs.} As mentioned in Section~\ref{sec:dpm_sd}, the basis functions are selected to be the zonal spherical harmonics in the form of:
\begin{align}
\phi_{\nu}(\theta,\varphi) = P^0_{\nu}(\cos\theta),\quad \nu=0,1,\dots,M.
\end{align}
We use a different number of zonal harmonics in the spectral approximations along the boundary and along the interface for Test~\eqref{test:blowup_at_center} and Test~\eqref{test:blowup_at_boundary}. For Test~\eqref{test:blowup_at_center}, only 1 zonal harmonic is needed for each term in the 3-term extension operator along the boundary or the interface, due to: (i) the radial symmetry of the initial data and the solutions at later time levels; and (ii) the symmetry of the spherical domain. 

For Test~\eqref{test:blowup_at_boundary}, we use 20 harmonics for each term in the 3-term extension operator along the continuous interface $Z$, since we expect the solutions to be smooth along the interface $Z$ at any time. However, we vary the number of harmonics used for each term in the 2-term extension operator along the boundary $\Gamma$ in both the SD and DD approaches, and the number of harmonics used will be specified in the numerical results. The 2-term extension operators are used along the boundary for Test~\eqref{test:blowup_at_boundary} since solution $\rho$ loses regularity at the boundary due to blow-up at the boundary.

We should also note that all results from the SD approach are obtained using the reduced BEP~\eqref{eqn:reduced_BEP} and results from the DD approach are obtained using the reduced BEPs~\eqref{eqn:reduced_DD1_BEP1}--\eqref{eqn:reduced_DD1_BEP2} or \eqref{eqn:reduced_DD2_BEP} with $\ell=1,2$.


\subsection{Comparisons between single domain (SD) and domain decomposition (DD) approaches}
In this subsection, we compare and contrast the numerical results obtained from the SD approach and the DD approach using Test~\eqref{test:blowup_at_center} in the following aspects:
\begin{itemize}
    \item Convergence in space: Since there is no exact solution to the chemotaxis system~\eqref{eqn:minimal_chemotaxis}, we use the discrete solutions on a finer mesh to construct the reference solutions. The $\infty$-norm error for the density $\rho$ in space is computed as:
    \begin{align}
        E_{\infty}&=\max_{j,k,l}\left|\mathds{1}_{M^+}(x_j,y_k,z_l)(\bar{\rho}^i_h-\bar{\rho}^i_{h^*})\right|,\quad \mbox{(SD)}\label{eqn:sd-error-in-max-norm}\\
        E_{\infty}&=\max_{j,k,l}\left|\mathds{1}_{M^+_\ell}(x_j,y_k,z_l)(\bar{\rho}^i_{h_\ell}-\bar{\rho}^i_{h^*_\ell})\right|,\quad \ell=1,2 \quad\mbox{(DD)}\label{eqn:dd-error-in-max-norm}
    \end{align}
    where $\bar{\rho}^i_{h^*}$ and $\bar{\rho}^i_{h^*_\ell}$ ($\ell=1,2$) are reference solutions with grid sizes $h^*$ and $h^*_\ell$ in the SD and DD approaches respectively, and the function $\mathds{1}_{S}$ denotes the characteristic function for a point set $S$:
    \begin{align}\label{fcn:Mplus}
        \mathds{1}_{S}(x_j,y_k,z_l)=\left\{
        \begin{array}{ll}
        1,&\quad (x_j,y_k,z_l)\in S,\\
        0,&\quad (x_j,y_k,z_l)\notin S.
        \end{array}\right.
    \end{align}
    Here $S=M^+$ in the SD approach, or $S=M^+_\ell$ ($\ell=1,2$) in the DD approach. The $\infty$-norm error of the chemoattractant concentration $c$ in space for the SD or DD approach is computed similarly.

    \item $L^2$-norm of the relative error in time of the maximum value of the density $\rho$:
    \begin{align}
        E_{rel}&=\dfrac{\left(\sum_{i=1}^{N_T}(||\bar{\rho}^i_h||_{\infty}-||\bar{\rho}^i_{h^*}||_{\infty})^2\Delta t\right)^{1/2}}{\left(\sum_{i=1}^{N_T}(||\bar{\rho}^i_{h^*}||_{\infty})^2\Delta t\right)^{1/2}},\quad \mbox{(SD)}\label{eqn:sd-conv-max-density-in-space}\\
        E_{rel}&=\dfrac{\left(\sum_{i=1}^{N_T}(\max\limits_{\ell=1,2}\{||\bar{\rho}^i_{h_\ell}||_{\infty}\}-\max\limits_{\ell=1,2}\{||\bar{\rho}^i_{h^*_\ell}||_{\infty}\})^2\Delta t\right)^{1/2}}{\left(\sum_{i=1}^{N_T}(\max\limits_{\ell=1,2}\{||\bar{\rho}^i_{h^*_\ell}||_{\infty}\})^2\Delta t\right)^{1/2}},\quad \mbox{(DD)}\label{eqn:dd-conv-max-density-in-space}
    \end{align}
    where $N_T$ is the total number of time steps taken, and  $||\cdot||_{\infty}$ is the maximum value of $\rho$ in space at time $t^i$. Note that, the time steps $\Delta t$ in \eqref{eqn:sd-conv-max-density-in-space} and \eqref{eqn:dd-conv-max-density-in-space} are set to be the same in our numerical results, see also Table~\ref{table:convergence-rho_max-in-space}.

    \item Convergence in time:
    We test convergence in time by fixing the meshes and refining the time steps: $\Delta t=\Delta T/\tau$, where $\Delta T$ is a constant chosen to be $10^{-7}$ and $\tau$ is a multiple of 2. Then, the error of the density $\rho$ at the final time $t^i=10^{-6}$ is computed as:
    \begin{align}
        E_t&=\max_{j,k,l}\left|\mathds{1}_{M^+}(x_j,y_k,z_l)(\bar{\rho}^i_{\Delta t}-\bar{\rho}^i_{\Delta t*})\right|,\quad \mbox{(SD)}\label{eqn:sd-conv-in-time}\\
        E_t&=\max_{j,k,l}\left|\mathds{1}_{M_\ell^+}(x_j,y_k,z_l)(\bar{\rho}^i_{\ell,\Delta t}-\bar{\rho}^i_{\ell,\Delta t*})\right|,\quad\ell=1,2\quad \mbox{(DD)}\label{eqn:dd-conv-in-time}
    \end{align}
    where $\bar{\rho}^i_{\Delta t*}$ and $\bar{\rho}^i_{\ell,\Delta t*}$ are the reference solutions at the same final time with the time step $\Delta t^*$ in the SD and DD approaches respectively. The convergence in time for the concentration $c$ at the final time $t^{i}=10^{-6}$ for the SD or the DD approach is computed similarly.

    \item Max density at time $t^{i}$:
    \begin{align}
        ||\bar{\rho}^i||_{\infty}&=\max_{j,k,l} \left|\mathds{1}_{M^+}(x_j,y_k,z_l)\bar{\rho}^i_{j,k,l}\right|,\quad \mbox{(SD)}\label{sd-norm:max}\\
        ||\bar{\rho}^i||_{\infty}&=\max_{\ell=1,2}\max_{j,k,l} \left|\mathds{1}_{M^+_\ell}(x_j,y_k,z_l)\{\bar{\rho}^i_{\ell}\}_{j,k,l}\right|, \quad \mbox{(DD)}\label{dd-norm:max}
    \end{align}
    which we will use to approximate the blow-up time.

    \item Second moment of $\rho$ at time $t^{i}$ (see \cite{2017arXiv170903931J}):
    \begin{align}
        M^i_h&=\sum_{j,k,l=1}^N\mathds{1}_{M^+}(x_j,y_k,z_l)h^3(x_j^2+y_k^2+z_l^2)\bar{\rho}^i_{j,k,l},\quad\mbox{(SD)}\label{sd-norm:second-moment}\\
        M^i_h&=\sum_{\ell=1}^2\sum_{j,k,l=1}^{N_\ell}\mathds{1}_{M^+_\ell}(x_j,y_k,z_l)h^3_\ell(x_j^2+y_k^2+z_l^2)\{\bar{\rho}^i_\ell\}_{j,k,l} \quad\mbox{(DD)}\label{dd-norm:second-moment}
    \end{align}

    \item Free energy (see \cite{2017arXiv170903931J,Per}) of the chemotactic system \eqref{eqn:minimal_chemotaxis} at time $t^{i}$:
    \begin{align}
        \mathcal{E}^i_h&=\sum_{j,k,l=1}^N\mathds{1}_{M^+}h^3\Big\{\bar{\rho}^i_{j,k,l}\ln\bar{\rho}^i_{j,k,l}-\bar{\rho}^i_{j,k,l}c^i_{j,k,l}+\frac{1}{2}(c^i_{j,k,l})^2\nonumber\\
        &\quad+\frac{1}{8h^2}\big[(c^i_{j+1,k,l}-c^i_{j-1,k,l})^2+(c^i_{j,k+1,l}-c^i_{j,k-1,l})^2\nonumber\\
        &\quad+(c^i_{j,k,l+1}-c^i_{j,k,l-1})^2\big]\Big\},\quad \mbox{(SD)}\label{sd-norm:free-energy}\\
        \mathcal{E}^i_h&=\sum_{\ell=1}^2\sum_{j,k,l=1}^{N_\ell}\mathds{1}_{M^+_\ell}h_\ell^3\Big\{\{\bar{\rho}^i_\ell\}_{j,k,l}\ln\{\bar{\rho}^i_\ell\}_{j,k,l}-\{\bar{\rho}^i_\ell\}_{j,k,l}\{c^i_\ell\}_{j,k,l}+\frac{1}{2}(\{c^i_\ell\}_{j,k,l})^2\nonumber\\
        &\quad+\frac{1}{8h_\ell^2}\big[(\{c^i_\ell\}_{j+1,k,l}-\{c^i_\ell\}_{j-1,k,l})^2+(\{c^i_\ell\}_{j,k+1,l}-\{c^i_\ell\}_{j,k-1,l})^2\nonumber\\
        &\quad+(\{c^i_\ell\}_{j,k,l+1}-\{c^i_\ell\}_{j,k,l-1})^2\big]\Big\} \quad \mbox{(DD)}\label{dd-norm:free-energy}
    \end{align}
    In particular, we expect to see a decrease in free energy over time according to the second law of thermodynamics.

\end{itemize}


\subsubsection{Comparison of the convergence at time $10^{-6}$}
In this subsection, we present the comparison of convergence studies both in space and in time for the SD and the DD approaches using Test \eqref{test:blowup_at_center}. The final time for the convergence test is set to be $10^{-6}$ so that (i) the solutions $\rho$ and $c$ are sufficiently smooth; and (ii) the reference solution on the finest mesh $516\times516\times516$ in the SD approach can be obtained within a reasonable wall-clock time.

\begin{remark}
Note that the single domain approach is more computationally expensive than the domain decomposition approach.
\end{remark}

\paragraph{Convergence in Space.}

\begin{table}
\caption{Comparison of the errors in space (as computed in~\eqref{eqn:sd-error-in-max-norm} and~\eqref{eqn:dd-error-in-max-norm}) and the order of the convergence between SD and DD approaches to approximate $\rho$ in Test~\eqref{test:blowup_at_center} at time $10^{-6}$, with fixed time step $\Delta t=10^{-8}$. The reference solutions in both SD and DD approaches are obtained using the discrete solution on mesh $516\times516\times516$ in the SD approach.}
\label{table:convergence-rho-max-norm}
\centering
\footnotesize
\sisetup{
  output-exponent-marker = \text{ E},
  exponent-product={},
  retain-explicit-plus
}
\begin{tabular}{c S[table-format=1.4e2] c c S[table-format=1.4e2] c S[table-format=1.4e2] c}
\toprule
\multicolumn{3}{c}{Single Domain} & \multicolumn{5}{c}{Domain Decomposition ($h_1/h_2=1/2$)} \\
\cmidrule(lr){1-3}\cmidrule(lr){4-8}
{$N$} & {$E_\infty$: $\Omega$} & {Rate: $\Omega$} & {$N_1/N_2$} & {$E_{\infty}$: $\Omega_1$} & {Rate: $\Omega_1$} & {$E_{\infty}$: $\Omega_2$} & {Rate: $\Omega_2$} \\
\cmidrule(lr){1-3}\cmidrule(lr){4-8}
 36 & 1.4046e+00 & \textemdash & 20/20   & 1.4046e+00 & \textemdash & 3.5954e-02 & \textemdash \\
 68 & 3.6990e-01 & 1.93        & 36/36   & 3.6990e-01 & 1.93        & 8.3515e-03 & 2.11 \\
132 & 1.2224e-01 & 1.60        & 68/68   & 1.2224e-01 & 1.60        & 2.6776e-03 & 1.64 \\
260 & 2.7815e-02 & 2.14        & 132/132 & 2.7815e-02 & 2.14        & 7.7097e-04 & 1.80 \\
\cmidrule(lr){1-3}\cmidrule(lr){4-8}
\multicolumn{3}{c}{} & \multicolumn{5}{c}{Domain Decomposition ($h_1/h_2=1/4$)} \\
\cmidrule(lr){4-8}
      &     &                  & {$N_1/N_2$} & {$E_\infty$: $\Omega_1$} & {Rate: $\Omega_1$} & {$E_\infty$: $\Omega_2$} & {Rate: $\Omega_2$} \\
\cmidrule(lr){4-8}
    &            &             & 20/12   & 1.4046e+00 & \textemdash & 1.1226e-01 & \textemdash \\
    &            &             & 36/20   & 3.6990e-01 & 1.93        & 3.7364e-02 & 1.59 \\
    &            &             & 68/36   & 1.2224e-01 & 1.60        & 1.0093e-02 & 1.89 \\
    &            &             & 132/68  & 2.7815e-02 & 2.14        & 3.2778e-03 & 1.62 \\
\bottomrule
\end{tabular}
\end{table}

\begin{table}
\caption{Comparison of the errors in space (computed similarly to~\eqref{eqn:sd-error-in-max-norm} and~\eqref{eqn:dd-error-in-max-norm}) and the order of the convergence between SD and DD approaches to approximate $c$ in Test~\eqref{test:blowup_at_center} at time $10^{-6}$, with fixed time step $\Delta t=10^{-8}$. The reference solutions in both SD and DD approaches are obtained using the discrete solution on mesh $516\times516\times516$ in the SD approach.}
\label{table:convergence-c-max-norm}
\centering
\footnotesize
\sisetup{
  output-exponent-marker = \text{ E},
  exponent-product={},
  retain-explicit-plus
}
\begin{tabular}{c S[table-format=1.4e2] c c S[table-format=1.4e2] c S[table-format=1.4e2] c}
\toprule
\multicolumn{3}{c}{Single Domain} & \multicolumn{5}{c}{Domain Decomposition ($h_1/h_2=1/2$)} \\
\cmidrule(lr){1-3}\cmidrule(lr){4-8}
{$N$} & {$E_\infty$: $\Omega$} & {Rate: $\Omega$} & {$N_1/N_2$} & {$E_\infty$: $\Omega_1$} & {Rate: $\Omega_1$} & {$E_\infty$: $\Omega_2$} & {Rate: $\Omega_2$} \\
\cmidrule(lr){1-3}\cmidrule(lr){4-8}
 36 & 5.6759e+00 & \textemdash & 20/20   & 5.6759e+00 & \textemdash & 1.0804e+00 & \textemdash \\
 68 & 1.4766e+00 & 1.94        & 36/36   & 1.4766e+00 & 1.94        & 2.8497e-01 & 1.92 \\
132 & 3.5615e-01 & 2.05        & 68/68   & 3.5615e-01 & 2.05        & 7.1525e-02 & 1.99 \\
260 & 7.1459e-02 & 2.32        & 132/132 & 7.1459e-02 & 2.32        & 1.7233e-02 & 2.05 \\
\cmidrule(lr){1-3}\cmidrule(lr){4-8}
\multicolumn{3}{c}{} & \multicolumn{5}{c}{Domain Decomposition ($h_1/h_2=1/4$)} \\
\cmidrule(lr){4-8}
    &               &                  & {$N_1/N_2$} & {$E_\infty$: $\Omega_1$} & {Rate: $\Omega_1$} & {$E_\infty$: $\Omega_2$} & {Rate: $\Omega_2$} \\
\cmidrule(lr){4-8}
    &            &             & 20/12   & 5.6759e+00 & \textemdash & 3.4957e+00 & \textemdash \\
    &            &             & 36/20   & 1.4766e+00 & 1.94        & 1.0809e+00 & 1.69 \\
    &            &             & 68/36   & 3.5615e-01 & 2.05        & 2.8524e-01 & 1.92 \\
    &            &             & 132/68  & 7.1459e-02 & 2.32        & 7.1702e-02 & 1.99 \\
\bottomrule
\end{tabular}
\end{table}

Observe that in Table \ref{table:convergence-rho-max-norm} and Table \ref{table:convergence-c-max-norm} the errors in sub-domain $\Omega_1$ from both DD ($h_1/h_2=1/2$) and DD ($h_1/h_2=1/4$) are identical to those from using SD approach, if the grid spacing $h$ and $h_1$ are the same. Also, the errors in sub-domain $\Omega_2$ are less in magnitude than errors in sub-domain $\Omega_1$, even though the meshes for sub-domain $\Omega_2$ are coarser than meshes for sub-domain $\Omega_1$. This shows that our strategy of using domain decomposition as in Fig.~\ref{fig:DD1_2D_example} is effective and the error in the entire domain $\Omega$ is dominated by the error from $\Omega_1$. In addition, the errors in sub-domain $\Omega_2$ from using DD ($h_1/h_2=1/2$) are smaller than using DD ($h_1/h_2=1/4$), since DD ($h_1/h_2=1/2$) employs finer mesh for sub-domain $\Omega_2$. Overall, errors from sub-domain $\Omega_1$ dominate the errors in $\Omega$, and second order convergence in space are observed as expected in the SD approach, DD ($h_1/h_2=1/2$) or DD ($h_1/h_2=1/4$) for both $\rho$ and $c$.


\paragraph{Convergence of max density in space.}

\begin{table}
\caption{Convergence of relative errors of $\rho$ (as computed in \eqref{eqn:sd-conv-max-density-in-space} and \eqref{eqn:dd-conv-max-density-in-space}) for Test \eqref{test:blowup_at_center}. The results are obtained using fixed time step $\Delta t=10^{-8}$ to final time $10^{-6}$. The reference solutions for both SD and DD approaches are obtained using the discrete solution on mesh $516\times516\times516$ in the SD approach.}
\label{table:convergence-rho_max-in-space}
\centering
\footnotesize
\sisetup{
  output-exponent-marker = \text{ E},
  exponent-product={},
  retain-explicit-plus
}
\begin{tabular}{c S[table-format=1.4e2] c c S[table-format=1.4e2] c c S[table-format=1.4e2] c}
\toprule
\multicolumn{3}{c}{SD} & \multicolumn{3}{c}{DD ($h_1/h_2=1/2$)} & \multicolumn{3}{c}{DD ($h_1/h_2=1/4$)} \\
\cmidrule(lr){1-3}\cmidrule(lr){4-6} \cmidrule(lr){7-9}
{$N$} & {$E_{rel}$} & {Rate} & {$N_1/N_2$} & {$E_{rel}$} & {Rate} & {$N_1/N_2$} & {$E_{rel}$} & {Rate} \\
\cmidrule(lr){1-3}\cmidrule(lr){4-6} \cmidrule(lr){7-9}
 36 & 1.0196e-01 & \textemdash & 20/20   & 1.0196e-01 & \textemdash & 20/12   & 1.0196e-01 & \textemdash \\
 68 & 2.6378e-02 & 1.95        & 36/36   & 2.6378e-02 & 1.95        & 36/20   & 2.6378e-02 & 1.95 \\
132 & 6.3461e-03 & 2.06        & 68/68   & 6.3461e-03 & 2.06        & 68/36   & 6.3461e-03 & 2.06 \\
260 & 1.2724e-03 & 2.32        & 132/132 & 1.2724e-03 & 2.32        & 132/68  & 1.2724e-03 & 2.32 \\
\bottomrule
\end{tabular}
\end{table}

In Table~\ref{table:convergence-rho_max-in-space}, we observe identical relative errors and second order convergence among SD, DD ($h_1/h_2=1/2$) and DD ($h_1/h_2=1/4$), which shows the robustness of the domain decomposition approach.


\paragraph{Convergence in Time.}

In Table~\ref{table:convergence-rho-in-time} and Table \ref{table:convergence-c-in-time}, we observe first order convergence in time for SD approach and for both sub-domains in DD approach as we expected. The errors in sub-domain $\Omega_1$ are similar to errors in the single domain $\Omega$, regardless of the ratio $h_1/h_2=1/2$ or $h_1/h_2=1/4$. Although the meshes for $\Omega_2$ are coarser than meshes for $\Omega_1$, the errors in sub-domain $\Omega_2$ are much smaller than errors in sub-domain $\Omega_1$. We should note that errors from using SD approach and from using DD approach are similar, but are not exactly the same, because the mesh sizes are slightly different, i.e. $h_1/h_2\approx1/2$ as opposed to $h_1/h_2=1/2$. Overall, the errors from sub-domain $\Omega_1$ dominate the errors in the entire $\Omega$ as in the previous tests.

\begin{table}
\caption{Comparison of the accuracy in time (as computed in \eqref{eqn:sd-conv-in-time} and \eqref{eqn:dd-conv-in-time}) and the order of the convergence between SD and DD for $\rho$ in Test~\eqref{test:blowup_at_center} at the final time $10^{-6}$ with time step $\Delta t=10^{-7}/\tau$. The underlying mesh for SD is $N=255$ and the meshes for DD are $N_1/N_2=127/127$ ($h_1/h_2\approx1/2$) and $N_1/N_2=127/63$ ($h_1/h_2\approx1/4$). The reference solutions are obtained using time step $\Delta t=10^{-7}/512$ at the same final time $10^{-6}$.}
\label{table:convergence-rho-in-time}
\centering
\footnotesize
\sisetup{
  output-exponent-marker = \text{ E},
  exponent-product={},
  retain-explicit-plus
}
\begin{tabular}{c S[table-format=1.4e2] c c S[table-format=1.4e2] c S[table-format=1.4e2] c}
\toprule
\multicolumn{3}{c}{Single Domain} & \multicolumn{5}{c}{Domain Decomposition ($h_1/h_2\approx1/2$)} \\
\cmidrule(lr){1-3}\cmidrule(lr){4-8}
{$\tau$} & {$E_t$: $\Omega$} & {Rate: $\Omega$} & {$\tau$} & {$E_t$: $\Omega_1$} & {Rate: $\Omega_1$} & {$E_t$: $\Omega_2$} & {Rate: $\Omega_2$} \\
\cmidrule(lr){1-3}\cmidrule(lr){4-8}
 16 & 8.2335e-03 & \textemdash &  16 & 8.2331e-03 & \textemdash & 2.0476e-11 & \textemdash \\
 32 & 3.9846e-03 & 1.05        &  32 & 3.9844e-03 & 1.05        & 1.0063e-11 & 1.02 \\
 64 & 1.8596e-03 & 1.10        &  64 & 1.8595e-03 & 1.10        & 5.0106e-12 & 1.01 \\
128 & 7.9700e-04 & 1.22        & 128 & 7.9697e-04 & 1.22        & 2.3668e-12 & 1.08 \\
256 & 2.6567e-04 & 1.58        & 256 & 2.6566e-04 & 1.58        & 9.4522e-13 & 1.32 \\
\cmidrule(lr){1-3}\cmidrule(lr){4-8}
\multicolumn{3}{c}{} & \multicolumn{5}{c}{Domain Decomposition ($h_1/h_2\approx1/4$)} \\
\cmidrule(lr){4-8}
      &     &                  & {$\tau$} & {$E_t$: $\Omega_1$} & {Rate: $\Omega_1$} & {$E_t$: $\Omega_2$} & {Rate: $\Omega_2$} \\
\cmidrule(lr){4-8}
    &            &             &  16 & 8.2331e-03 & \textemdash & 2.0316e-11 & \textemdash \\
    &            &             &  32 & 3.9844e-03 & 1.05        & 9.9324e-12 & 1.03 \\
    &            &             &  64 & 1.8595e-03 & 1.10        & 4.8182e-12 & 1.04 \\
    &            &             & 128 & 7.9697e-04 & 1.22        & 2.3817e-12 & 1.02 \\
    &            &             & 256 & 2.6566e-04 & 1.58        & 1.0446e-12 & 1.19 \\
\bottomrule
\end{tabular}
\end{table}

\begin{table}
\caption{Comparison of the accuracy in time (computed similarly to \eqref{eqn:sd-conv-in-time} and \eqref{eqn:dd-conv-in-time}) and the order of the convergence between SD and DD for $c$ in Test~\eqref{test:blowup_at_center} at the final time $10^{-6}$ with time step $\Delta t=10^{-7}/\tau$. The underlying mesh for SD is $N=255$ and the mesh for DD is $N_1/N_2=127/127$ ($h_1/h_2\approx1/2$) and $N_1/N_2=127/63$ ($h_1/h_2\approx1/4$). The reference solutions are obtained using time step $\Delta t=10^{-7}/512$ at the same final time $10^{-6}$.}
\label{table:convergence-c-in-time}
\centering
\footnotesize
\sisetup{
  output-exponent-marker = \text{ E},
  exponent-product={},
  retain-explicit-plus
}
\begin{tabular}{c S[table-format=1.4e2] c c S[table-format=1.4e2] c S[table-format=1.4e2] c}
\toprule
\multicolumn{3}{c}{Single Domain} & \multicolumn{5}{c}{Domain Decomposition ($h_1/h_2\approx1/2$)} \\
\cmidrule(lr){1-3}\cmidrule(lr){4-8}
{$\tau$} & {$E_t$: $\Omega$} & {Rate: $\Omega$} & {$\tau$} & {$E_t$: $\Omega_1$} & {Rate: $\Omega_1$} & {$E_t$: $\Omega_2$} & {Rate: $\Omega_2$} \\
\cmidrule(lr){1-3}\cmidrule(lr){4-8}
 16 & 6.1248e-08 & \textemdash &  16 & 1.2023e-07 & \textemdash & 2.6835e-08 & \textemdash \\
 32 & 2.9665e-08 & 1.05        &  32 & 5.8245e-08 & 1.05        & 1.2975e-08 & 1.05 \\
 64 & 1.3889e-08 & 1.09        &  64 & 2.7205e-08 & 1.10        & 6.0458e-09 & 1.10 \\
128 & 5.9737e-09 & 1.22        & 128 & 1.1680e-08 & 1.22        & 2.5827e-09 & 1.23 \\
256 & 2.0236e-09 & 1.56        & 256 & 3.9206e-09 & 1.57        & 8.5743e-10 & 1.59 \\
\cmidrule(lr){1-3}\cmidrule(lr){4-8}
\multicolumn{3}{c}{} & \multicolumn{5}{c}{Domain Decomposition ($h_1/h_2\approx1/4$)} \\
\cmidrule(lr){4-8}
      &     &                  & {$\tau$} & {$E_t$: $\Omega_1$} & {Rate: $\Omega_1$} & {$E_t$: $\Omega_2$} & {Rate: $\Omega_2$} \\
\cmidrule(lr){4-8}
    &            &             &  16 & 3.3382e-07 & \textemdash & 2.1151e-09 & \textemdash \\
    &            &             &  32 & 1.6168e-07 & 1.05        & 1.0236e-09 & 1.05 \\
    &            &             &  64 & 7.5496e-08 & 1.10        & 4.7699e-10 & 1.10 \\
    &            &             & 128 & 3.2382e-08 & 1.22        & 2.0256e-10 & 1.24 \\
    &            &             & 256 & 1.0826e-08 & 1.58        & 6.7669e-11 & 1.58 \\
\bottomrule
\end{tabular}
\end{table}


\subsubsection{Comparison of the evolution of $\rho$ to final time $6\times10^{-5}$}

\begin{figure}
    \centering
    \begin{subfigure}[b]{0.4\textwidth}
        \includegraphics[width=\textwidth]{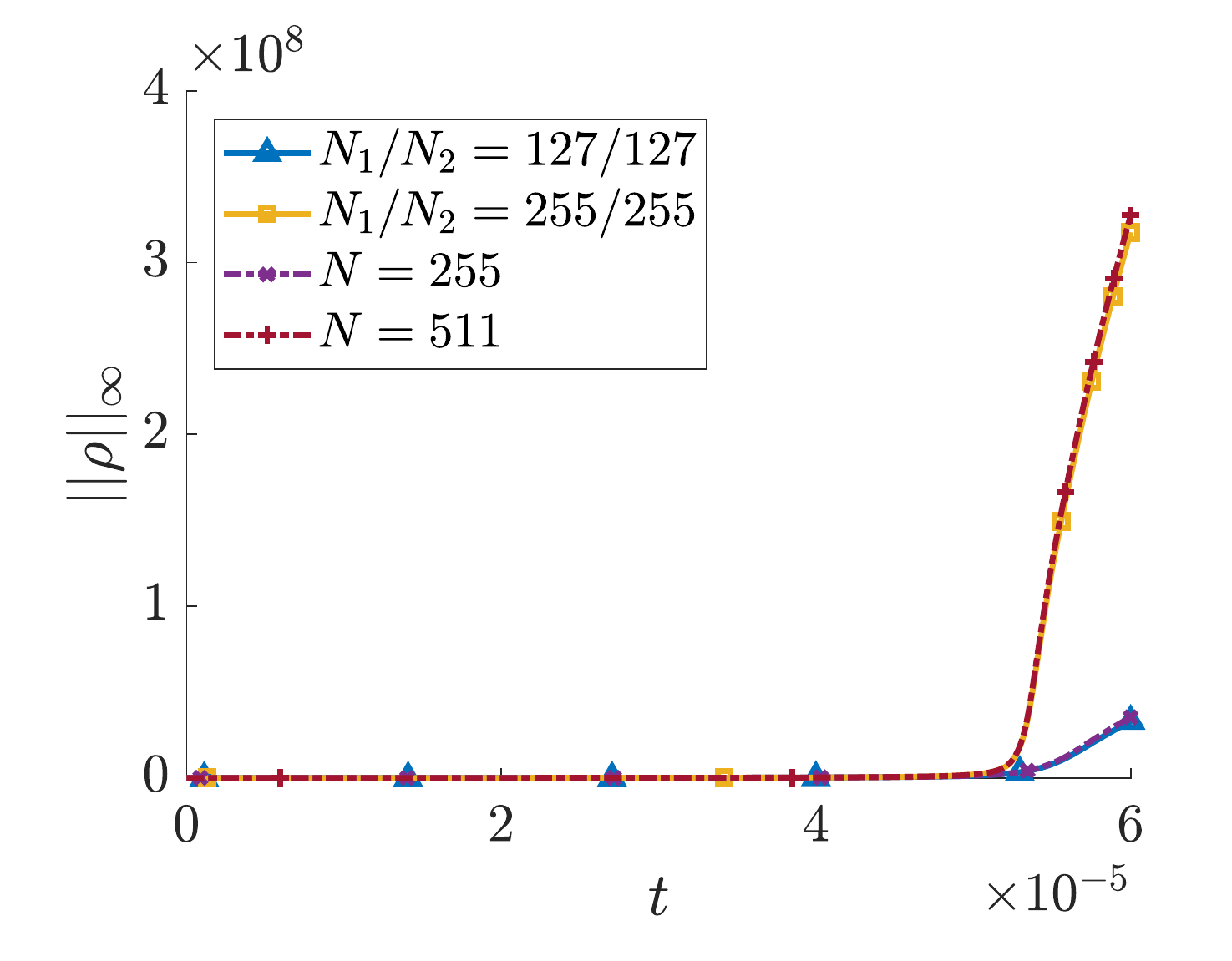}
        \caption{SD vs. DD ($h_1/h_2\approx1/2$)}
    \end{subfigure}
    \begin{subfigure}[b]{0.4\textwidth}
        \includegraphics[width=\textwidth]{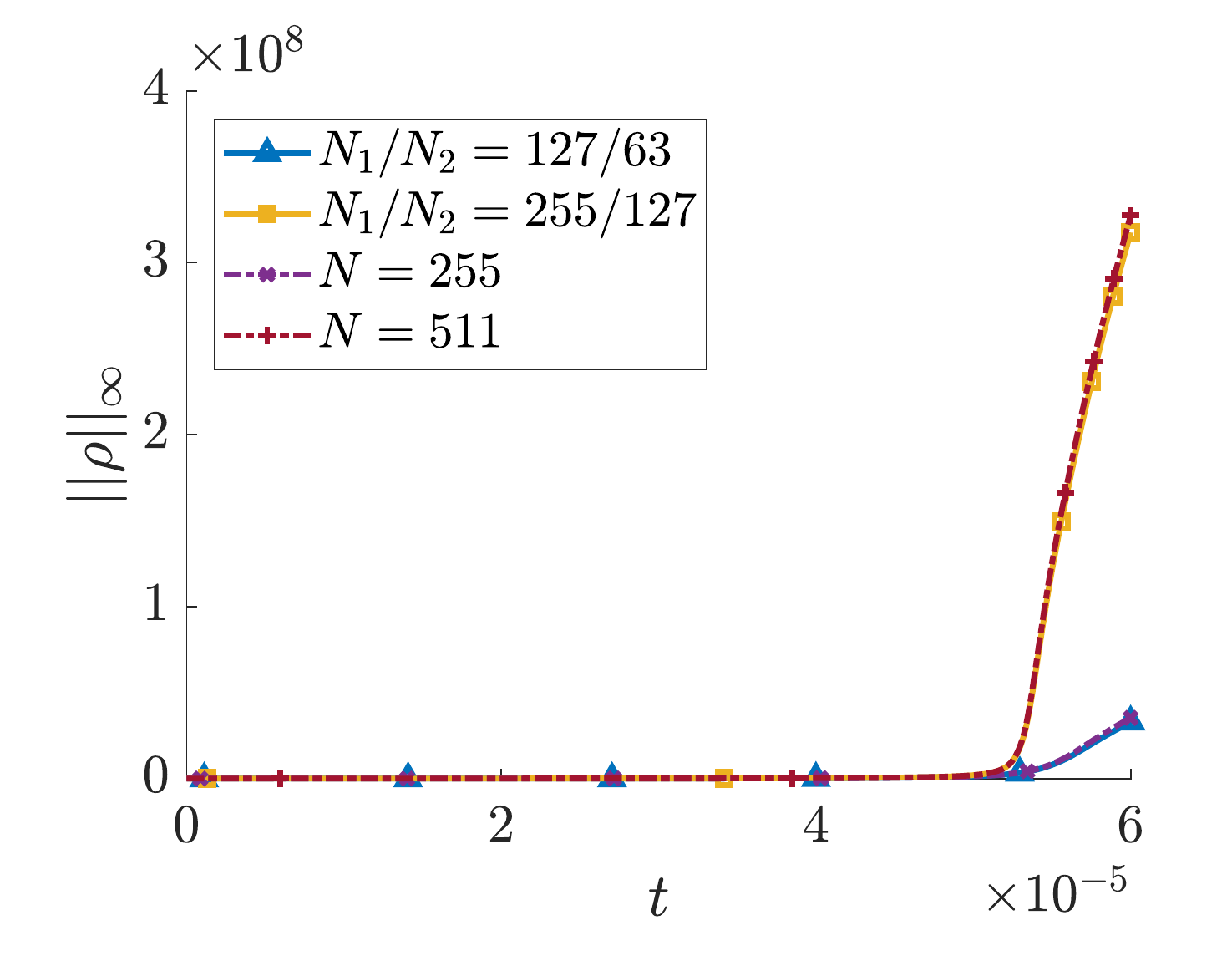}
        \caption{SD vs. DD ($h_1/h_2\approx1/4$)}
    \end{subfigure}
    \caption{Max Density \eqref{sd-norm:max} and \eqref{dd-norm:max} for Test~\eqref{test:blowup_at_center}.}
    \label{fig:comp-max-density}
\end{figure}

\begin{figure}
    \centering
    \begin{subfigure}[b]{0.4\textwidth}
        \includegraphics[width=\textwidth]{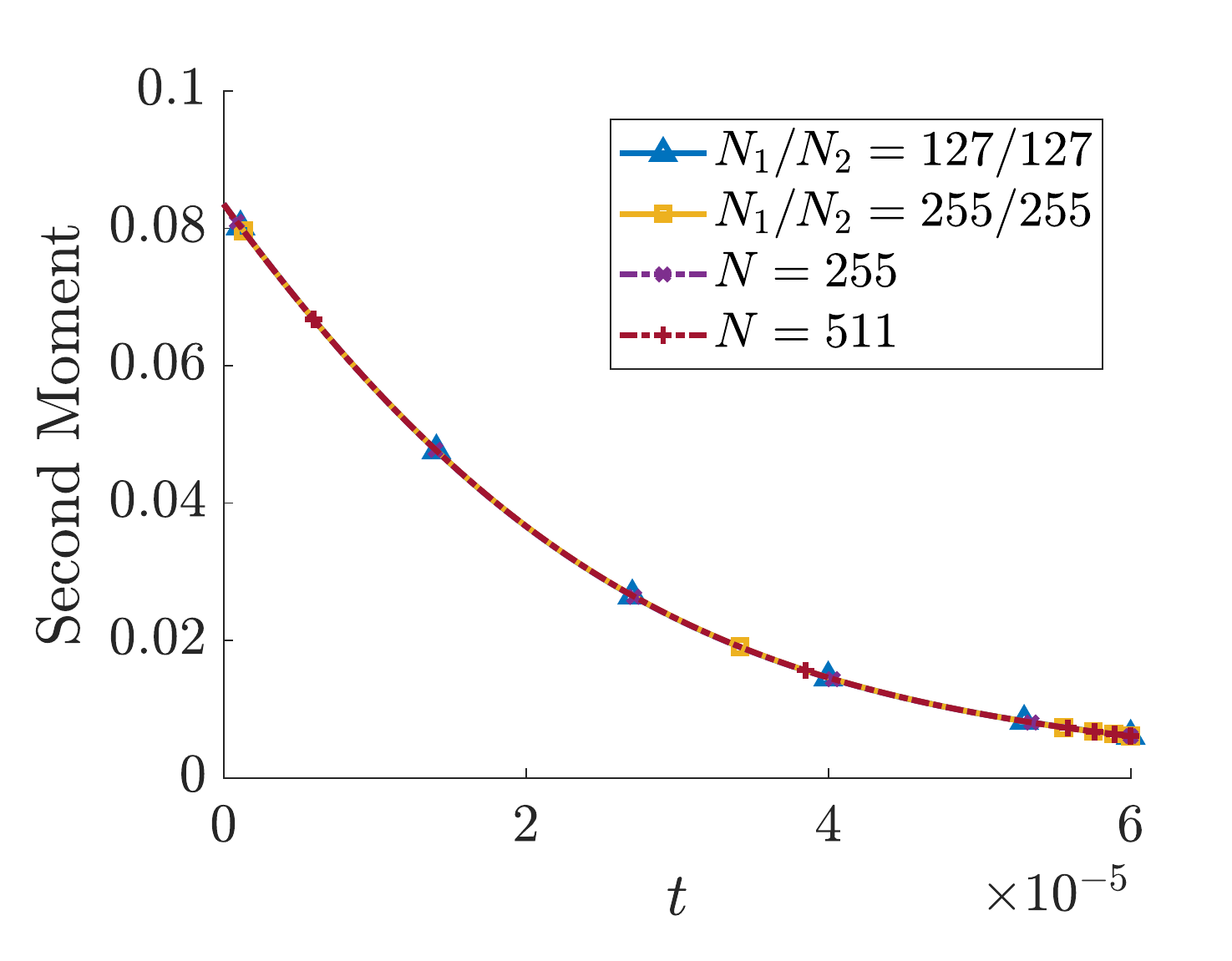}
        \caption{SD vs. DD ($h_1/h_2\approx1/2$)}
    \end{subfigure}
    \begin{subfigure}[b]{0.4\textwidth}
        \includegraphics[width=\textwidth]{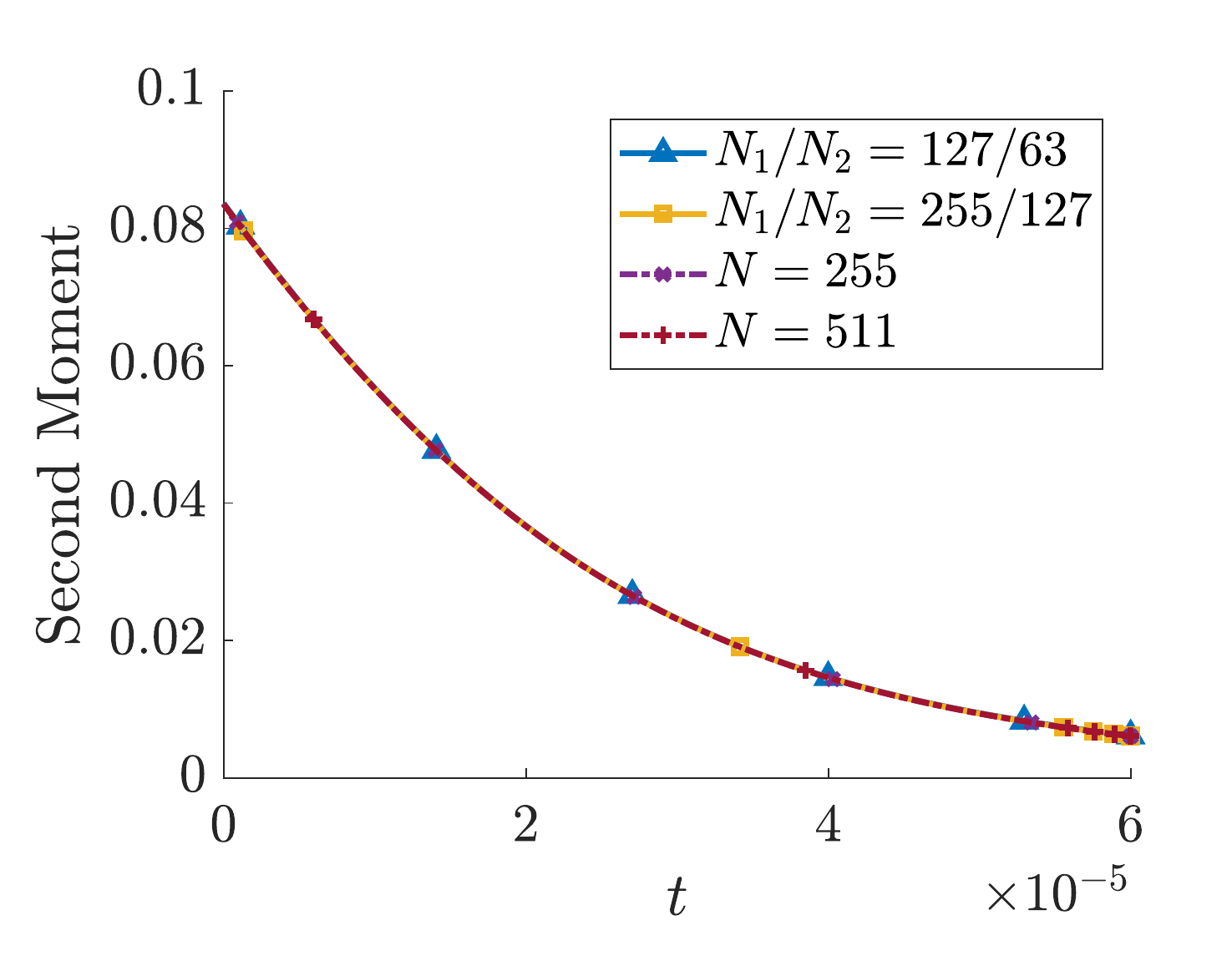}
        \caption{SD vs. DD ($h_1/h_2\approx1/4$)}
    \end{subfigure}
    \caption{Second Moment \eqref{sd-norm:second-moment} and \eqref{dd-norm:second-moment} for Test~\eqref{test:blowup_at_center}.}
    \label{fig:comp-second-moment}
\end{figure}

\begin{figure}
    \centering
    \begin{subfigure}[b]{0.4\textwidth}
        \includegraphics[width=\textwidth]{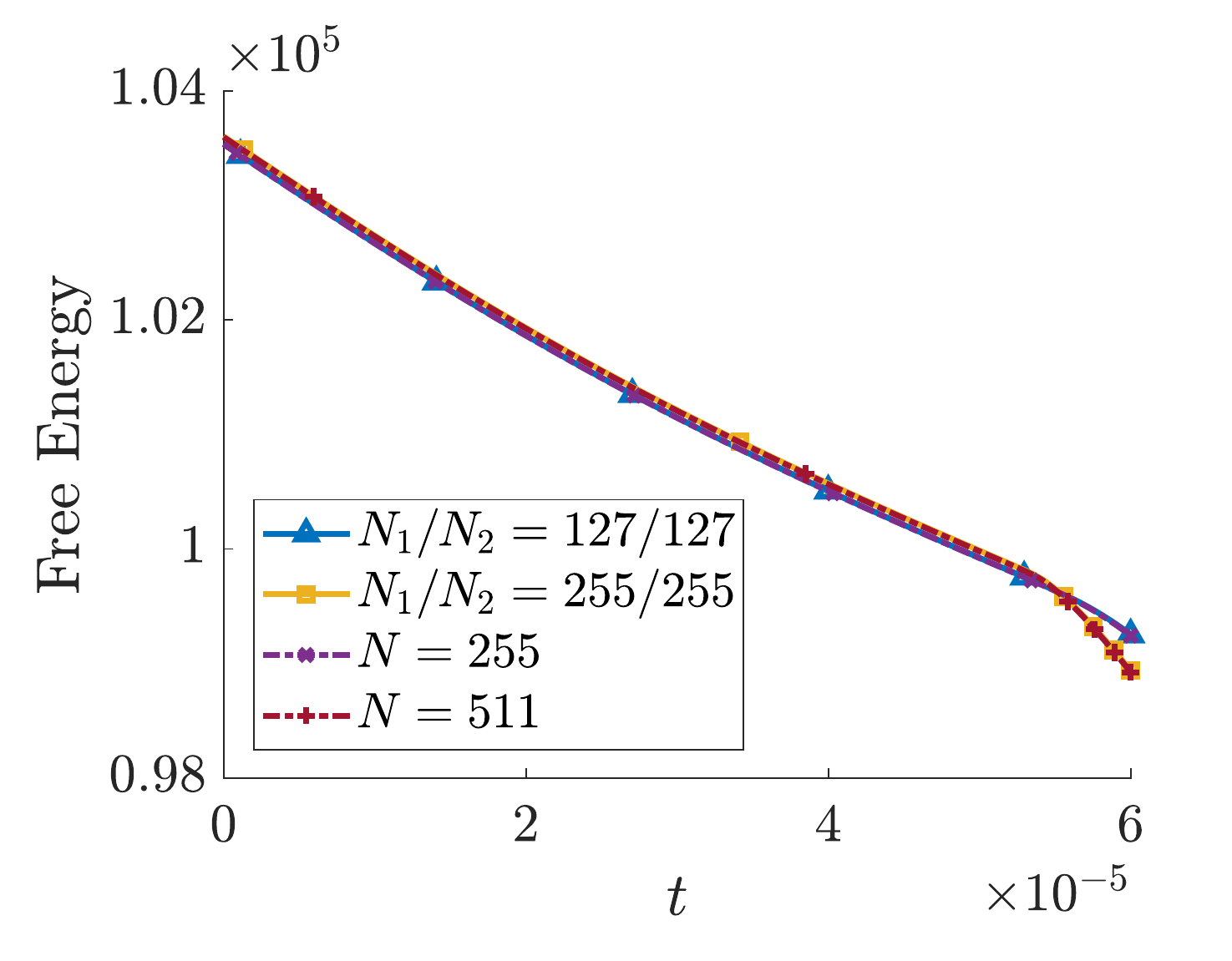}
        \caption{SD vs. DD ($h_1/h_2\approx1/2$)}
    \end{subfigure}
    \begin{subfigure}[b]{0.4\textwidth}
        \includegraphics[width=\textwidth]{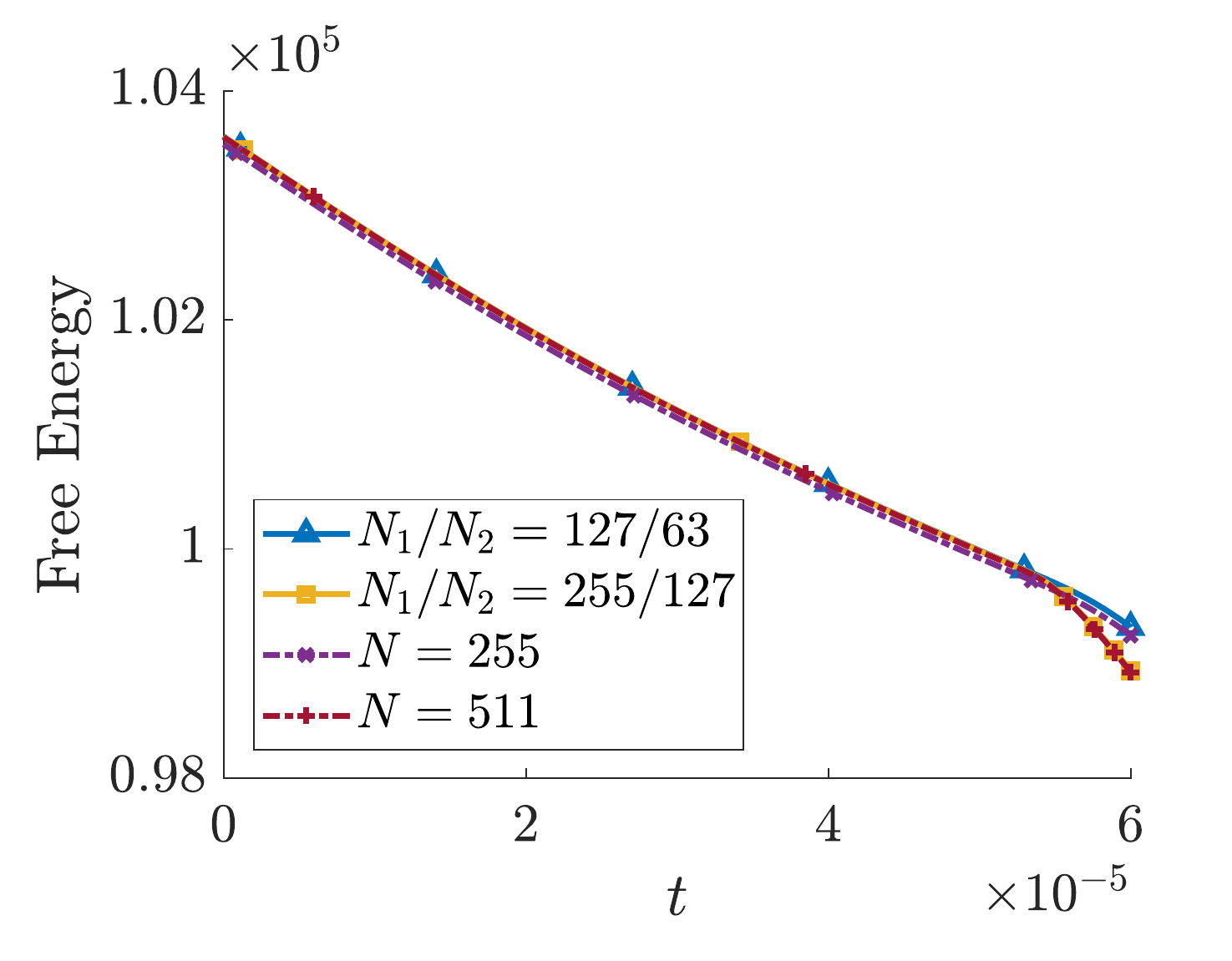}
        \caption{SD vs. DD ($h_1/h_2\approx1/4$)}
    \end{subfigure}
    \caption{Free Energy \eqref{sd-norm:free-energy} and \eqref{dd-norm:free-energy} for Test~\eqref{test:blowup_at_center}.}
    \label{fig:comp-free-energy}
\end{figure}
The reason for choosing the final time $6\times10^{-5}$ is that blow-up already occurs around $T=6\times10^{-5}$ (see Fig.~\ref{fig:evolution-blowup-center}). We consider mesh size $h$ for the SD approach and mesh sizes $h_1/h_2$ for the DD approach in sub-domain $\Omega_1$ and sub-domain $\Omega_2$ respectively. Before blow-up, the time step is selected to be $\Delta t=0.5h^2$ for the SD approach and $\Delta t = \min\{0.5h_1^2,0.5h^2_2\}$ for the DD approach. Near blow-up, the time step is selected according to the CFL-type condition~\eqref{eqn:cfl_condition}. Through Figs. \ref{fig:comp-max-density}--\ref{fig:comp-free-energy}, we obtain similar results in all aspects ($||\rho||_{\infty}$, free energy, etc.) between the SD and the DD approaches, as long as the finest grid sizes $h$ and $h_1$ are close. DD approaches with ratios $h_1/h_2\approx1/2$ and $h_1/h_2\approx1/4$ give similar approximations, due to smooth solutions in $\Omega_2$. 

We have seen that the domain decomposition approach gives similar accuracy as the single domain approach. Now, to illustrate the efficiency of the domain decomposition approach, we first define the following speedup ratio $R_S$:
\begin{align}\label{eqn:speedup-ratio}
R_S = \frac{T_{SD}}{T_{DD}}
\end{align}
where $T_{SD}$ is the wall-clock time of running the simulation with mesh size $h$ to the final time $T$ in the single domain approach, and $T_{DD}$ is the wall-clock time of running the same simulation with mesh sizes $h_1/h_2$ to the same final time $T$ in the domain decomposition approach. 

In Table~\ref{table:speedup}, we summarize and compare the speedup ratio between two choices of mesh adaptivity in the domain decomposition approach. We observe that DD approach with $h_1/h_2\approx1/4$ gives approximately $7\times$/$8\times$ speedup, while DD approach with $h_1/h_2\approx1/2$ gives approximately $4\times$ speedup, in comparison to using the SD approach with similar accuracy in space. The speedup is mainly due to the reduction of degrees of freedom and better parallelization properties of DD algorithms.

\begin{table}[H]
\caption{Comparison of speedup ratio $R_S$ as computed in \eqref{eqn:speedup-ratio} between DD ($h_{1}/h_{2}\approx1/2$)/DD ($h_{1}/h_{2}\approx1/4$) and SD ($h\approx h_1$) to a final time $6\times10^{-5}$. The speedup ratios $R_S$ are computed using wall-clock time obtained by running the codes on the University of Utah CHPC cluster using one Notchpeak node (32 cores, 96/192 GB memory).}
\label{table:speedup}
\centering
\scriptsize
\sisetup{
  output-exponent-marker = \text{ E},
  exponent-product={},
  retain-explicit-plus
}
\begin{tabular}{cccccccc}
\toprule
{SD: $N$} & {SD: $h$} & {DD: $N_{1}/N_{2}$} & {DD: $h_{1}/h_{2}\approx1/2$} & {$R_S$} & {DD: $N_{1}/N_{2}$} & {DD: $h_{1}/h_{2}\approx1/4$} & {$R_S$} \\
\cmidrule(lr){1-2}\cmidrule(lr){3-5} \cmidrule(lr){6-8}
$127$ & $1/123$& $63/63$     & $\frac{1}{118}/\frac{1}{59}$  & 3.27 & $63/31$     & $\frac{1}{118}/\frac{1}{27}$ & 4.5\\
$255$ & $1/251$& $127/127$   & $\frac{1}{246}/\frac{1}{123}$ & 6.46 & $127/63$    & $\frac{1}{246}/\frac{1}{59}$ &  9.26\\
$511$ & $1/507$& $255/255$   & $\frac{1}{502}/\frac{1}{251}$ & 3.70 & $255/127$   & $\frac{1}{502}/\frac{1}{123}$ & 7.14\\     
\bottomrule
\end{tabular}
\end{table}

\subsection{Numerical Results for Test \eqref{test:blowup_at_center} to final time $10^{-4}$}
As can be seen in the last two subsections, DD ($h_1/h_2\approx1/4$) gives similar accuracy to SD and DD ($h_1/h_2\approx1/2$), but is much more efficient. Hence we will continue with DD ($h_1/h_2\approx1/4$) for Test \eqref{test:blowup_at_center} to explore and simulate the chemotactic process on finer meshes to a longer time $T=10^{-4}$. This final time is already post blow-up, as can be seen in Fig.~\ref{fig:evolution-blowup-center}. Before blow-up, the time step is also selected to be $\Delta t=0.5h^2$ for the SD approach and $\Delta t = \min\{0.5h_1^2,0.5h^2_2\}$ for the DD approach. Near blow-up, the time step is selected according to the CFL-type condition~\eqref{eqn:cfl_condition}. 

\begin{figure}
    \centering
    \begin{subfigure}[b]{0.4\textwidth}
        \includegraphics[width=\textwidth]{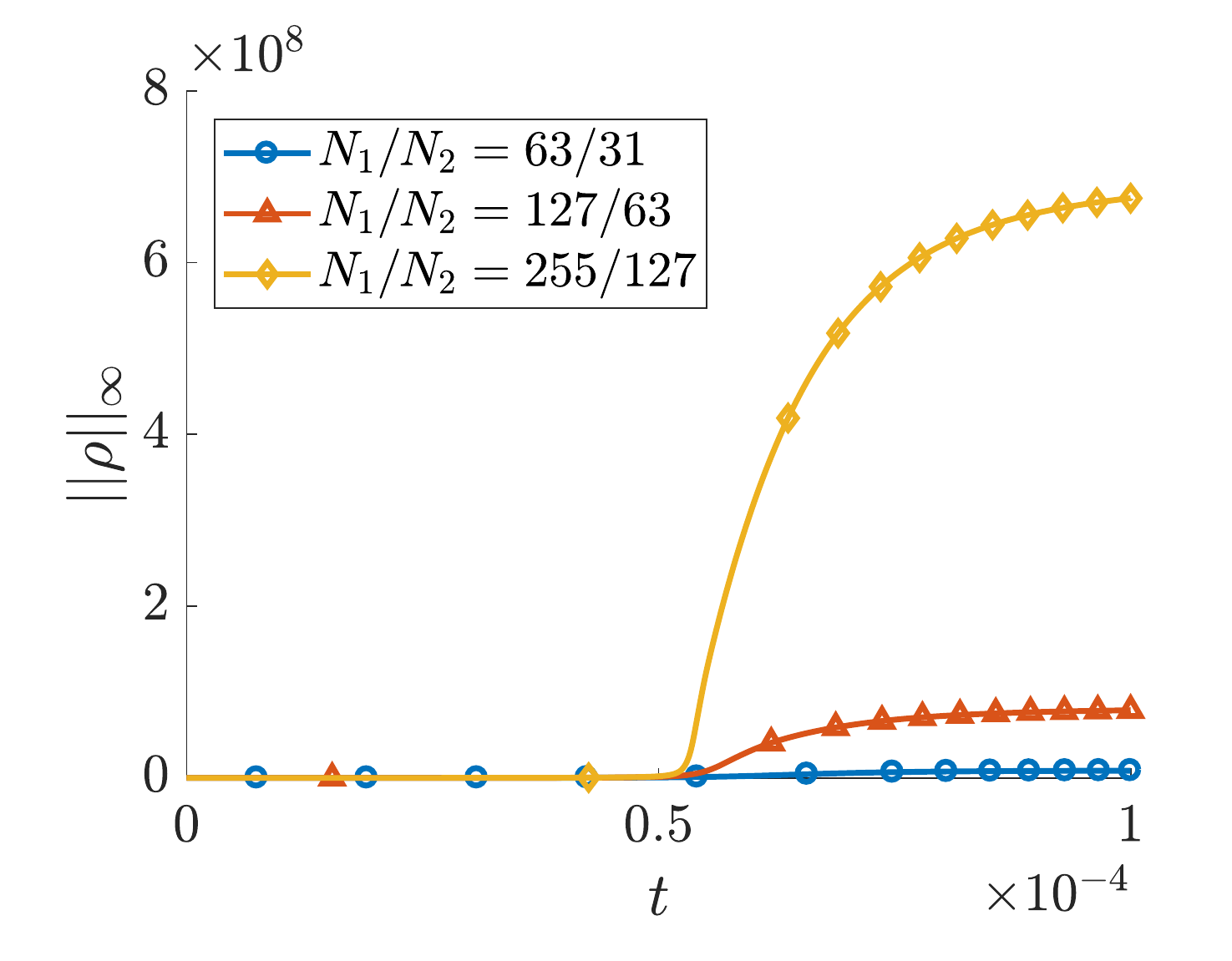}
        \caption{Max Density}\label{fig:sub-CC-md}
    \end{subfigure}
    \begin{subfigure}[b]{0.4\textwidth}
        \includegraphics[width=\textwidth]{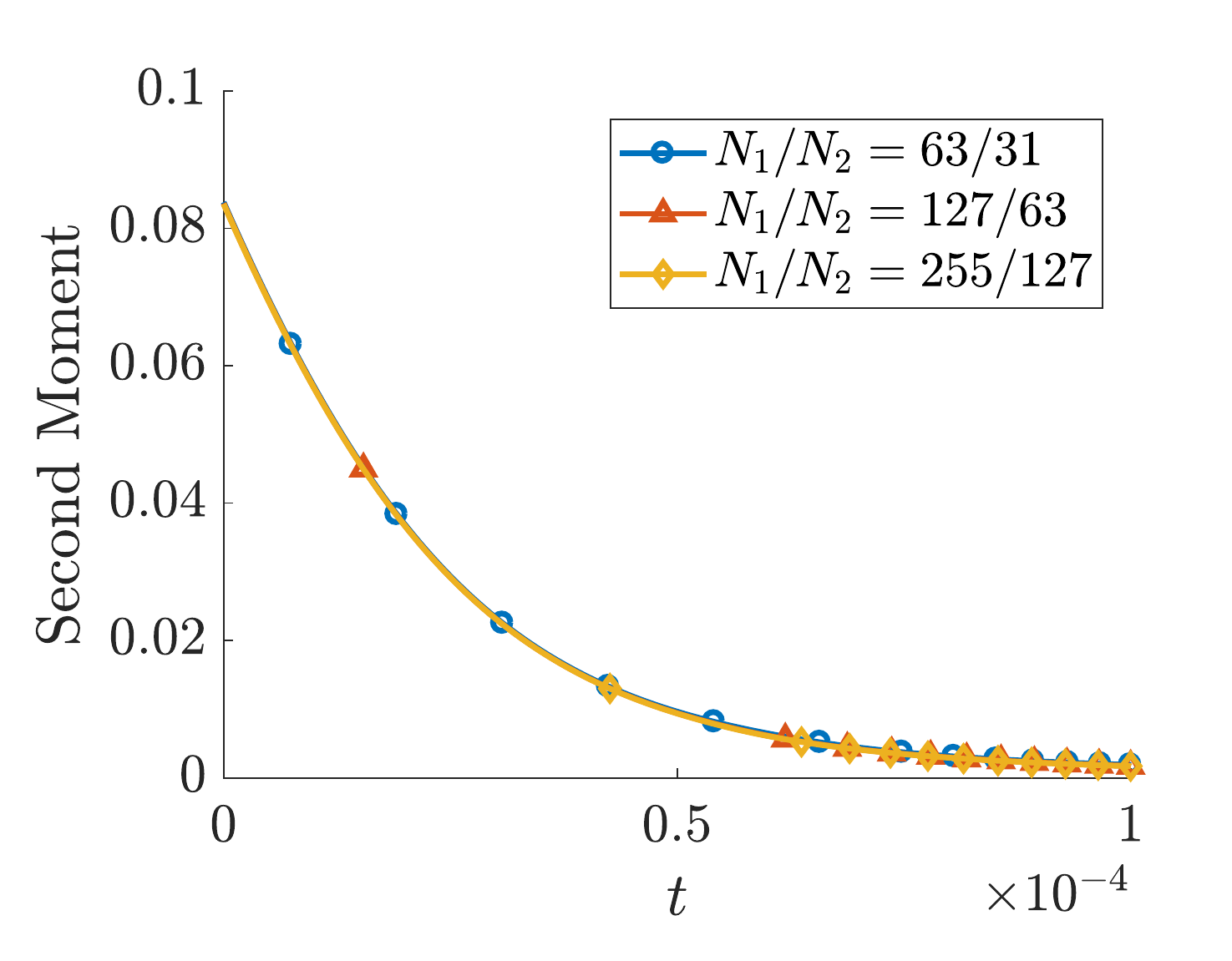}
        \caption{Second Moment}\label{fig:sub-CC-sm}
    \end{subfigure}
    \begin{subfigure}[b]{0.4\textwidth}
        \includegraphics[width=\textwidth]{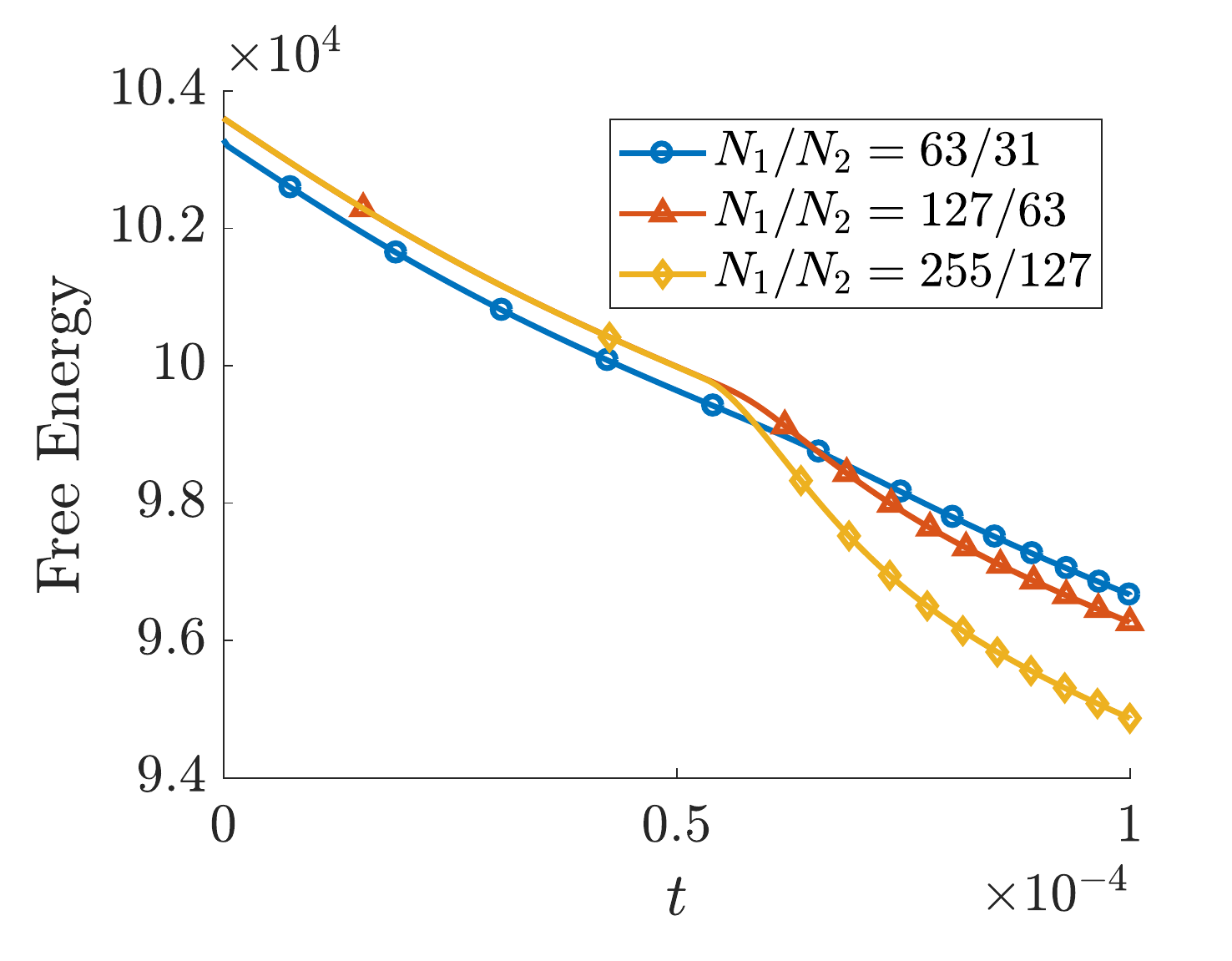}
        \caption{Free Energy}\label{fig:sub-CC-fe}
    \end{subfigure}
    \caption{Test~\eqref{test:blowup_at_center} to final time $10^{-4}$ using DD ($h_1/h_2\approx1/4$).}
    \label{fig:all-blowup-center}
\end{figure}

Before commenting on the longer-time simulation results, we should note that DD with mesh $N_1/N_2=255/127$ gives similar results to SD with mesh $N=511$, but SD with mesh $N=511$ would take significantly longer time to finish as we observed in simulations.

We plot the evolution of $||\rho||_{\infty}$ versus time in Fig.~\ref{fig:sub-CC-md}, and this test can be used to detect blow-up time, similar to~\cite{Chertock_2017,Chertock_2008,Epsh1}.
In Fig.~\ref{fig:sub-CC-sm} and Fig.~\ref{fig:sub-CC-fe}, we show that the second moment and the free energy are both decreasing, which obey the second law of thermodynamics. In addition, the decrease rates of free energy are similar on different meshes before time step becomes restrictive due to the CFL-type condition~\eqref{eqn:cfl_condition} near blow-up time. After the time step becomes smaller, the decrease rates of the free energy are even larger, especially on finer meshes. This might be explained by the faster aggregations of cells on finer meshes.

\begin{figure}
    \centering
    \begin{subfigure}[b]{0.33\textwidth}
        \includegraphics[width=\textwidth]{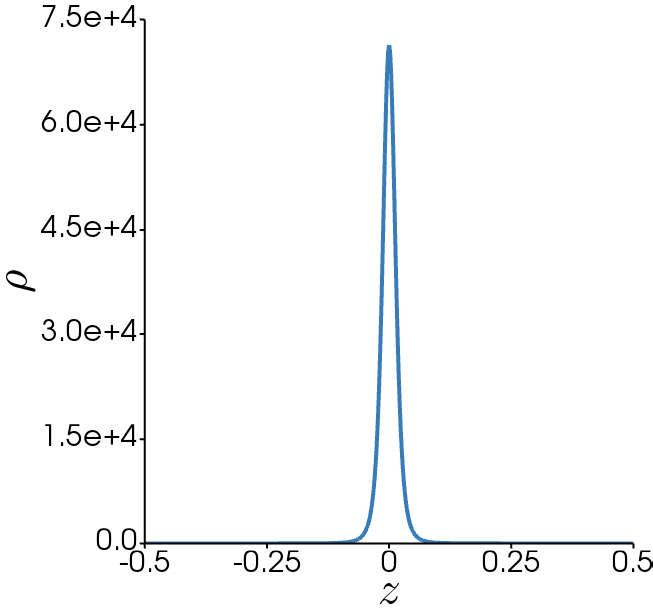}
        \caption{$\rho$ at $t=3\times10^{-5}$}
    \end{subfigure}
    \begin{subfigure}[b]{0.33\textwidth}
        \includegraphics[width=\textwidth]{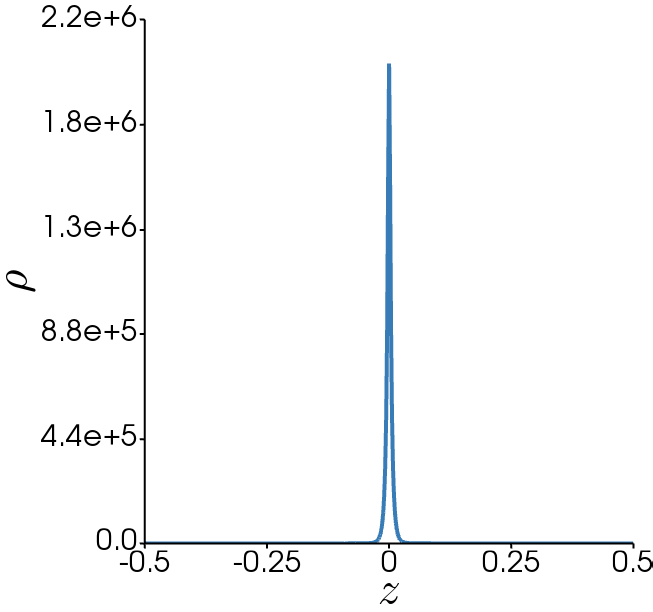}
        \caption{$\rho$ at $t=5\times10^{-5}$}
    \end{subfigure}
    \begin{subfigure}[b]{0.33\textwidth}
        \includegraphics[width=\textwidth]{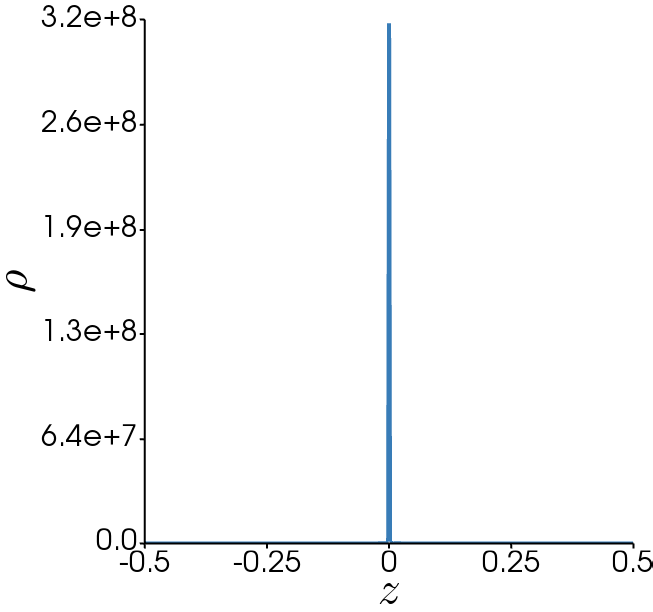}
        \caption{$\rho$ at $t=6\times10^{-5}$}
    \end{subfigure}
    \caption{Cutline plots along $z$-axis of $\rho$ at different times for Test~\eqref{test:blowup_at_center} on mesh $N_1/N_2=255/127$ with grid sizes $h_1/h_2\approx1/4$ using domain decomposition approach.}
    \label{fig:evolution-blowup-center}
\end{figure}

In Fig.~\ref{fig:evolution-blowup-center}, snapshots of the cutline plots along $z$-axis of the solutions $\rho$ at different times are given. We can see that the max value of $\rho$ increases by two magnitudes from $t=3\times10^{-5}$ to $t=5\times10^{-5}$. Density $\rho$ already starts to develop singularity at the origin at $t=5\times10^{-5}$. Eventually at $t=6\times10^{-5}$, $\rho$ develops a almost singular solution.

In Fig.~\ref{fig:isosurface-blowup-center}, we present the 3D view of the isosurface plots of the solutions $\rho$ and $c$ at time $t=3\times10^{-5}$, from which we saw that indeed the peak values of both solutions $\rho$ and $c$ occur at the origin only. We chose this time $t=3\times10^{-5}$ to give a better 3D visualization of the isosurface plots. As can be seen in Fig.~\ref{fig:evolution-blowup-center}, $\rho$ at later times will be highly concentrated at the origin, which makes the 3D isosurface plots at later time levels uninformative. Moreover, the outlines of two cubic auxiliary domains (for $\Omega_1$ and $\Omega_2$ respectively) are also given in Fig.~\ref{fig:isosurface-blowup-center}, from which we can observe that the peak values indeed occur in sub-domain $\Omega_1$.

\begin{figure}
    \centering
    \begin{subfigure}[b]{0.4\textwidth}
        \includegraphics[width=\textwidth]{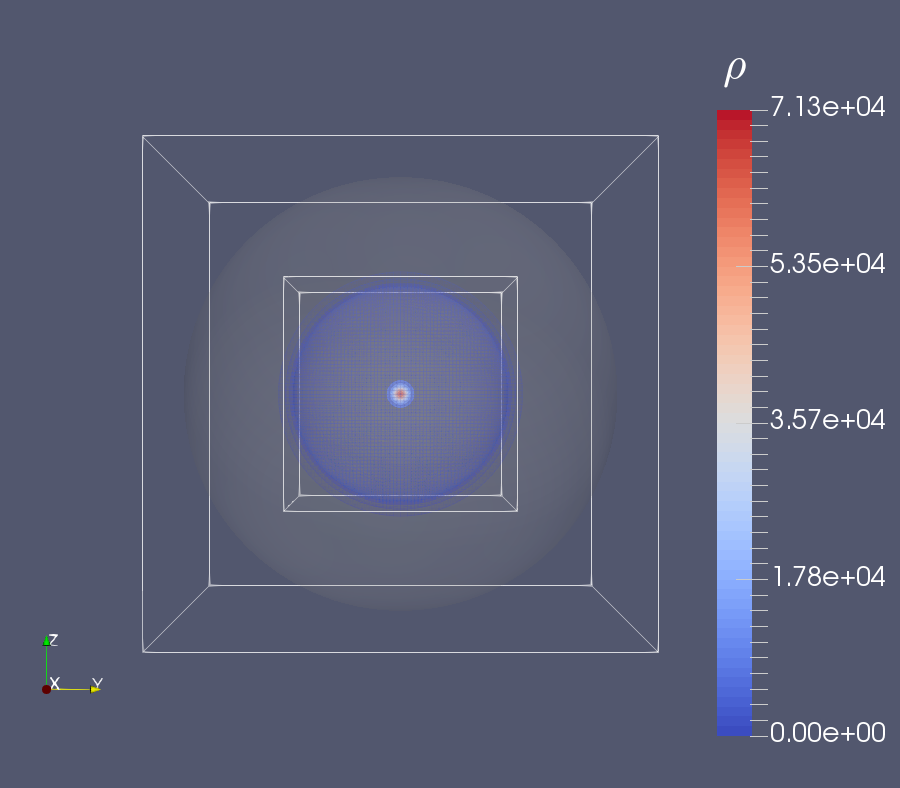}
        \caption{$\rho$ at $t=3\times10^{-5}$}
    \end{subfigure}
    \begin{subfigure}[b]{0.4\textwidth}
        \includegraphics[width=\textwidth]{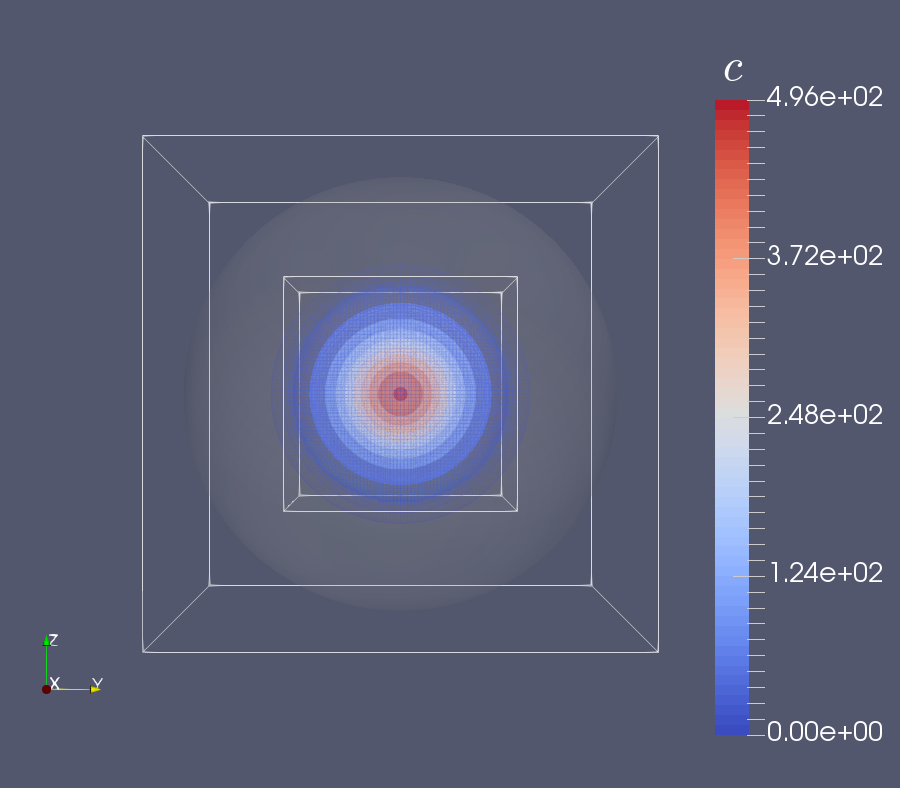}
        \caption{$c$ at $t=3\times10^{-5}$}
    \end{subfigure}
    \caption{Isosurface plots of $\rho$ (left) and $c$ (right) for Test~\eqref{test:blowup_at_center} on mesh $N_1/N_2=255/127$ with grid sizes $h_1/h_2\approx1/4$ using domain decomposition approach at $t=3\times10^{-5}$.}
    \label{fig:isosurface-blowup-center}
\end{figure}


\subsection{Selected Numerical Results for Test~\eqref{test:blowup_at_boundary}}\label{subsec:blowup_boundary}

In this subsection, we present selected simulation results for Test~\eqref{test:blowup_at_boundary}. The blow-up is expected to occur at the north pole and our simulation results agree with our expectations. Again, the time step is selected to be $\Delta t=0.5h^2$ for the SD approach and $\Delta t = \min\{0.5h_1^2,0.5h^2_2\}$ for the DD approach before blow-up. Near blow-up, the time step is selected according to the CFL-type condition~\eqref{eqn:cfl_condition}. In our simulations, the codes are terminated once the change in $||\rho||_{\infty}$ over two consecutive time steps exceeds 1000, i.e.
\begin{align}\label{criterion:stoppage}
\Delta ||\rho||_{\infty}\geq1000
\end{align} 
We use criterion \eqref{criterion:stoppage} to determine that the solutions have reached blow-up. As can be seen in Fig.~\ref{fig:sub-BB-md}, Fig.~\ref{fig:sub-BB-rho-change} and Table~\ref{table:rho-max-interval}, the max density $||\rho||_{\infty}$ undergoes a drastic increase in a very small time interval.

\begin{figure}
    \centering
    \begin{subfigure}[b]{0.45\textwidth}
        \includegraphics[width=\textwidth]{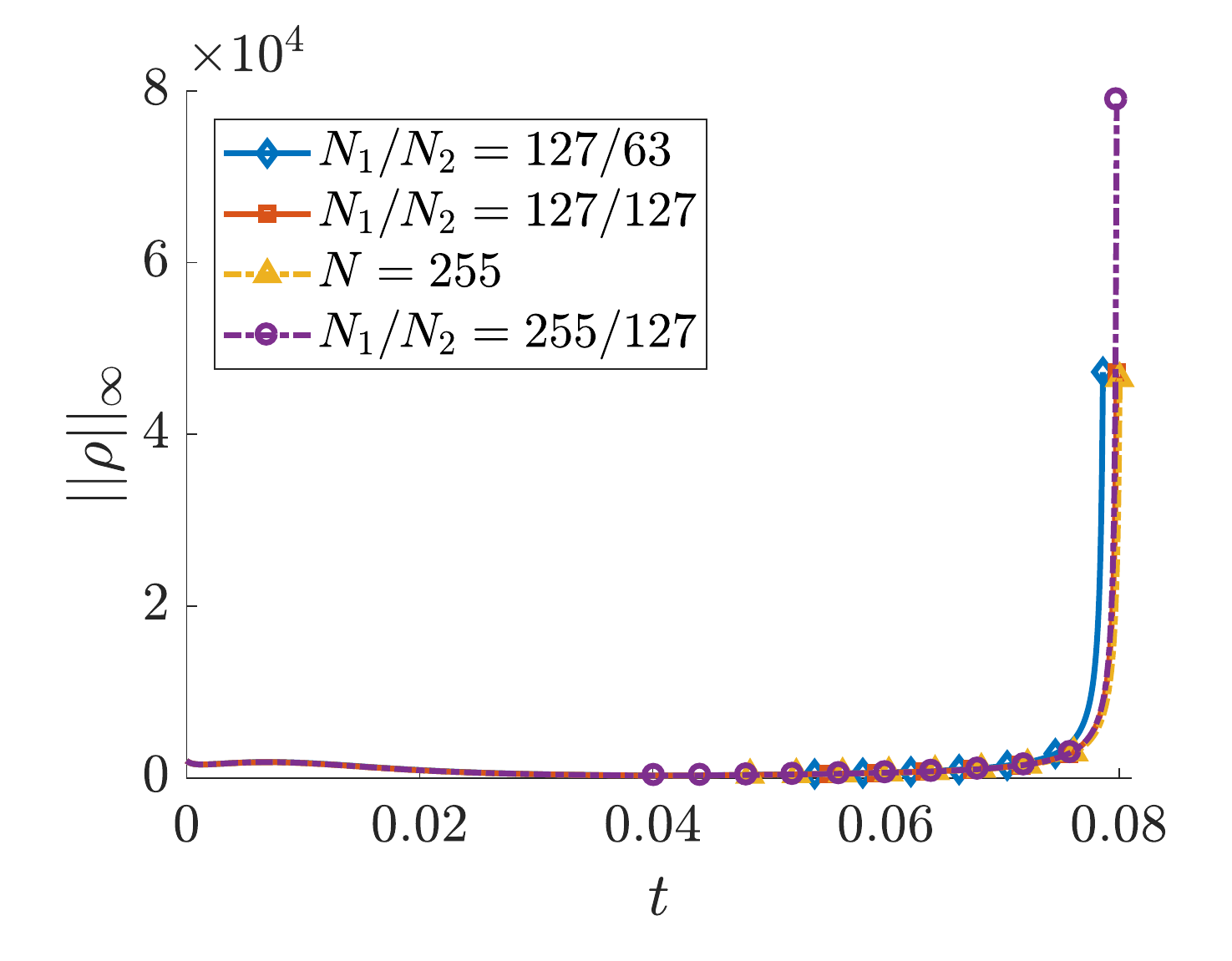}
        \caption{Max Density}\label{fig:sub-BB-md}
    \end{subfigure}
    \begin{subfigure}[b]{0.45\textwidth}
        \includegraphics[width=\textwidth]{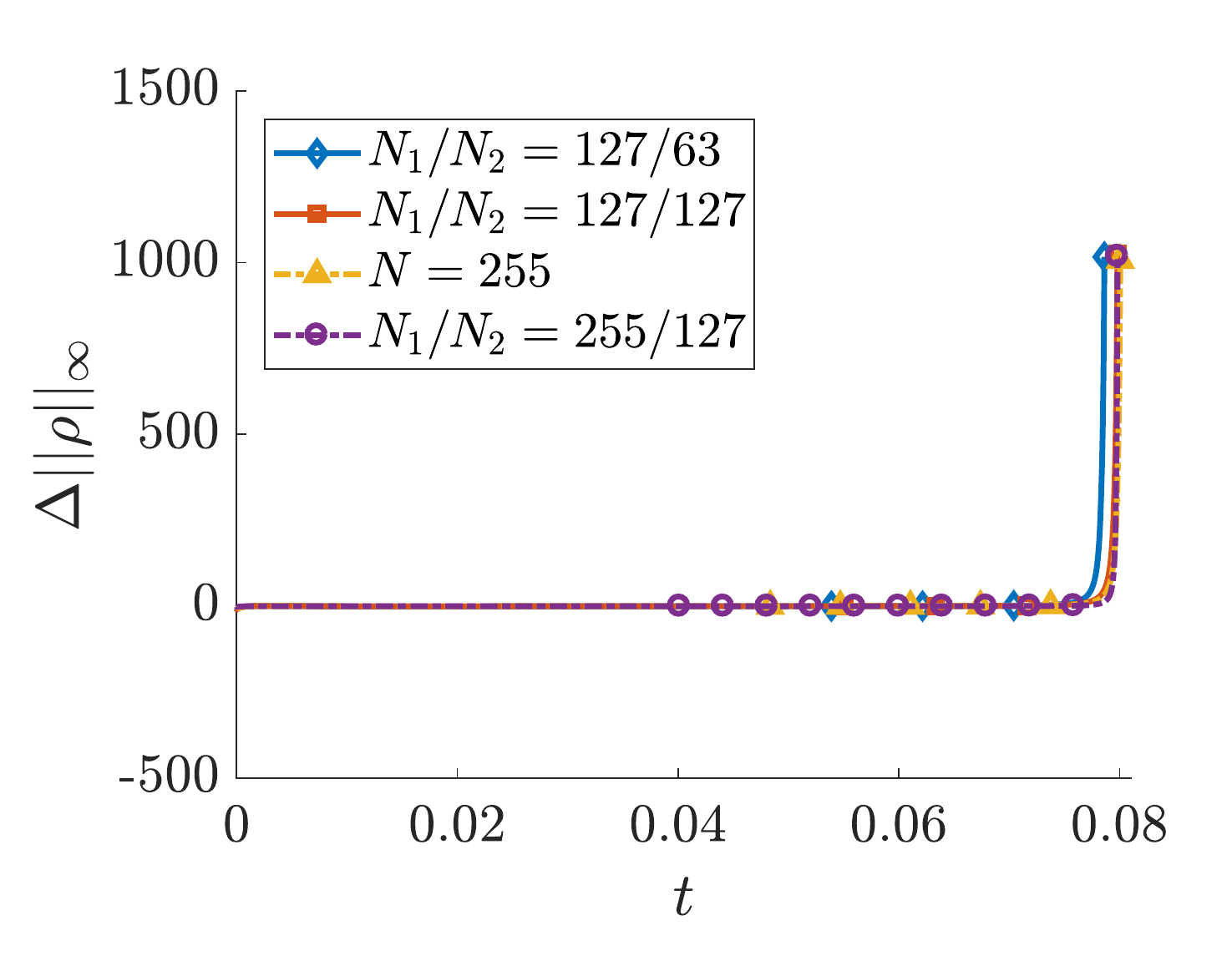}
        \caption{Change of $||\rho||_{\infty}$}\label{fig:sub-BB-rho-change}
    \end{subfigure}
    \begin{subfigure}[b]{0.45\textwidth}
        \includegraphics[width=\textwidth]{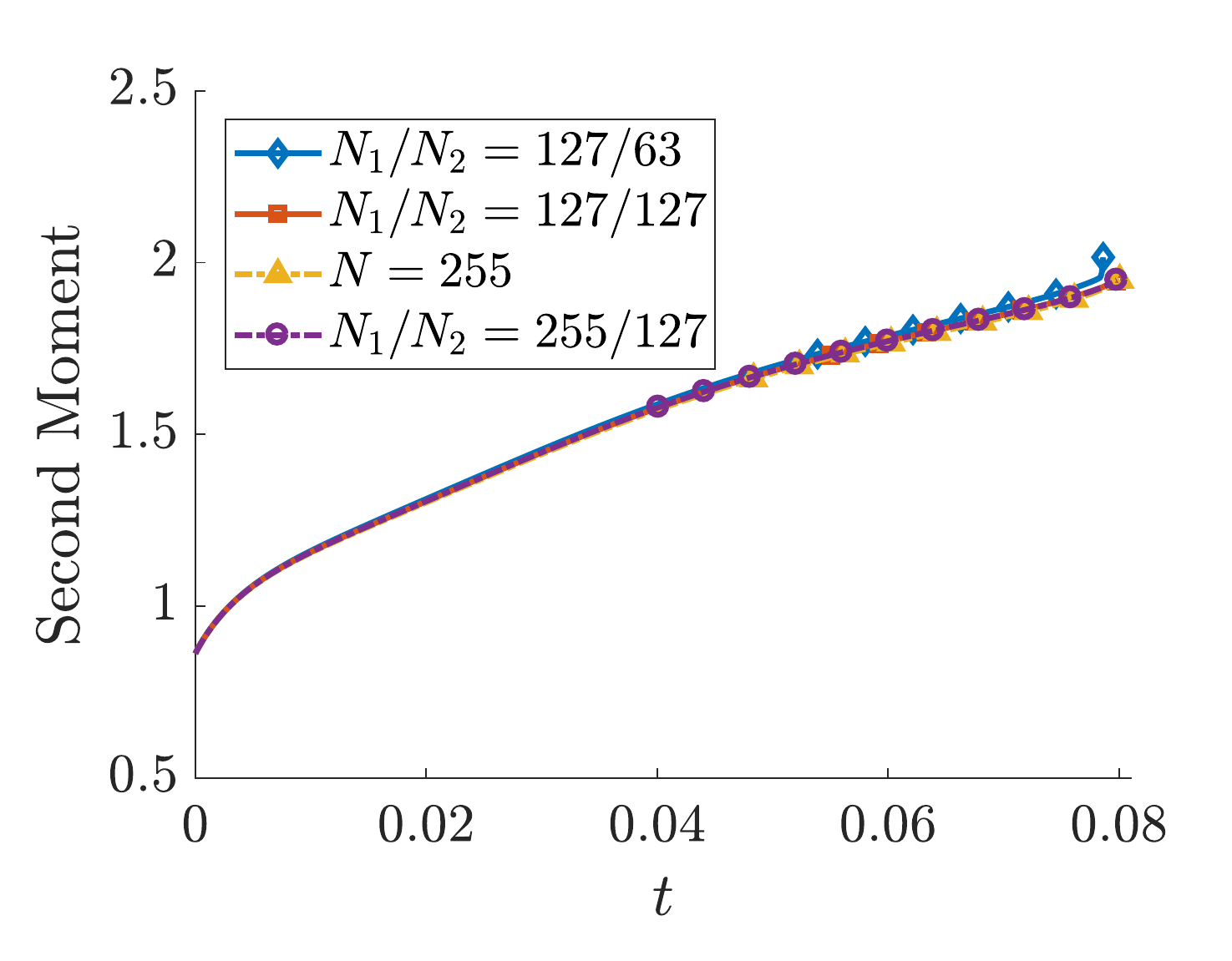}
        \caption{Second Moment}\label{fig:sub-BB-sm}
    \end{subfigure}
    \begin{subfigure}[b]{0.45\textwidth}
        \includegraphics[width=\textwidth]{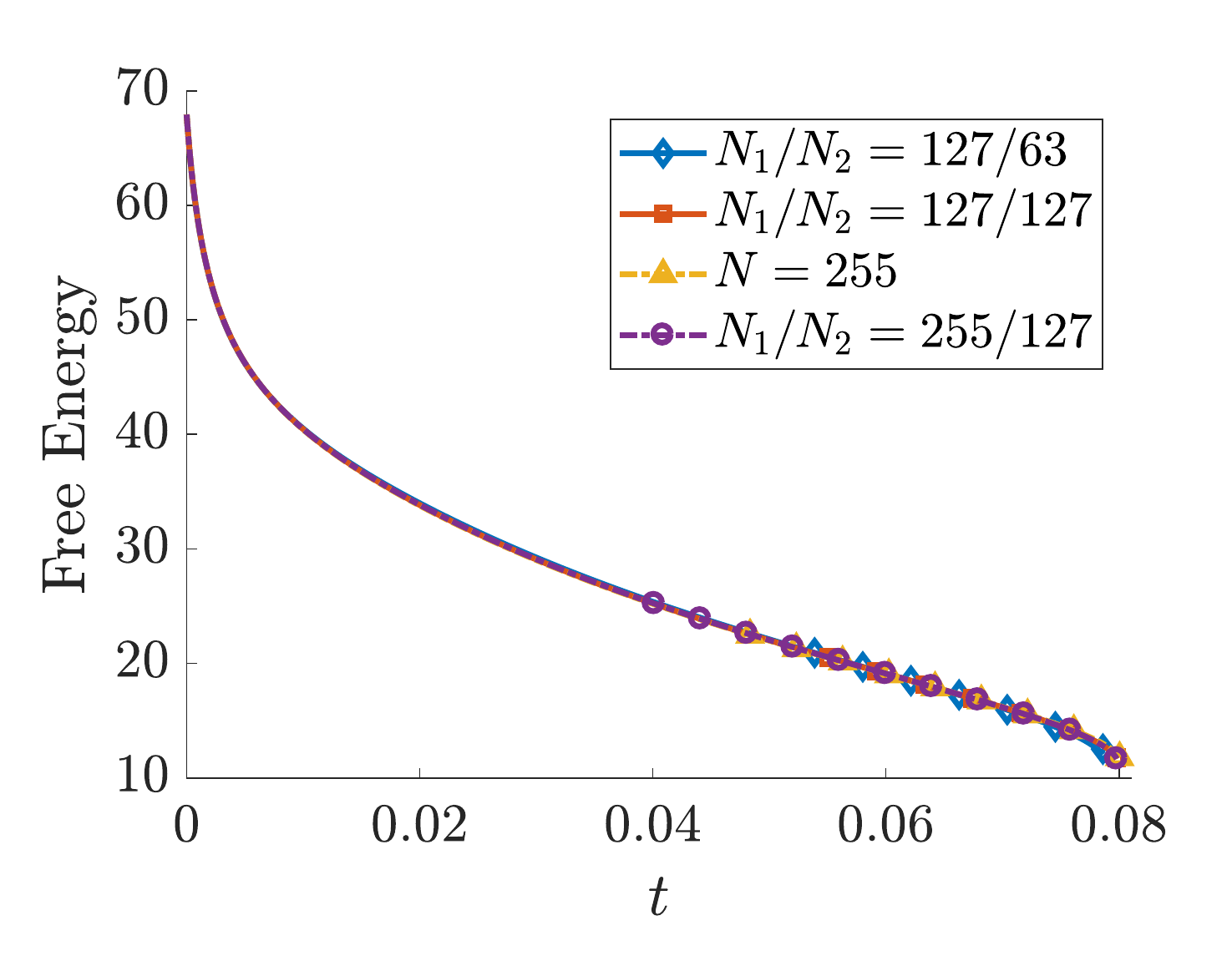}
        \caption{Free Energy}\label{fig:sub-BB-fe}
    \end{subfigure}
    \caption{Evolution of quantities for Test~\eqref{test:blowup_at_boundary}. 150, 300, 300 and 500 harmonics are used for each term in the 2-term extension operator along the boundary on mesh $127/63$, $127/127$, $255$ and $255/127$ respectively.}
    \label{fig:all-blowup-boundary}
\end{figure}

\begin{table}
\caption{$||\rho||_{\infty}$ at different times for Test~\eqref{test:blowup_at_boundary}. 500 harmonics are used on mesh 255/127 and 300 harmonics on mesh 127/127 for each term in the 2-term extension operator along the boundary.}
\label{table:rho-max-interval}
\centering
\footnotesize
\sisetup{
  output-exponent-marker = \text{ E},
  exponent-product={},
  retain-explicit-plus
}
\begin{tabular}{c c S[table-format=1.6e2] c S[table-format=1.6e2]}
\toprule
{} & \multicolumn{2}{c}{DD ($N_1/N_2=127/127$)} & \multicolumn{2}{c}{DD ($N_1/N_2=255/127$)}\\
\cmidrule(lr){2-3} \cmidrule(lr){4-5}
{} & {$t$} & {$||\rho||_{\infty}$} & {$t$} & {$||\rho||_{\infty}$}\\
\cmidrule(lr){2-3} \cmidrule(lr){4-5}
{$t_{\min}$} & 0.079744 & 3.825078E+04 & 0.079792 & 6.981745E+04\\
{$t_{\max}$} & 0.079804 & 4.714852E+04 & 0.079846 & 1.014028E+05\\
\bottomrule
\end{tabular}
\end{table}

The second moment in Fig.~\ref{fig:sub-BB-sm} is increasing as the cells are moving away from the origin. The free energy in Fig.~\ref{fig:sub-BB-fe} are decreasing, which also obeys the second law of thermodynamics. We should note that the free energies highly agree on difference meshes. This is due to the relatively small value of $c$ (in the magnitude of 10). Overall, we observe in Fig.~\ref{fig:all-blowup-boundary} that the SD approach, DD ($h_1/h_2\approx1/2$) and DD ($h_1/h_2\approx1/4$) approaches give similar results for Test~\eqref{test:blowup_at_boundary} as well. 

\begin{figure}
    \centering
    \begin{subfigure}[b]{0.4\textwidth}
        \includegraphics[width=\textwidth]{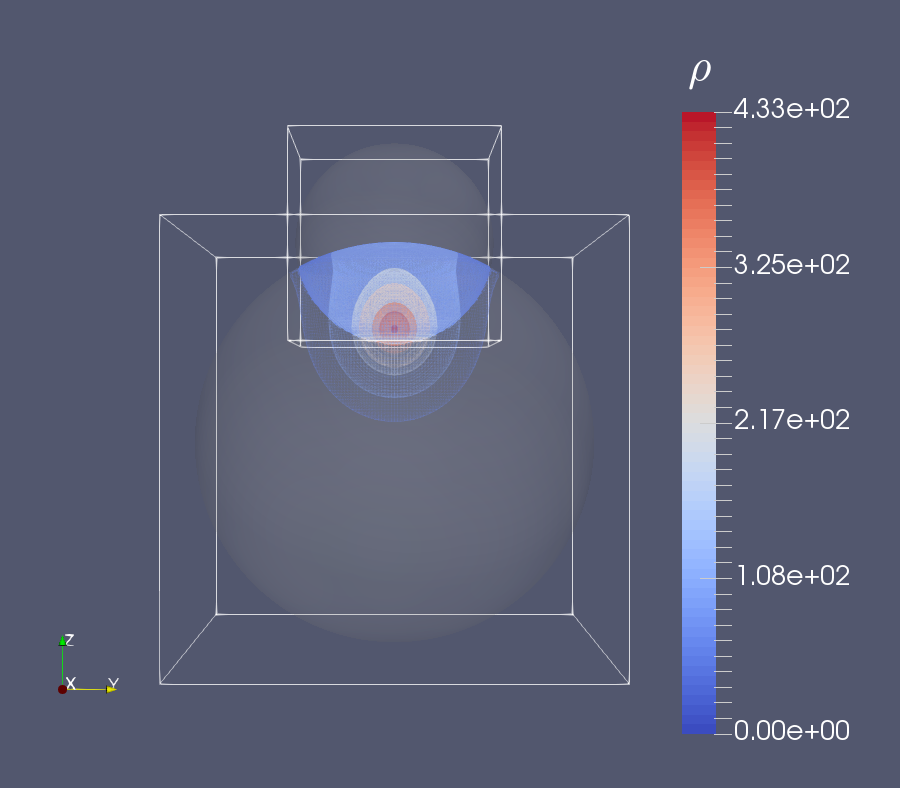}
        \caption{$\rho$ at $t=0.03$}
    \end{subfigure}
    \begin{subfigure}[b]{0.4\textwidth}
        \includegraphics[width=\textwidth]{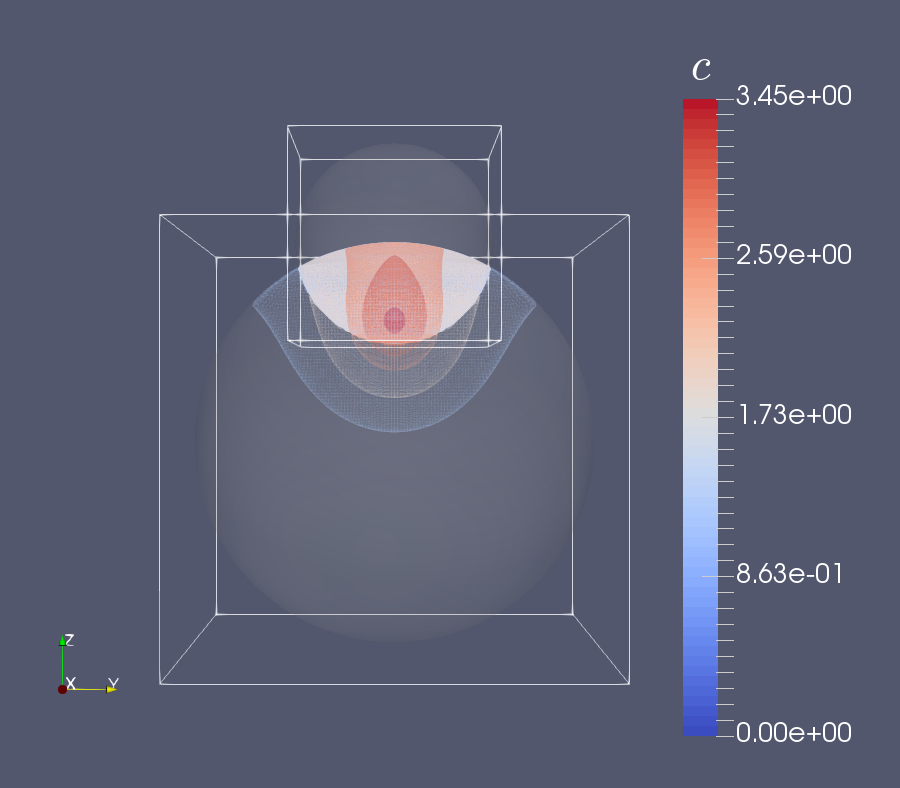}
        \caption{$c$ at $t=0.03$}
    \end{subfigure}
    \begin{subfigure}[b]{0.4\textwidth}
        \includegraphics[width=\textwidth]{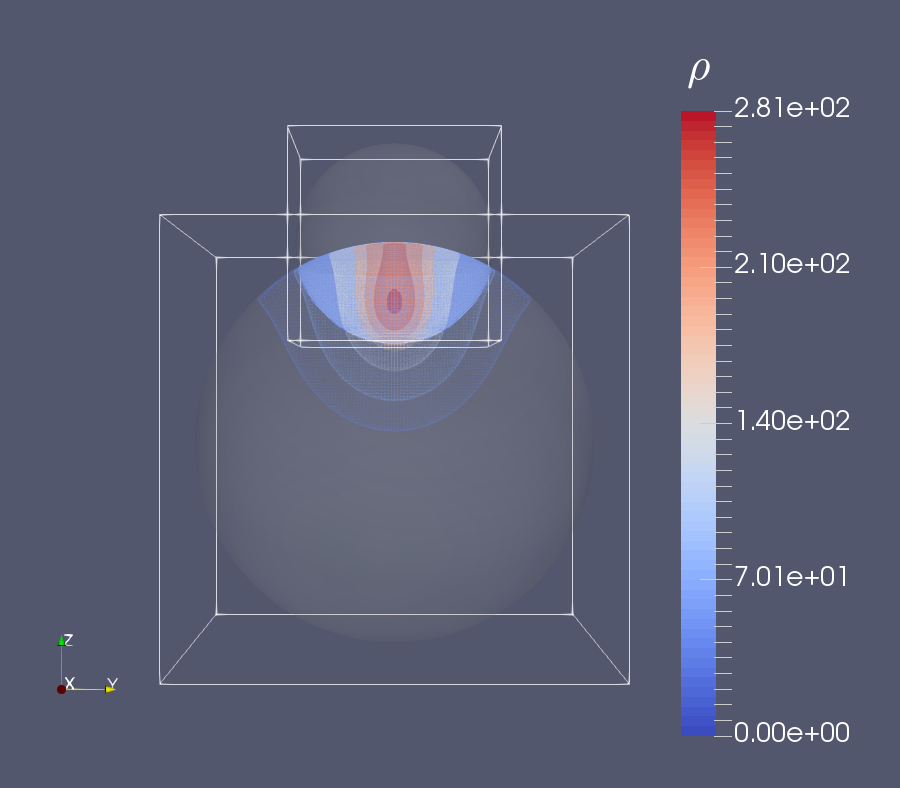}
        \caption{$\rho$ at $t=0.04$}
    \end{subfigure}
    \begin{subfigure}[b]{0.4\textwidth}
        \includegraphics[width=\textwidth]{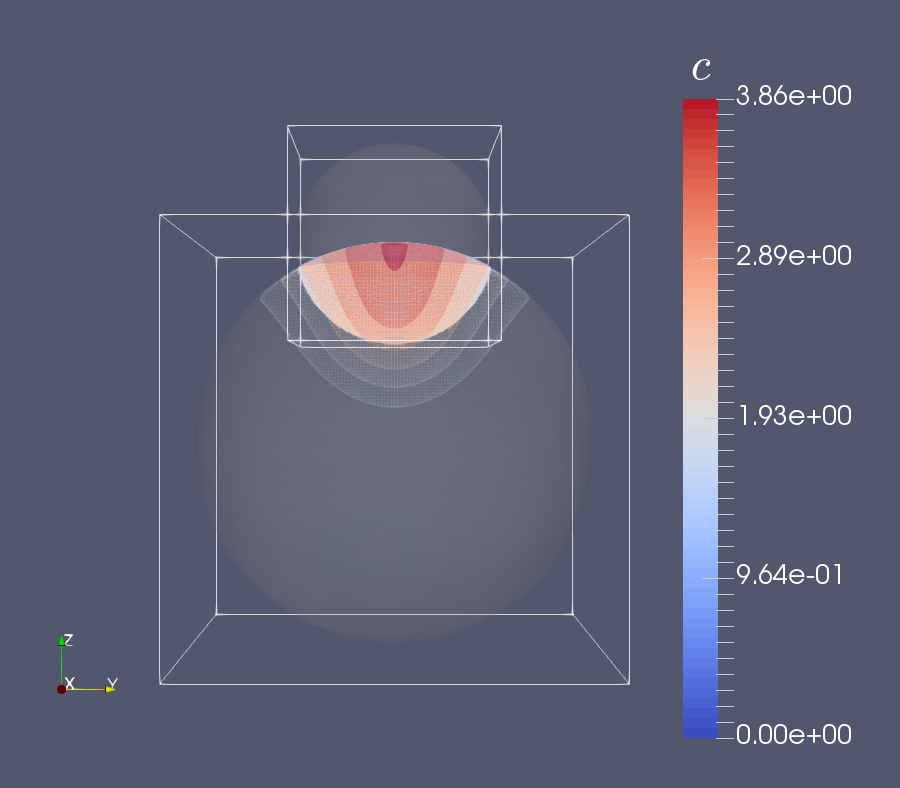}
        \caption{$c$ at $t=0.04$}
    \end{subfigure}
    \begin{subfigure}[b]{0.4\textwidth}
        \includegraphics[width=\textwidth]{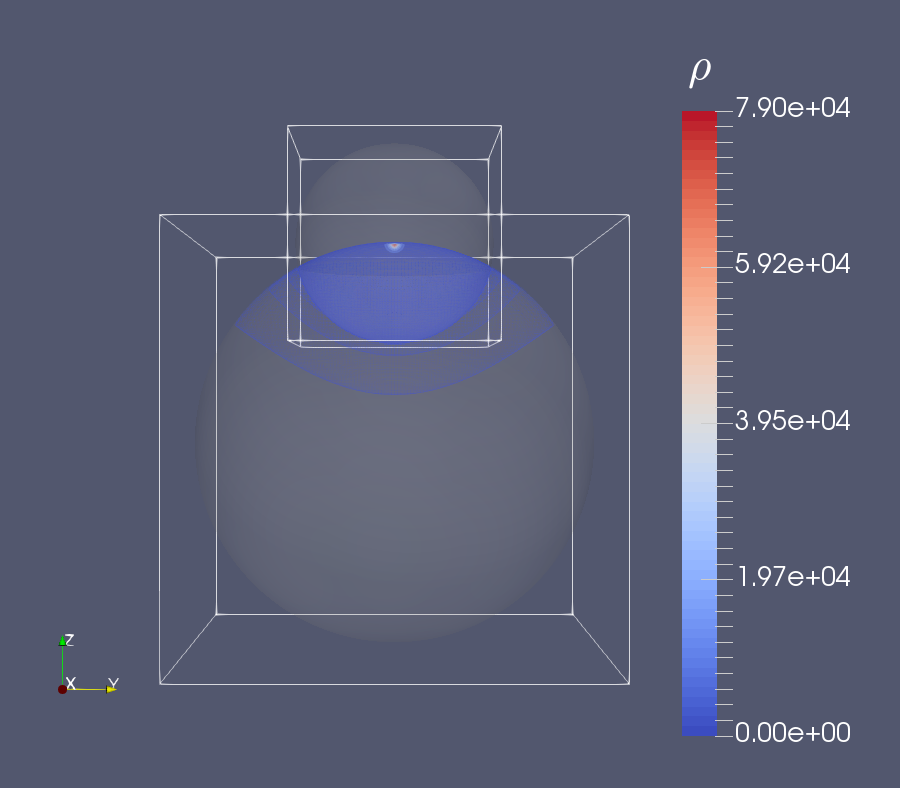}
        \caption{$\rho$ at $t\approx0.079812$}
    \end{subfigure}
    \begin{subfigure}[b]{0.4\textwidth}
        \includegraphics[width=\textwidth]{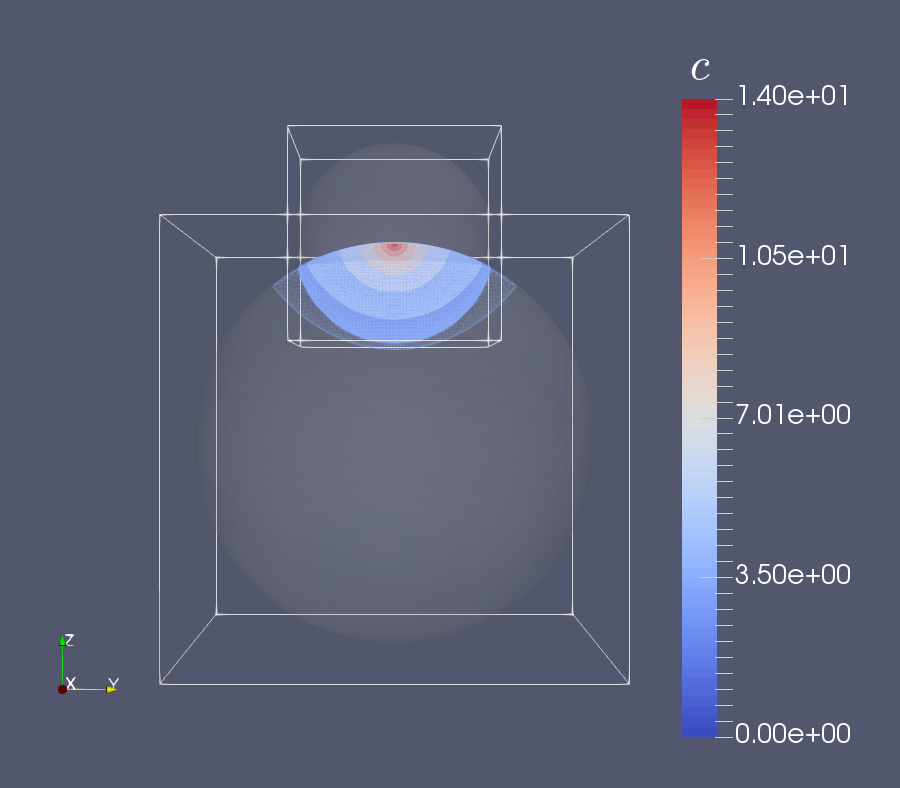}
        \caption{$c$ at $t\approx0.079812$}
    \end{subfigure}
    \caption{Isosurface plots of $\rho$ (left) and $c$ (right) for Test~\eqref{test:blowup_at_boundary} on mesh $N_1/N_2=255/127$ with grid sizes $h_1/h_2\approx1/4$ using domain decomposition approach at different times. 500 harmonics are used for each term of the 2-term extension operator along the boundary.}
    \label{fig:isosurface-blowup-boundary}
\end{figure}

\begin{figure}
    \centering
    \begin{subfigure}[b]{0.4\textwidth}
        \includegraphics[width=\textwidth]{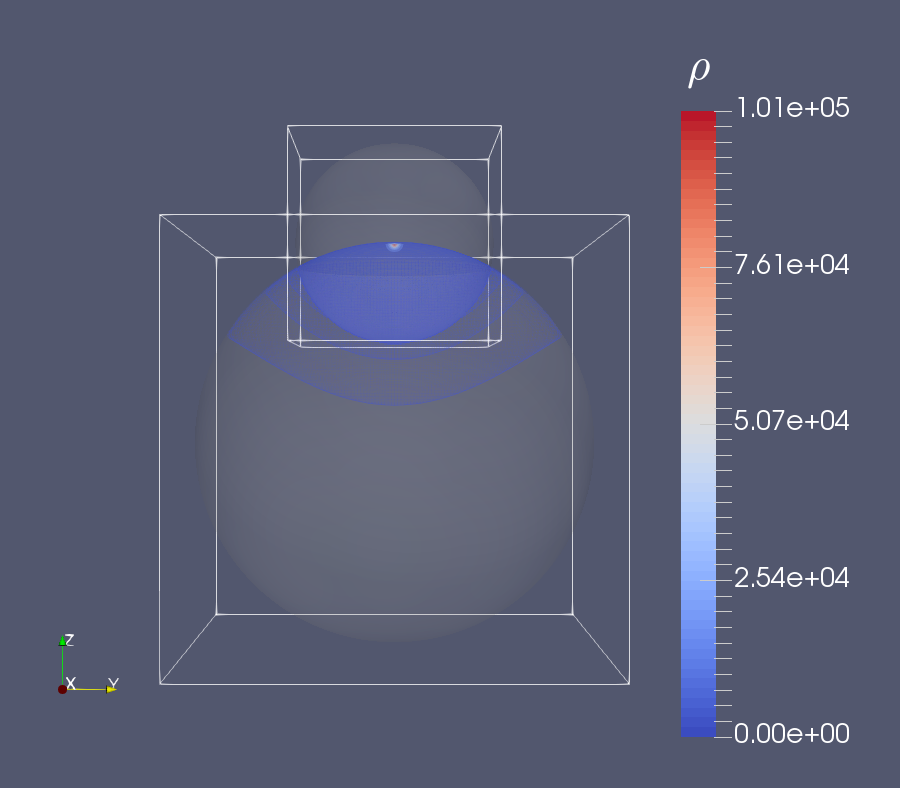}
        \caption{$\rho$ at $t\approx0.079846$}
    \end{subfigure}
    \begin{subfigure}[b]{0.4\textwidth}
        \includegraphics[width=\textwidth]{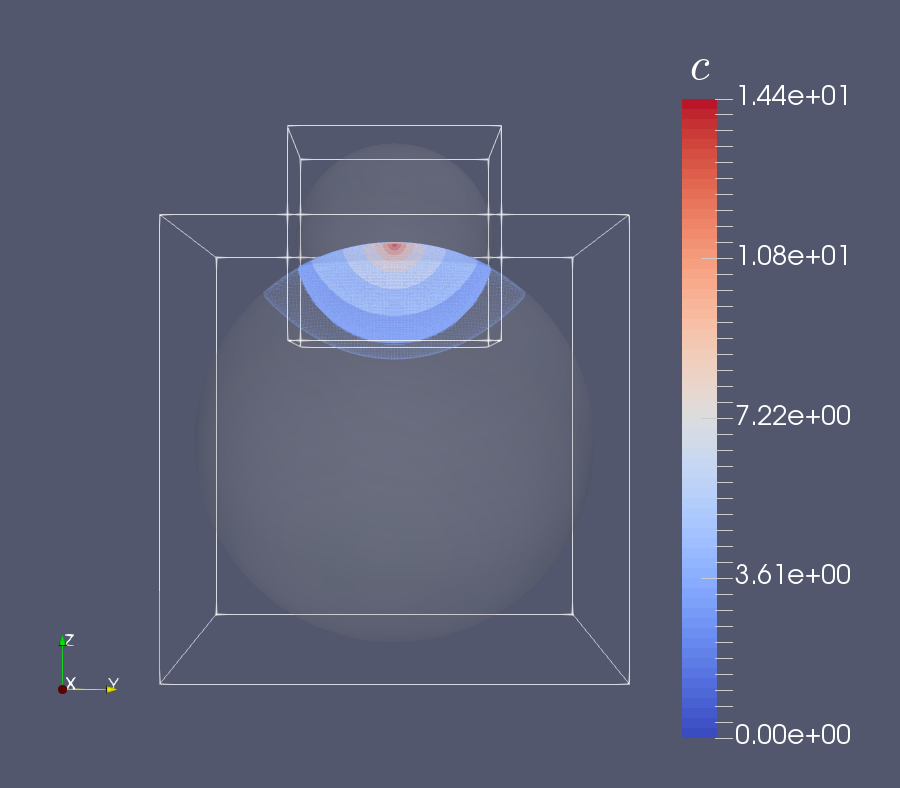}
        \caption{$c$ at $t\approx0.079846$}
    \end{subfigure}
    \caption{Isosurface plots of $\rho$ (left) and $c$ (right) for Test~\eqref{test:blowup_at_boundary} on mesh $N_1/N_2=255/127$ with grid sizes $h_1/h_2\approx1/4$ using domain decomposition approach at later time $t\approx0.079846$. 500 harmonics are used for each term of the 2-term extension operator along the boundary.}
    \label{fig:isosurface-blowup-boundary-later}
\end{figure}

Finally, in Fig.~\ref{fig:isosurface-blowup-boundary}, we present the 3D view of isosurface plots of $\rho$ and $c$ for Test~\eqref{test:blowup_at_boundary} at different times:
\begin{itemize} 
    \item At time $t=0.03$, we observe that both solutions $\rho$ and $c$ are in motion towards the north pole and continuity of solutions are also observed across the interface. 
    \item At time $t=0.04$, we observe that the peak value of $c$ already occurs at the north pole, while the peak value of $\rho$ is still in motion towards the north pole. This unveils the role that the chemoattractant plays and might help to explain why the blow-up solution of $\rho$ occurs at the north pole eventually. 
    \item At time $t\approx0.079812$, indeed we observe an almost singular solution at the north pole only. For both $\rho$ and $c$, the values are stratified away from the point of aggregation, i.e. the north pole. Moreover, no negative values are observed. We should also note that the simulations can proceed to a longer final time by choosing a different stopping criterion, but the simulation will not give a significant difference (see Fig.~\ref{fig:isosurface-blowup-boundary-later} for simulation results at $t\approx0.079846$).
\end{itemize}
The two cubic domains with white outlines in Fig.~\ref{fig:isosurface-blowup-boundary} and Fig.~\ref{fig:isosurface-blowup-boundary-later} are auxiliary domains for sub-domain $\Omega_1$ and sub-domain $\Omega_2$ respectively.



\section{Appendix}\label{sec:appendix}


\subsection{Proof to Theorem~\ref{thm:full_BEP}} \label{apd:proof-thm-full-BEP}

\begin{proof} For the reader's convenience, let us recall proof of Theorem~\ref{thm:full_BEP} in~\cite{Ryab,Epsh1}. 
First of all, let us assume that there is a density $u^{i+1}_{\gamma}$ that satisfies the BEP~\eqref{eqn:full_BEP}. Using $u^{i+1}_{\gamma}$, we can construct a superposition of the Particular Solution using (\ref{eqn:discrete_ap}--\ref{rhs:particular_solution}) and the Difference Potential using (\ref{eqn:discrete_ap}--\ref{eqn:discrete_ap_boundary}) and \eqref{rhs:difference_potentials}:
\begin{equation}\label{eqn:superposition_part_dp_solution}
u^{i+1}_{j,k,l}=P_{N^+\gamma}u^{i+1}_{\gamma}+G_{h,\Delta t}f^{i},\quad(x_j,y_k,z_l)\in M^0.
\end{equation} 
It can be seen that $u^{i+1}$ defined in \eqref{eqn:superposition_part_dp_solution} is a solution to the following discrete system on $N^0$:
\begin{equation}\label{eqn:superposition_part_dp}
\begin{aligned}
L_{h,\Delta t}u^{i+1}_{j,k,l}&=\left\{
\begin{array}{ll}
f^i_{j,k,l}, & \quad(x_j,y_k,z_l)\in M^+,\\
L_{h,\Delta t}[u^{i+1}_{\gamma}],&\quad (x_j,y_k,z_l)\in M^-,
\end{array}
\right.\\
u^{i+1}_{j,k,l}&=0,\quad(x_j,y_k,z_l)\in N^0\backslash M^0,
\end{aligned}
\end{equation}
which implies that $u^{i+1}_{j,k,l},\;(x_i,y_j,z_k)\in M^+$ in~\eqref{eqn:superposition_part_dp_solution} is some solution to~\eqref{eqn:shorter_discrete_chemotaxis} on $M^+$, and note that the trace of $u^{i+1}$ satisfies
\begin{equation}
Tr_{\gamma}u^{i+1}=Tr_{\gamma}\left(P_{N^+\gamma}u^{i+1}_{\gamma}+G_{h,\Delta t}f^{i}\right)=P_{\gamma}u^{i+1}_{\gamma}+G_{h,\Delta t}f^i_\gamma,
\end{equation}
from the definition of trace operator. Moreover, from our assumption on the density $u^{i+1}_{\gamma}$, we have $P_{\gamma}u^{i+1}_{\gamma}+G_{h,\Delta t}f^i_\gamma=u^{i+1}_{\gamma}$.
Thus, $Tr_{\gamma}u^{i+1}=u^{i+1}_{\gamma}$ and we can conclude that if some density $u^{i+1}_{\gamma}$ satisfies the BEP~\eqref{eqn:full_BEP}, it is a trace to some solution of system \eqref{eqn:shorter_discrete_chemotaxis}.

Now assume the density $u^{i+1}_{\gamma}$ is a trace to some solution of system \eqref{eqn:shorter_discrete_chemotaxis}: $u^{i+1}_{\gamma}=Tr_{\gamma}u^{i+1}$ where $u^{i+1}_{j,k,l}$, $(x_j,y_k,z_l)\in N^+$ is the unique solution to the difference equation~\eqref{eqn:shorter_discrete_chemotaxis} subject to certain boundary condition. Then by our assumption on $u^{i+1}$, $u^{i+1}$ with zero extension from $N^+$ to $N^0$ is also a solution to the system of difference equation~\eqref{eqn:superposition_part_dp} on $N^0$. Since $P_{N^+\gamma}u^{i+1}_{\gamma}+G_{h,\Delta t}f^i$ is also a solution the system of difference equation~\eqref{eqn:superposition_part_dp}, we conclude that:
\begin{equation}
u^{i+1}\equiv P_{N^+\gamma}u^{i+1}_{\gamma}+G_{h,\Delta t}f^i, \quad (x_j,y_k,z_l)\in N^+,
\end{equation}
by the uniqueness argument.
Finally, restriction of the $u^{i+1}$ from $N^+$ to the discrete grid boundary $\gamma$ would give us
\begin{equation}
Tr_{\gamma}u^{i+1}=Tr_\gamma\left(P_{N^+\gamma}u^{i+1}_{\gamma}+G_{h,\Delta t}f^i\right)=P_{\gamma}u^{i+1}_{\gamma}+G_{h,\Delta t}f^i_{\gamma}.
\end{equation}
and $u^{i+1}_{\gamma}=P_{\gamma}u^{i+1}_{\gamma}+G_{h,\Delta t}f^i_{\gamma}$, since $u^{i+1}_{\gamma}=Tr_{\gamma}u^{i+1}$ by our assumption. In other words, if $u^{i+1}_{\gamma}$ is a trace to some solution of system \eqref{eqn:shorter_discrete_chemotaxis}, it satisfies the BEP~\eqref{eqn:full_BEP}. 
\end{proof}


\subsection{Proof to Proposition~\ref{prop:rank}}\label{apd:proof-prop-rank}
\begin{proof}
The proof is similar to ideas in \cite{Ryab,Epsh1}. Without boundary conditions, the difference equations~\eqref{eqn:shorter_discrete_chemotaxis} would admit infinite number of solutions and so will the BEP \eqref{eqn:full_BEP}, due to Theorem~\ref{thm:full_BEP}. However, if we provide density values $u^{i+1}_{\gamma_{ex}}$ on $\gamma_{ex}$, the difference equation~\eqref{eqn:shorter_discrete_chemotaxis} will be well-posed and will admit a unique solution. Hence, the BEP~\eqref{eqn:full_BEP} will have a unique solution if $u^{i+1}_{\gamma_{ex}}$ is given. The solution to BEP~\eqref{eqn:full_BEP} is uniquely determined by densities on $\gamma_{ex}$, thus the solution $u^{i+1}_{\gamma}$ to BEP~\eqref{eqn:full_BEP} has dimension $|\gamma_{ex}|$, which is the cardinality of set $\gamma_{ex}$. As a result, the BEP \eqref{eqn:full_BEP} has rank $|\gamma|-|\gamma_{ex}|=|\gamma_{in}|$.
\end{proof}


\subsection{Proof to Theorem~\ref{thm:reduced_BEP}}\label{apd:proof-thm-reduced-BEP}
\begin{proof}
The proof is based on ideas in \cite{Ryab,Epsh1}. First, define the grid function:
\begin{align}\label{eqn:grid-function}
v^{i+1}:=P^{i+1}+G^{i+1}-u^{i+1}_{\gamma},\quad \mbox{on } N^0,
\end{align}
where $P^{i+1}$ is a solution to the AP~\eqref{eqn:discrete_ap}--\eqref{eqn:discrete_ap_boundary} on $N^0$ with right hand side \eqref{rhs:difference_potentials} using density $u^{i+1}_{\gamma}$, $G^{i+1}$ is a solution to the AP~\eqref{eqn:discrete_ap}--\eqref{eqn:discrete_ap_boundary} on $N^0$ with right hand side~\eqref{rhs:particular_solution}, and $u^{i+1}_\gamma$ is extended from $\gamma$ to $N^0$ by zero. By the construction of $v^{i+1}$, one can see that $v^{i+1}$ is a solution to the following difference equation:
\begin{equation}\label{eqn:full-laplace-equation}
\begin{aligned}
L_{h,\Delta t}[v^{i+1}]
&=\left\{
\begin{array}{ll}
f^i-L_{h,\Delta t}[u^{i+1}_{\gamma}],&\quad\mbox{on } M^+,\\
0, & \quad\mbox{on }M^-.
\end{array}
\right.
\end{aligned}
\end{equation}
Thus, we conclude that $v^{i+1}$ solves the following homogeneous difference equations on $M^-$:
\begin{align}\label{eqn:laplace-on-Mm}
L_{h,\Delta t}v^{i+1}=0,\quad \mbox{on } M^-.
\end{align}
In addition, due to the construction of $v^{i+1},P^{i+1}$ and $G^{i+1}$, the grid function $v^{i+1}$ satisfies the following boundary condition:
\begin{align}\label{eqn:bc}
v^{i+1}=0,\quad \mbox{on } N^0\backslash M^0.
\end{align}

Next, observe that the BEP~\eqref{eqn:full_BEP} and the reduced BEP~\eqref{eqn:reduced_BEP} can be reformulated using grid function $v^{i+1}$ in \eqref{eqn:grid-function} as follows:
\begin{align}\label{eqn:equiv-full-BEP}
v^{i+1}=0, \quad\mbox{on }\gamma, \quad(\mbox{BEP }\eqref{eqn:full_BEP}),
\end{align}
and
\begin{align}\label{eqn:equiv-reduced-BEP}
v^{i+1}=0, \quad\mbox{on }\gamma_{in},\quad (\mbox{BEP }\eqref{eqn:reduced_BEP}).
\end{align}
Hence, it is enough to show that \eqref{eqn:equiv-full-BEP} is equivalent to \eqref{eqn:equiv-reduced-BEP} to prove the equivalence between the BEP~\eqref{eqn:full_BEP} and the reduced BEP~\eqref{eqn:reduced_BEP}. First, note that if \eqref{eqn:equiv-full-BEP} is true, then \eqref{eqn:equiv-reduced-BEP} is obviously satisfied. 

Now, assume that \eqref{eqn:equiv-reduced-BEP} is true and let us show that \eqref{eqn:equiv-full-BEP} holds.  Let us consider problem~\eqref{eqn:laplace-on-Mm}: $L_{h,\Delta t}v^{i+1}=0$ on $M^-$, subject to boundary conditions \eqref{eqn:bc} and \eqref{eqn:equiv-reduced-BEP}, since the set $\gamma_{in}\cup(N^0\backslash M^0)$ is the boundary set for set $M^-$. Then we have the following discrete boundary value problem:
\begin{align}\label{eqn:discrete-bvp}
L_{h,\Delta t}v^{i+1}&=0,\quad\mbox{on }M^-,\\
v^{i+1}&=0,\quad \mbox{on }N^0\backslash M^0,\\
v^{i+1}&=0,\quad \mbox{on }\gamma_{in},
\end{align}
which admits a unique zero solution: $v^{i+1}=0$ on $M^-$. Since $\gamma_{ex}\subset M^-$, we conclude that $v^{i+1}=0$ on $\gamma_{ex}$, as well as on $\gamma\equiv\gamma_{ex}\cup\gamma_{in}$, which shows that \eqref{eqn:equiv-reduced-BEP} implies \eqref{eqn:equiv-full-BEP}. 

Thus, we showed that \eqref{eqn:equiv-full-BEP} is equivalent to \eqref{eqn:equiv-reduced-BEP}, and therefore, BEP~\eqref{eqn:full_BEP} is equivalent to the reduced BEP~\eqref{eqn:reduced_BEP}. Moreover, due to Proposition~\ref{prop:rank}, the reduced BEP~\eqref{eqn:reduced_BEP} consists of only linearly independent equations.

\end{proof}


\section*{Acknowledgements}
The research of Y. Epshteyn and Q. Xia was partially supported by National Science Foundation Grant \# DMS-1112984. Q. Xia would also like to thank N. Beebe, M. Berzins, M. Cuma and H. Sundar for helpful discussions on different aspects of high-performance computing related to this work. Q. Xia gratefully acknowledges the support of the University of Utah, Department of Mathematics.


\bibliographystyle{plainnat}      
\bibliography{RevisionChemotaxis3D.bib}   

\end{document}